%% file: thesis.tex
\documentclass[12pt,a4paper, twoside]{book}
\title{Bundle gerbes and the Weyl map}
\author{Kimberly Becker}
\date{\today}
\def\degree{Master of Philosophy}
\def\discipline{Pure Mathematics}%
 \usepackage[shortlabels]{enumitem}   
\usepackage{graphicx}

\makeatletter
\renewcommand{\@chapapp}{}
\newenvironment{chapquote}[2][2em]
{\setlength{\@tempdima}{#1}%
	\def\chapquote@author{#2}%
	\parshape 1 \@tempdima \dimexpr\textwidth-2\@tempdima\relax%
	\itshape}
{\par\normalfont\hfill--\ \chapquote@author\hspace*{\@tempdima}\par\bigskip}
\makeatother
\usepackage{mathscinet} 
\usepackage[thinlines]{easytable}
\usepackage{scalerel}
\usepackage{epigraph}
\usepackage{array}
\usepackage[]{graphics}

 \usepackage{amsmath,amsthm,amssymb,amsfonts}
 \usepackage{xfrac}

\usepackage[T1]{fontenc}
 
 \newcommand{\R}{\mathbb{R}}
 \newcommand{\N}{\mathbb{N}}
 \newcommand{\Z}{\mathbb{Z}}
 \newcommand{\C}{\mathbb{C}}
 
  \newcommand{\su}{{SU(n)}/{T}}
   \newcommand{\basicbg}{\left(P_b, Y,SU(n), \pi  \right)}
      \newcommand{\basicbgnopi}{\left(P_b, Y,SU(n)\right)}
         \newcommand{\cupbg}{\left(P_c, X,T\times \su\right)}

      \newcommand{\pullbackbasicbg}{ p^{-1}\left(P_b, Y,SU(n)  \right)}
         \newcommand{\Pbasic}{  P_{b}}
         
                  \newcommand{\PbasicT}{  P_{b, T}}
         
    \newcommand{\basicbgtorus}{\left( P_{b, T}, Y_T, T  \right)}
    
     \newcommand{\nablabasic}{\nabla_b}
      \newcommand{\pnablabasic}{\nabla_{p^*b}}
            \newcommand{\omegabasic}{\omega_b}
                
                  \newcommand{\fbasic}{f_b}
                     \newcommand{\pfbasic}{f_{p^*b}}
                        \newcommand{\nubasic}{\nu_b}
                              
        \newcommand{\LL}{\left(}
            \newcommand{\RR}{\right)}
\DeclareMathOperator*{\Motimes}{\text{\raisebox{0.25ex}{\scalebox{0.8}{$\bigotimes$}}}}
\DeclareMathOperator*{\Moplus}{\text{\raisebox{0.25ex}{\scalebox{0.8}{$\bigoplus$}}}}

\usepackage{accents}
\newlength{\dhatheight}
\newcommand{\doublehat}[1]{%
	\settoheight{\dhatheight}{\ensuremath{\hat{#1}}}%
	\addtolength{\dhatheight}{-0.35ex}%
	\hat{\vphantom{\rule{1pt}{\dhatheight}}%
		\smash{\hat{#1}}}}

\usepackage{xcolor}
\usepackage{graphicx,thesis,amsmath,amsthm} 
\usepackage[twoside,left=2cm,right=2cm,bindingoffset=1cm]{geometry}
\newlength{\mylen}
\setbox1=\hbox{$\bullet$}\setbox2=\hbox{\tiny$\bullet$}
\usepackage{tikz-cd}
   \usepackage{tikz}
   \usepackage{caption}
\usetikzlibrary{automata,positioning,arrows}

\newcommand\dra{\stackrel{\displaystyle 
		\rightarrow}{\rightarrow}}
	
\theoremstyle{plain}
\newtheorem{theorem}{Theorem}[chapter]
\newtheorem{corollary}[theorem]{Corollary}
\newtheorem{proposition}[theorem]{Proposition}
\newtheorem{lemma}[theorem]{Lemma}

\theoremstyle{definition}
\newtheorem{definition}[theorem]{Definition} 
\newtheorem{example}[theorem]{Example}

\newtheorem{remark}[theorem]{Remark}
\newtheorem{remarks}[theorem]{Remarks}
\numberwithin{equation}{chapter}

\newcommand*\bigcdot{\mathpalette\bigcdot@{.5}}
\def\smath#1{\text{\scalebox{.7}{$#1$}}}

\begin{document}
   \def\hmath#1{\text{\scalebox{1.6}{$#1$}}}
\def\lmath#1{\text{\scalebox{1.4}{$#1$}}}
\def\mmath#1{\text{\scalebox{1.2}{$#1$}}}
\def\smath#1{\text{\scalebox{.8}{$#1$}}}

\def\hfrac#1#2{\hmath{\frac{#1}{#2}}}
\def\lfrac#1#2{\lmath{\frac{#1}{#2}}}
\def\mfrac#1#2{\mmath{\frac{#1}{#2}}}
\def\sfrac#1#2{\smath{\frac{#1}{#2}}}
\def\pow{^\mmath}
\frontmatter

\maketitle

\tableofcontents

\include{front/statement}
\include{front/acknowledgements}
\include{front/dedication}
\include{front/abstract}

\mainmatter

\include{background/background}
\include{chapter2/CHP2}

\include{chapter3/CHP3}
\include{chapter3-5/CHP3-5TEST}

\include{chapter4/CHP4}

\appendix
\include{appendix1/appendix}

\backmatter
\addcontentsline{toc}{chapter}{Bibliography}

\nocite{*}


\end{document}

%% file: front/statement.tex
\chapter{Signed Statement}
{\parindent=0pt\parskip=4ex

%
%
%

I certify that this work contains no material which has been accepted for the award of any other degree or diploma in my name, in any
university or other tertiary institution and, to the best of my knowledge and belief, contains no material previously published or written
by another person, except where due reference has been made in the text. In addition, I certify that no part of this work will, in the
future, be used in a submission in my name, for any other degree or diploma in any university or other tertiary institution without the
prior approval of the University of Adelaide and where applicable, any partner institution responsible for the joint-award of this degree.\\ \text{}\\
I give permission for the digital version of my thesis to be made available on the web, via the University’s digital research repository,
the Library Search and also through web search engines, unless permission has been granted by the University to restrict access for a
period of time.\\ \text{}\\ 
I acknowledge the support I have received for my research through the provision of an Australian Government Research Training
Program Scholarship.

Signed: \dotfill\quad 
Date: \dotfill

}

%% file: front/acknowledgements.tex
\chapter{Acknowledgements}
\label{ch:acknowledgements}
First and foremost, thank you to my supervisors, Prof. Michael Murray and Dr~Daniel Stevenson. Thank you for enduring the endless drafting and reference writing that I subjected you to. Most of all, thank you for your expert supervision, without which this thesis would not have been possible. Thank you to my family for encouraging and supporting me through all my pursuits. To my mother and sister, I treasure your love and support, thank you. Thank you to the Hallams, for giving me a home away from home. Thank you to the staff and students at The University of Adelaide and Kathleen Lumley College for a wonderful past couple of years. Finally, to Michael Hallam, thank you for the boundless love and joy you bring me every day, even from other side of the world. Here's to our new chapter together. 

%% file: front/dedication.tex
\chapter{Dedication}
\label{ch:dedication}
\begin{center}
	\vspace*{5cm}
\large{\textit{To Michael Hallam}} \end{center} 

%% file: front/abstract.tex
\chapter{Abstract}
\label{ch:abstract}

This thesis reviews bundle gerbe theory and the well-known basic bundle gerbe over $SU(n)$. We introduce the cup product bundle gerbe, and show it is stably isomorphic to the pullback of the basic bundle gerbe by the Weyl map. This result enriches our understanding of the basic bundle gerbe, which has numerous applications in physics.

%% file: background/background.tex
\setcounter{chapter}{-1}
\chapter{Introduction}
\indent \indent A \textit{bundle gerbe} is a differential geometric object first developed by M. K. Murray in his 1996 paper \textit{Bundle Gerbes} \cite{1996bundlegerbes}. Based on J. Giraud's theory of \textit{gerbes,} bundle gerbes provide us with a geometric realisation of degree-three integer cohomology. This is similar to the geometric realisation of degree-two integer cohomology by line bundles. For this reason, bundle gerbes can be understood as `higher-dimensional' analogues of line bundles. Just as line bundles have become foundational to differential geometry, so too have bundle gerbes, with Murray's paper \cite{1996bundlegerbes} marking the beginning of what would become a deeply rich and fruitful field.\\
\indent The applications of bundle gerbes are far reaching, permeating many fields of mathematics and physics. For instance, they offer geometric interpretations of the `twists' in twisted $K$-theory \cite{ktheorybg}, the $B$-field in mathematical physics \cite{bfield}, and the Wess-Zumino-Witten models \cite{INTROTOBG}. Other applications of bundle gerbes to physical problems appear in the studies of quantum field theory \cite{carey2000, mickleson2006}, topological insulators \cite{bundlegerbesfortopologicalinsulators, insulator2}, Chern-Simons theory \cite{chernsimmons, mickler}, and mirror symmetry \cite{hitchenmirrorsymmetry}. They have also been exploited to good effect in the study of $D$-brane charges in string theory \cite{ktheorybg, dbranes1, dbranes2, dbranes3, dbranes4}. These constitute just a few of the many applications of bundle gerbes that we will explore throughout this thesis. \\
\indent Roughly, a bundle gerbe is a triple of manifolds $(P, Y, M)$ for $Y\to M$ a surjective submersion and $P$ a hermitian line bundle over $Y^{[2]}\subset Y^2$, the fibre product of $Y$ with itself. Additionally, we require an associative \textit{bundle gerbe multiplication} on fibres of $P$ $$m: P_{(y_1, y_2)}\otimes P_{(y_2, y_3)} \to P_{(y_1, y_3)}.$$ Just as with line bundles, there are notions of bundle gerbe constructions (pullback, product, dual, etc.), {connections, curvature}, and holonomy. Unlike line bundles, to define the \textit{three-curvature} and holonomy of a bundle gerbe connection, we require additional data besides the connection, namely a \textit{curving} or $B$-field in the physics literature. We call the pair consisting of a connection and curving \textit{connective data}. There are several types of bundle gerbe morphisms, including \textit{isomorphisms, stable isomorphisms,} and \textit{$D$-stable isomorphisms}. Crucially, stable isomorphism (respectively $D$-stable isomorphism) classes of bundle gerbes over $M$ are in bijective correspondence with $H^3(M, \Z)$ ($H^3(M, \Z(3)_D)$) via the \textit{Dixmier-Douady class} (\textit{Deligne class}). Moreover, the real part of the {Dixmier-Douady class} of a bundle gerbe can be represented by a multiple of its three-curvature, and the Deligne class of a bundle gerbe can be related to its three-curvature and holonomy. An in-depth treatment of each of these topics will be provided in Chapter \ref{ch:two}. \\
\indent In this work, we will focus on two specific bundle gerbes. The first of these bundle gerbes we call the \textit{cup product bundle gerbe}, defined over $T\times {SU(n)}/{T}$ for $T$ a maximal torus of $SU(n)$. This construction is based on the work of S. Johnson in \cite{stuartthesis}, allowing us to realise certain {wedge products} of differential forms as the {three-curvature} of a bundle gerbe. The second bundle gerbe of interest to us is the popularly studied \textit{basic bundle gerbe} introduced in \cite{higherbgs} and described in greater detail in \cite{dannythesis}. This is the classic example of a bundle gerbe defined over a compact, simple, simply connected Lie group $G$. It is particularly useful to physicists \cite{branesandgerbes}, allowing them to realise the Wess-Zumino-Witten term of a Sigma model as the holonomy of a geometric object. Our work will be concerned with the basic bundle gerbe {over $SU(n)$} constructed in \cite{53UNITARY}. The cup and basic bundle gerbes are commonly studied in isolation in the literature. We will show that there is a powerful result relating the cup and basic bundle gerbes, thereby deepening our understanding of these objects and enriching the field as a whole.   \\
\indent In this thesis, we aim to prove the following, for $T$ the subgroup of $SU(n)$ consisting of diagonal matrices which acts on $SU(n)$ by left multiplication. \\ \text{}\\
\textit{The pullback of the basic bundle gerbe over} $SU(n)$ \textit{by the Weyl map $p: T\times {SU(n)}/{T}\to SU(n),$ $(t, gT)\mapsto gtg^{-1},$ is $SU(n)$-stably isomorphic to the cup product bundle gerbe over }$T\times \su$.\\ \text{}\\ 
The crude interpretation of this result is that, under a transformation by the Weyl map, the basic and cup product bundle gerbes are `the same' as bundle gerbes enriched with an $SU(n)$-action. In addition to the above statement, we will also prove that the cup and basic bundle gerbes are \textit{not} $D$-stably isomorphic with respect to their naturally induced connective data (Theorem \ref{finalresult}). While this result is significant in of itself, it also opens up many avenues for further research. These include replacing the group $SU(n)$ by an arbitrary compact, simple, simply-connected Lie group $G$, considering this result in the holomorphic category, or proving an analogous result for gerbes in the sense of Giraud \cite{giraud}. These and other possible extensions will be discussed in Chapter \ref{ch:four}.\\
\indent There is a straightforward reason as to why we expect our main result to be true. As we will see, the three-curvature of the basic bundle gerbe is $2\pi i \nu$ for $\nu := -\frac{1}{24\pi^2}\text{tr}(g^{-1}dg)^3$ the basic $3$-form on $SU(n)$. Now, the pullback of the basic $3$-form by the Weyl map, $p^*\nu$, is an element of $H^3\left(T\times \su\right)$. By the Kunneth formula, and noting that the cohomology of $\su$ vanishes in odd degree \cite{odd}, we see that 
$$[p^*\nu] \in H^3(T)\oplus \left(H^2\left(\su\right)\otimes H^1(T)\right).$$ It follows from $T$ being abelian that the restriction of $\nu$ to $T$ vanishes. Therefore we have $$[p^*\nu] \in H^2\left(\su\right)\otimes H^1(T).$$ For this reason, we expect $p^*\nu$ (the three-curvature of the pullback of the basic bundle gerbe by the Weyl map) to equal a wedge product of forms (modulo exact forms). This wedge product is the three-curvature of a cup product bundle gerbe, which we suspect will be stably isomorphic to the pullback of the basic bundle gerbe. Our aim is to identify this cup product bundle gerbe, and construct explicitly the stable isomorphism. In doing so, we will understand the geometry underlying this topological result. \\
\indent As `higher versions' of line bundles, it is unsurprising that bundle gerbes require a foundational knowledge of {vector bundles}. This is the purpose of Chapter \ref{ch:intro}, \textit{Vector bundles}. Here, we shall review elementary vector bundle definitions, constructions, and connections (Sections \ref{a1}--\ref{section constructions}). The geometry of subbundles of the trivial bundle will be considered in Section \ref{a3} to be used later in this work. The classification of line bundles via the {first Chern class} will then be recalled (Section \ref{a4}), followed by a brief study of {line bundle holonomy} (Section \ref{a5}). We conclude this chapter by introducing equivariant vector bundles, and considering the categorical equivalence of homogeneous vector bundles and linear representations, which will aid later computations (Section \ref{homogeneous vector bundles}). Each of these topics were thoughtfully chosen to prepare the reader for analogous topics on bundle gerbes discussed in the next chapter. \\
\indent In Chapter \ref{ch:two}, \textit{Bundle gerbes}, we present the general theory of bundle gerbes in preparation for our later work. We are especially fond of this chapter, as it appears to fill a void in the student literature on bundle gerbes, particularly with its inclusion of bundle gerbe holonomy and Deligne cohomology. Moreover, it is the only introduction to bundle gerbes that we know of relying on the theory of line bundles, rather than principal $U(1)$-bundles. With the addition of historical remarks throughout and the refinement of several well-known definitions, we hope this chapter will also prove engaging to the expert reader. The format of this chapter proceeds similarly to Chapter \ref{ch:intro}. Namely, we discuss bundle gerbe preliminaries, definitions, constructions, and morphisms (Sections \ref{b1}--\ref{section morphisms}), before progressing to bundle gerbe connections and classification (Sections \ref{b3}--\ref{section dd class}). The more substantial topics of bundle gerbe holonomy and Deligne cohomology, together with a complete classification result for bundle gerbes (with and without connective data), will conclude this chapter (Sections \ref{section holonomy}--\ref{deligne cohomology}). \\ 
\indent Chapter \ref{cupproductchapter}, \textit{The cup product bundle gerbe}, marks the first of our more technically demanding chapters. Almost all of the definitions and computations in this chapter are our own, and provide excellent examples of the theory in Chapter \ref{ch:two}. We begin by considering the geometry of $\su$ (Section \ref{flagmanifold}). Our objective in this section is to realise the space $\su$ in terms of orthogonal projections, which will be useful to later calculations. In the remainder of this chapter, we define the cup product bundle gerbe, show it is $SU(n)$-equivariant, and calculate its connective data and three-curvature. We do so via a three-tiered approach. First, we introduce the \textit{general cup product bundle gerbe} (Section \ref{difference cup}), extending constructions of \cite{stuartthesis}, and calculate its connective data and three-curvature. After this, we consider general cup product bundle gerbes defined over $T\times \su$, called the \textit{$i$-th cup product bundle gerbes}, for $i=1, ..., n$. We show these bundles gerbes are $SU(n)$-equivariant, and apply the results from Section \ref{difference cup} to compute their connective data and three-curvature (Section \ref{ithbg}). Finally, the cup product bundle gerbe is defined as the product of the $i$-th cup product bundle gerbes, and is shown to be $SU(n)$-equivariant. Its connective data and three-curvature are presented using computations from the previous sections (Section \ref{cupprodbgsection}). \\ 
\indent In our penultimate chapter, \textit{The basic bundle gerbe and the Weyl map}, we define the pullback of the basic bundle gerbe by the Weyl map, show it is $SU(n)$-equivariant, and present its connective data and three-curvature. Unlike the previous chapter, we will additionally consider stable isomorphisms of the pullback of the basic bundle gerbe, which will be applied in later computations. The results in this chapter rely heavily on the work of Murray--Stevenson in \cite{53UNITARY}. We begin by introducing the Weyl map, due to H. Weyl \cite{weyl1, weyl2, weyl3} (Section \ref{c2}). The basic bundle gerbe over $SU(n)$ and its pullback by the Weyl map are then constructed, and their connective data presented, using \cite{53UNITARY} (Sections \ref{section: definition of basic bg}--\ref{c3}). We choose to present this connective data without proof, due to its reliance on delicate arguments in holomorphic functional calculus. To conclude this chapter, we present a series of original computations, culminating in a stable isomorphism of the pullback of the basic bundle gerbe with a product of general cup product bundle gerbes (Section \ref{subsection stable isos}). This result will be crucial to solving our central research problem in the final chapter.\\
%
%
\indent In Chapter \ref{ch:four}, \textit{The stable isomorphism}, we combine the work of Chapters \ref{ch:intro}--\ref{ch:three} to solve the central problem of this thesis. Using results on general cup product bundle gerbes from Section \ref{difference cup}, together with the stable isomorphism from Section \ref{subsection stable isos}, we show that the research problem simplifies significantly (Section \ref{d1}). We then compare the connective data of our bundle gerbes computed in Chapters \ref{cupproductchapter} and \ref{ch:three} with the connective data we would expect to obtain if these bundle gerbes were stably isomorphic (Section \ref{d2}). Doing so allows us to find an explicit stable isomorphism between the cup product bundle gerbe and the pullback of the basic bundle gerbe (Proposition \ref{prp}). Finally, by considering the holonomy of our bundle gerbes, we show that the cup product bundle gerbe and pullback of the basic bundle gerbe are not $D$-stably isomorphic (Section \ref{d3}), and summarise our findings in  Theorem \ref{finalresult}.

\chapter{Vector bundles\label{ch:intro}}
\begin{chapquote}{Paul A. M. Dirac}
	Theoretical physicists accept the need for mathematical beauty as an act of faith... For example, the main reason why the theory of relativity is so universally accepted is its mathematical beauty.
\end{chapquote}
%
%
\indent Vector bundles and connections are ubiquitous in mathematics and physics, forming the backbone of General Relativity and Gauge theory. Outside of these fields, the language of vector bundles appears in elliptic operator theory, PDE theory, and classical mechanics, to name a few. 
 The historical development of vector bundles and connections spans a century, catalysed by Einstein's special and general theories of relativity between 1905 and 1916. Consequently, the theory is deeply intertwined in physics. For example, in 1918, H. Weyl discovered that electromagnetism can be described as a connection on a real line bundle.
Not long after, Dirac considered this result in a complex setting, leading him to discover magnetic monopoles. 
  Similar applications continue well into modern times. We refer the interested reader to \cite{conceptualhistory, freemanthesis, VBapplications, historicalnotes} for a detailed history of vector bundles and connections.

%

%
%

We begin this chapter by defining vector bundles, connections, curvature, and constructions (Sections \ref{a1}--\ref{section constructions}). Although bundle gerbes rely chiefly on the theory of line bundles (vector bundles of rank $1$), the general theory is presented here as it requires little extra effort and is relevant to later constructions. In Section \ref{a3}, we turn our attention to subbundles of the trivial bundle. Here we will consider the naturally induced connection on a subbundle of a trivial bundle, and compute its two-curvature, which will appear in later computations. The more robust topics of line bundle classification and holonomy will then be recalled in Sections \ref{a4}--\ref{a5}. The final portion of this chapter, Section \ref{homogeneous vector bundles}, details equivariant and homogeneous vector bundles, concluding with the categorical equivalence of homogeneous vector bundles and linear representations. The relevance of this section will become clear in due course. For now, it suffices to say that the bundle gerbes of interest to this project can naturally be described in terms of equivariant line bundles, and that the categorical equivalence will simplify later calculations. \\
\indent The purpose of this chapter is two-fold. Our main objective is to present the prerequisites for the following chapters, and ease the reader into bundle gerbes, so to speak. We also hope that this chapter will make clear the parallels between line bundle and bundle gerbe theories later on. It is for this reason that we advise even the expert reader to skim read this chapter, which has intentionally been formatted in a similar way to Chapter \ref{ch:two}. We assume the reader has encountered vector bundles before, and present many standard results without proof.

\section{Definitions and morphisms}   \label{a1}
\indent \indent In this section, we will consider vector bundles and vector bundle morphisms.
 For the sake of brevity, we provide limited examples, and refer the reader to \cite{LEE, TU} for a holistic approach. We start with the definition of a vector bundle.

\begin{definition}[\cite{TU}]\label{vectorbundledefinition}
	Let $M$ be a smooth manifold. A \textit{smooth complex vector bundle on $M$} consists of a smooth manifold $E$ together with a smooth map $\pi: E\to M$ such that \begin{enumerate}[(1),font=\upshape]
	\item for all $p\in M$, the set $E_p:= \pi^{-1}(p)$ is a complex vector space;
	\item for all $p\in M$, there exists an open neighbourhood $U$ of $p$ and a diffeomorphism $$\phi_U: \pi^{-1}(U)\to U\times \C^n $$ for some $n$ that is \textit{fibre-preserving}, that is, 
	 \[\begin{tikzcd}[column sep=small]
	\pi^{-1}(U)\arrow[dr, "\pi", swap] \arrow{rr}{\phi_U} & & U\times \C^n\arrow[dl, "
	pr_1"] \\
	& U&
	\end{tikzcd}\] 
	commutes (for $pr_1$ projection onto the first factor); 
	\item for all $q\in U$, the map $\phi_U$ restricts to a linear isomorphism $E_q \to \{q\}\times \C^n$. 
	\end{enumerate}
\end{definition}
\noindent \noindent We call $E$ the \textit{total space}, $M$ the \textit{base space}, $\pi$ the \textit{projection}, and $E_p$ the \textit{fibre over} $p$. The \textit{rank} of $E_p$, rank$(E_p)$, is defined to be the dimension of the vector space $E_p$.
\begin{remark}
	By replacing $\C$ by $\R$ in Definition \ref{vectorbundledefinition} we obtain the definition of a \textit{smooth real vector bundle}. Throughout this text we will be working in the complex setting. 
\end{remark}
\begin{remark}
Consider a vector bundle $E\to M$. By Definition \ref{vectorbundledefinition} (2), the rank is constant over connected components of $M$. That is, if $U\subset M$ is connected, rank$(E_p)$ is constant for all $p\in U$. If rank$(E_p) =n \in \N$ for all $p\in M$, we call $E\to M$ a \textit{rank $n$} vector bundle. Several of the vector bundles introduced later will not have constant rank. 
\end{remark}
\begin{definition}
A smooth complex vector bundle of rank $1$ is called a \textit{smooth complex line bundle.}
\end{definition}

\begin{example}\textbf{(The trivial vector bundle)}\label{trivialvb} 
	Let $M$ be a smooth manifold and $n\in \N$. Then projection onto the first factor makes $M\times \C^n\to M$ into a smooth complex vector bundle of rank $n$, called the \textit{trivial complex vector bundle of rank $n$ on $M$}, or simply the \textit{trivial bundle}. 
\end{example}
We will proceed rather quickly with the following definitions and remarks, which we assume the reader to be familiar with. Here we will define sections, frames, hermitian vector bundles, bundle maps and subbundles. 
%
\begin{definition}
	A \textit{smooth section} of a vector bundle $\pi: E\to M$ is a smooth map $s: M\to E$ such that $\pi\circ s = id_M$. The space of all sections of $E\to M$ is denoted $\Gamma(M, E)$ or $\Gamma(E)$ when the base space is understood.   
\end{definition}
\begin{definition}
	Let $E\to M$ be a vector bundle of rank $r$. A \textit{frame} for $E$ over an open set $U\subset M$ is a collection of sections $e_1, ..., e_r:U\to E$ such that, for all $m\in U$, $\{e_1(m), ..., e_r(m)\}$ is a basis of $E_m$. 
\end{definition}

\indent As we have mentioned, our work will primarily be concerned with complex vector bundles. In this setting, a vector bundle can be endowed with some additional structure to become a \textit{hermitian vector bundle.}

\begin{definition} Let $E\to M$ be a smooth complex vector bundle. \begin{enumerate}[(1),font=\upshape]
\item A \textit{hermitian metric} on $E\to M$ is a smooth assignment of a hermitian inner product $\langle \cdot, \cdot\rangle_p$ on the vector space $E_p$ to all $p\in M$. 
\item The vector bundle $E\to M$ is a \textit{hermitian vector bundle} if it is equipped with a hermitian metric.
\item  A hermitian vector bundle of rank one is called a \textit{hermitian line bundle.}\end{enumerate}
\end{definition}

\begin{remark}
	The assignment $p\mapsto \langle \cdot, \cdot\rangle_p$ is said to be smooth if, given $s, t \in \Gamma(M, E)$, the complex-valued function $\langle s, t\rangle$ is smooth on $M$.
\end{remark}
\begin{remark}
	The standard inner product endows the trivial vector bundle of rank $n$ with a hermitian vector bundle structure. 
\end{remark}
\begin{definition}\label{unitsection}
Call a non-vanishing section $s$ of a hermitian vector bundle $E\to M$ a \textit{unit section} if $||s||=1$ with respect to the hermitian inner product.
\end{definition}
Throughout the remainder of this thesis, we shall refer to hermitian vector bundles simply as vector bundles, unless otherwise stated. We can now progress to the study of vector bundle morphisms, beginning with the definition of a \textit{smooth bundle map}. Naturally, a morphism of vector bundles consists of maps between the base and total spaces of the vector bundles, which preserve the vector bundle structure in the appropriate sense.
\begin{definition}[\cite{TU}]
	Let $\pi_1:E\to M$ and $\pi_2:F\to N$ be vector bundles. A \textit{smooth bundle map from $E$ to $F$} is a pair of smooth maps ($\hat{f}:E\to F, f: M\to N$) such that \begin{enumerate}[(1),font=\upshape]
		\item the diagram \[
		\begin{tikzcd}
		E\arrow{r}{\hat{f}} \arrow[swap]{d}{\pi_1} &F \arrow{d}{\pi_2} \\
	M \arrow[swap]{r}{{f}} & N
		\end{tikzcd}
		\] commutes;
		\item the map $\hat{f}$ restricts to a $\C$-linear map $\hat{f}_p: E_p\to F_{{f}(p)}$ for all $p\in M$. 
	\end{enumerate} 
	\end{definition} 
\begin{example}
	The \textit{identity map} of a vector bundle $E\to M$ is the pair of identity maps $(id_E: E\to E, id_M:M\to M)$, which is clearly a smooth bundle map. 
\end{example}
\begin{remarks} Let $E\to M$ and $F\to M$ be vector bundles.
	 \begin{enumerate}[(1),font=\upshape]
\item When the map $f:M\to N$ between base spaces is clear, we will denote a bundle map $(\hat{f}:E\to F, f:M\to N)$ by $\hat{f}:E\to F$. 

\item A smooth bundle map of the form $(\hat{f}: E\to F,  id_M)$ is called a \textit{(smooth) bundle map over }$M$. 

\item A bundle map $\hat{f}:E\to F$ over $M$ is called a \textit{bundle isomorphism} if $\hat{f}$ is a diffeomorphism. The vector bundles $E\to M$ and $F\to M$ are then said to be \textit{isomorphic over} $M$, or simply \textit{isomorphic}.\end{enumerate} 
 \end{remarks} 
As one might expect, a vector bundle is called \textit{trivial} if it is isomorphic to a trivial vector bundle of rank $n$ for some $n\in \N$. We conclude this section with one more foundational definition.
\begin{definition}[\cite{TU}]
	A \textit{smooth subbundle} of a smooth complex vector bundle $E\to M$ is a smooth complex vector bundle $F\to M$ such that \begin{enumerate}[(1),font=\upshape]
		\item the manifold $F$ is a submanifold of $E$;
		\item the inclusion map $F\to E$ is a smooth bundle map. 
	\end{enumerate}
\end{definition}
\begin{remark}[{\cite[p.$\,$13]{Hatchervectorbundles}}]\label{allvbaresubbundleoftrivbundle} 
	Every vector bundle is a subbundle of the trivial bundle of rank $n$ for some $n$. 
\end{remark}
\section{Connections and curvature}\label{a2}
\indent \indent To perform coordinate-invariant differential calculus on a vector bundle, we require the notion of a \textit{vector bundle connection}. Roughly speaking, a vector bundle connection generalises the directional derivative on $\R^n$ to an arbitrary manifold. It does this by providing us with means to `differentiate' a section $s$ of a vector bundle in the direction of a vector field $X$. Vector bundle connections will form the basis of \textit{bundle gerbe connections}, defined in the next chapter.\\
\indent In this section, we define vector bundle connections and curvature. Several standard results will then be presented without proof. We refer the reader to \cite{botttu, LEE, TU} for details. Familiarity with the tangent space to a manifold and smooth vector fields will be assumed. Let $C^{\infty}(M)$, $\mathfrak{X}(M)$ and $\Omega^k(M)$ denote the spaces of smooth maps $M\to \C$, smooth vector fields on $M$, and $k$-forms on $M$ respectively. Let us begin by defining a connection on a vector bundle.

\begin{definition} [\cite{TU}] \label{connectionvectorbundle}
A \textit{connection} on a complex vector bundle $E\to M$ is a map \begin{align*} \nabla: \mathfrak{X}(M)\times \Gamma(E)&\to \Gamma(E)\\	(X, s) &\mapsto \nabla_Xs \end{align*} such that, for all $X\in \mathfrak{X}(M)$ and $s\in \Gamma(E)$, \begin{enumerate}[(1),font=\upshape]
		\item $\nabla_Xs$ is $C^\infty(M)$-linear in $X$ and $\C$-linear in $s$;
		\item the \textit{Leibniz rule} holds, i.e. for $f\in C^{\infty}(M)$, $\nabla_X(fs) = df(X)s + f\nabla_Xs. $
	\end{enumerate}
\end{definition}
\begin{remark}
If $E\to M$ is a hermitian vector bundle, then we assume that the connection $\nabla$ is also hermitian. That is, for all $s, t \in \Gamma(M, E)$, $$d\langle s, t\rangle = \langle \nabla s, t\rangle + \langle s, \nabla t\rangle.$$
\end{remark}
\begin{remark}[{\cite[Theorem 10.6]{TU}}]
Using a partition of unity argument, it can be shown that every vector bundle admits a connection. 
\end{remark} 

To perform local computations, it is critical to know how a connection on a vector bundle $E\to M$ restricts to a connection on $E|_U \to U$ for $U\subseteq M$. This is given to us by the following remark.
\begin{remark}[{\cite[p.$\,$77]{TU}}]\label{newremark} Let $E\to M$ be a vector bundle and $(U, \psi)$ be a coordinate chart. Then there is a unique connection $\nabla^U: \mathfrak{X}(U)\times \Gamma(U, E)\to \Gamma(U, E)$ such that $\nabla^U_{X|_U}(s|_U) = \left(\nabla_Xs\right)\big\rvert_U$ for all $s\in \Gamma(M, E)$ and $X\in \mathfrak{X}(M)$. In particular, if $\partial_i = \frac{\partial}{\partial \psi^i}$ is the $i$-th standard basis vector of the tangent space at $U$, we can define $\nabla_i := \nabla^U_{\partial_i}$.
\end{remark}
%
We conclude this section by discussing the \textit{curvature tensor} $R$ of a connection, along with the (local) \textit{connection $1$-forms} and  (local) \textit{curvature $2$-forms}.
\begin{definition}
	Given a vector bundle $E\to M$ with connection $\nabla$, the \textit{curvature tensor} of $\nabla$ is the $\C$-multilinear map $R: \mathfrak{X}(M)\times \mathfrak{X}(M)\times \Gamma(E)\to \Gamma(E)$ defined by $$R(X, Y)s:= \nabla_X\nabla_Y s - \nabla_Y\nabla_X s - \nabla_{[X, Y]}s. $$
	We often denote $R$ by $F_{\nabla}$ and call it the \textit{curvature} of the connection $\nabla$. 
\end{definition}
\begin{definition}
Let $E\to M$ be a vector bundle with connection $\nabla$. Suppose $U\subseteq M$ is a trivialising open set and $e_1, ..., e_n$ is a frame on $U$. Let $i, j \in \{1, ..., n\}$ and $X\in \mathfrak{X}(U)$. The \textit{connection $1$-forms} of $\nabla$, denoted $A_j^i$, are defined implicitly by $$\nabla_X e_j = \sum_{i=1}^n A_j^i(X)e_i.$$ The matrix $A := (A_j^i)$ is called the \textit{connection matrix} of $\nabla$. 
\end{definition}
\begin{remark}\label{proposition two line bundle connections differ} 
	Any connection $\nabla$ on a line bundle $L\to M$ can be written locally as $\nabla = d + A$ for $d$ the ordinary differential and a local $1$-form $A\in \Omega^1(U)$ for some open $U\subset M$. To make this explicit, if $s:U\to L$ is a unit section and $\zeta = \zeta_Us$ for $\zeta_U:U\to \C$ is a section of $L$, then $\nabla \zeta = (d\zeta_U+ A\zeta_U)s$ where $\nabla^U s = As$.  It follows that the difference of any two connections on $L$ is a $1$-form on $M$. 
\end{remark}

\begin{proposition}[\cite{TU}]
	Let $E\to M$ be a vector bundle with connection $\nabla$ and $A_j^i$ be the connection $1$-forms on a trivialising open set $U\subseteq M$ relative to a frame $e_1, ..., e_n$. For each $j= 1, ..., n$ and $X, Y \in \mathfrak{X}(U)$, the curvature tensor can be expressed as \begin{align}\label{asdz} R(X, Y)e_j = \sum_{i=1}^n F_j^i(X, Y)e_i \end{align} for $2$-forms $F_j^i$ defined by \begin{align}\label{eqnforf} F_j^i := dA_j^i + \sum_{k=1}^n A_k^i\wedge A_j^k.\end{align}
\end{proposition}
\begin{definition} The $2$-forms $F_j^i$ in (\ref{eqnforf}) are called the \textit{curvature $2$-forms}, and equation (\ref{eqnforf}) is called the \textit{second structural equation}. \end{definition}
\begin{remark}
If $L\to M$ is a line bundle with connection $\nabla$, there is a single curvature $2$-form $F_\alpha$ over each trivialising open set $U_\alpha$ of $M$ (with respect to a fixed frame). By Remark \ref{proposition two line bundle connections differ} and equation (\ref{eqnforf}), $F_\alpha = dA_\alpha$ for $A_\alpha$ the connection $1$-form over $U_\alpha$. It is a standard fact that $dA_\alpha$ defines a global $2$-form $F$, and by equation (\ref{asdz}), $R = F = F_\nabla$. 
\end{remark}

\begin{remark} \label{connectiononeformsrelationship} 
	Let $L\to M$ be a line bundle and $\{U_\alpha\}_{\alpha \in I}$ be an open cover of $M$ with unit sections $s_\alpha: U_\alpha\to L$. Define transition functions $g_{\alpha \beta}: U_\alpha\cap U_\beta\to U(1)$ by $s_\alpha = g_{\alpha \beta}s_\beta$. Then it is not difficult to check that the connection $1$-forms $A_\alpha$ and $A_\beta$ over $U_\alpha$ and $U_\beta$ respectively satisfy $$A_\beta + g_{\alpha \beta}^{-1}dg_{\alpha \beta} = A_{\alpha} \text{ over } U_\alpha\cap U_\beta. $$
	%
\end{remark}
\section{Constructions} \label{section constructions} 
\indent \indent We will now consider {vector bundle constructions}.
In each of the following examples, we $(1)$ define (or recall) a vector bundle construction, $(2)$ define a natural connection on this vector bundle, and $(3)$ state the curvature of this connection in terms of the given data. These constructions will give rise to analogous constructions for bundle gerbes in the following chapter.  The reader is once again encouraged to seek further details in \cite{botttu, LEE, TU}, including proofs that these constructions yield vector bundles.  \\
\begin{example}\textbf{\textup{(The trivial vector bundle)}}
	\begin{enumerate}[(1)]
\item 	Recall the trivial vector bundle of rank $n$ over $M$, $M\times \C^n\to M$ (Example \ref{trivialvb}).
\item The derivative $d$ defines a connection on $M\times \C^n\to M$ called the \textit{flat connection}.
\item The curvature of this connection is zero.\\
	\end{enumerate}
%
\end{example}
\begin{example}\label{remarkonprojectionconnection}\textbf{(Subbundle of the trivial vector bundle)}
	\begin{enumerate}[(1)]
\item Consider a vector bundle $E\to M$ that is a subbundle of $\C^n\times M\to M$, the trivial bundle of rank $n$.
\item 	 The flat connection on the trivial bundle induces a connection $\nabla$ on $E\to M$ by $\nabla = P\circ d$ for $P$ orthogonal projection onto $E$.
\item The curvature of this connection is $F_{\nabla} = PdPdP$ (Proposition \ref{linebundleconnection}).\\
	\end{enumerate}
\end{example}
\begin{example}\textbf{\textup{(Dual)}} \label{example dual line bundle} 
	\begin{enumerate}[(1)]
		\item 
		The \textit{dual} of a vector bundle $E\to M$, denoted $E^*\to M$, is defined fibrewise by $(E^*)_x := (E_x)^*$, where $(E_x)^*$ denotes the dual of the vector space $E_x$. 
		\item Let $\nabla$ be a connection on $E\to M$. There is a connection $\nabla^*$ on $E^*\to M$ called the \textit{dual connection}, defined on sections $\alpha \in \Gamma(E)$ and $\zeta \in \Gamma(E^*)$ by $$(\nabla^{ *}\zeta)(\alpha) = d(\zeta(\alpha)) - \zeta(\nabla\alpha).$$
		\item If $E\to M$ is a line bundle, the curvature of $\nabla^*$ satisfies $F_{\nabla^*} = -F_{\nabla}$. 
	\end{enumerate} \newpage 
	%
\end{example}
	\begin{example}\textbf{\textup{(Product)}}\label{product line bundle}   \begin{enumerate}[(1)]
\item 	The \textit{tensor product} of vector bundles $E_1\to M$ and $E_2\to M$, denoted $E_1\otimes E_2\to M$, is a rank $mn$ vector bundle defined fibrewise by $(E_1\otimes E_2)_x := (E_1)_x\otimes (E_2)_x$.
\item  Let $\nabla_{E_1}$ and $\nabla_{E_2}$ be connections on $E_1\to M$ and $E_2\to M$, respectively. There is a connection $\nabla_{E_1\otimes E_2}$ on the product vector bundle $E_1\otimes E_2\to M$, called the \textit{product connection}, defined on sections $\alpha\in \Gamma(E_1)$ and $\beta\in  \Gamma(E_2)$ by $$\nabla_{E_1\otimes E_2}(\alpha\otimes \beta) = \nabla_{E_1}\alpha\otimes \beta + \alpha\otimes \nabla_{E_2} \beta.$$
\item The curvature of $\nabla_{E_1\otimes E_2}$ satisfies $F_{\nabla_{E_1\otimes E_2}} = F_{\nabla_{E_1}}\otimes id_{E_2} + id_{E_1}\otimes F_{\nabla_{E_2}}$. \\
		\end{enumerate}
	\end{example}
\begin{example}\label{powerconnection} \textbf{(Power)}
	\begin{enumerate}[(1)]
		\item Let $k\in \Z$. The \textit{$k$-th power} of a line bundle $L\to M$, denoted $L^k \to M$, is the line bundle defined by $L^k := \otimes_k L$ if $k>0$, $L^k := \otimes_{-k} L^*$ if $k<0$, and $\C$ if $k=0$.  
		\item Let $\nabla$ be a connection on $L\to M$. The product connection from Example \ref{product line bundle} induces a connection $\nabla^{k}$ on $L^k\to M$, defined on sections $\alpha_1, ..., \alpha_k \in \Gamma(L)$ by  $$\nabla^k(\alpha_1\otimes \cdots \otimes \alpha_k) = \sum_{i=1}^k \alpha_1 \otimes \cdots \otimes \nabla(\alpha_i)\otimes \cdots \otimes \alpha_k$$ if $k>0$, $d$ if $k=0$, and $(\nabla^{-k})^*$ if $k<0$.  
		\item The curvature of $\nabla^k$ satisfies $F_{\nabla^k} = k F_{\nabla}$.  \text{}\\
	\end{enumerate}
%
%
\end{example}
\begin{example}\textbf{\textup{(Determinant)}} \label{determinant line bundle } 
	\begin{enumerate}[(1)]
		\item Let $E\to M$ be a vector bundle of rank $n$. The \textit{determinant bundle},  denoted det$(E)\to M$, is the line bundle defined fibrewise by det$(E)_x := \bigwedge^n(E_x)$, the $n$-th exterior power of $E_x$.
		\item Let $\nabla$ be a connection on $E\to M$. The connection $\nabla$ induces a connection $\textup{det}(\nabla)$ on the determinant bundle $\textup{det}(E)\to M$, called the \textit{determinant connection}, defined on sections $\alpha_1, ..., \alpha_n \in \Gamma(E)$ by $$\textup{det}(\nabla)(\alpha_1\wedge \cdots \wedge \alpha_n) = \sum_{i=1}^n \alpha_1 \wedge \cdots \wedge \nabla \alpha_i \wedge \cdots \wedge \alpha_n.$$
		\item The curvature of det$(\nabla)$ satisfies $F_{\textup{det}(\nabla)} = \textup{tr}(F_{\nabla})$.\\
	\end{enumerate}
\begin{example} \label{line bundle pullback} \textbf{\textup{(Pullback)}}
	\begin{enumerate}[(1)]
		\item The \textit{pullback} of a vector bundle $E\to M$ by a smooth map $f:N\to M$, denoted $pr_1: f^{-1}(E)\to N$ (for $pr_1$ projection onto the first factor), is defined by $$f^{-1}(E) := \{(n, e) \in N\times E \ | \ f(n) = \pi(e) \}. $$ 
		\item Let $\nabla$ be a connection on $E\to M$. The connection $\nabla$ induces a connection $f^*\nabla$ on $f^{-1}(E)\to M$, called the \textit{pullback connection}, defined on $X\in \mathfrak{X}(M)$, $\alpha \in \Gamma(E)$ by $$\left(f^{*}\nabla\right)_X\left(f^{-1}\alpha \right) = f^{*}\left({\nabla}_{df(X)} \alpha\right).$$
		\item The curvature of $f^*\nabla$ satisfies $F_{f^*\nabla} = f^*F_{\nabla}$.\\  
	\end{enumerate}
%
%
 \end{example} 

%
	 \end{example}

\begin{example}\textbf{(Function powers (i))} \label{tereq2}\begin{enumerate}[(1)]
		\item Let $E\to M$ be a line bundle and $f:M\to \Z$ be a smooth map. On each connected component of $M$, $f$ is a constant integer value, so we can define $E^f\to M$ fibrewise by $E^f|_{m} := E_m^{f(m)}$ using Example \ref{product line bundle}.
		\item Let $\nabla$ be a connection on $E\to M$. Suppose $f(m) = c$ for all $m\in U_c\subset N$. Then $\nabla^c$ from Example \ref{powerconnection} defines a connection on $E^f|_{U_c}.$  Since we can do this on each connected component of $M$, this defines a global connection $\nabla^f$.
		\item The curvature of $\nabla^f$ satisfies $F_{\nabla^f} = fF_\nabla$. \\
	\end{enumerate}
\end{example}
\begin{example}\textbf{(Function powers (ii))} \label{tereq}\begin{enumerate}[(1)]
\item Let $E\to M$ be a line bundle and $f:N\to \Z$ be a smooth map. Define the vector bundle $E^f\to M\times N$ fibrewise by $E^f|_{(m, n)} := E_m^{f(n)}$ using Example \ref{product line bundle}.
\item Let $\nabla$ be a connection on $E\to M$. By almost identical arguments to Example \ref{tereq2}, $\nabla$ induces a connection $\nabla^f$ on $E^f\to M\times N$.
\item The curvature of $\nabla^f$ satisfies $F_{\nabla^f} = f\pi^*F_\nabla$ for $\pi:M\times N\to M$ projection. \\
\end{enumerate}
\end{example}
\begin{remark}\label{choccake}
The function power constructions in Examples \ref{tereq2} and \ref{tereq} can be related as follows. Consider a line bundle $E\to M$ and a smooth map $f: M\to \Z$. From Example \ref{tereq2}, we obtain a line bundle $E^f\to M$, which is the pullback of the line bundle $E^f\to M\times M$ from Example \ref{tereq} by the map $M\to M\times M$, $m\mapsto (m, m)$. \\
\end{remark}
\begin{example}\textbf{(Function powers (iii))}\label{function power} 
\begin{enumerate}[(1)]
	\item For $i=1,..., n$, let $E_i\to M$ be a line bundle, and $f_i:N\to \Z$ be a smooth map. By Examples \ref{product line bundle} and \ref{tereq}, $E_1^{f_1}\otimes \cdots \otimes E_n^{f_n}\to N\times M$ is a line bundle.
	\item Let $\nabla_{E_i}$ be a connection on $E_i\to M$ for each $i=1, ..., n$. By Examples \ref{product line bundle} and \ref{tereq}, there is a product connection $\nabla_{\Motimes E_i^{f_i}}$ on $E_1^{f_1}\otimes \cdots \otimes E_n^{f_n}\to M\times N$.
	\item The curvature of $\nabla_{\Motimes E_i^{f_i}}$ satisfies $F_{\nabla_{\Motimes E_i^{f_i}}} = \sum_{i=1}^n f_i\, \pi^{*}F_{\nabla_{E_i}}$ for $\pi:M\times N\to M$ projection.\\
\end{enumerate}
\end{example}
\begin{remark}
Unlike the other constructions, Examples \ref{tereq}--\ref{function power} are non-standard constructions that will be useful to our later work. 
\end{remark}
%

\section{Subbundles of the trivial bundle} \label{a3}
\indent \indent We diverge from our standard discussion of vector bundles for a moment, to consider the induced connection from Example \ref{remarkonprojectionconnection} in more detail. Our aim in this section is to prove that this connection has the claimed curvature. We do so for two reasons - the first being that this proof is more involved than what is required to show the other, more standard curvature results above. Secondly, this result will be crucial to our bundle gerbe calculations in Chapters \ref{cupproductchapter} and \ref{ch:three}, which ultimately enable us to prove the main result of this thesis. The reader may skip the proofs in this section, and return to them at a later time when their relevance is made clear.\\
\indent This section consists of three related results. We begin with a technical lemma, Lemma \ref{lemmalemon}. This will be used to prove Proposition \ref{linebundleconnection}, which states that the naturally induced connection on a subbundle of a trivial bundle has the curvature claimed in Example \ref{remarkonprojectionconnection}. Finally, Corollary \ref{linebundlecurvature} considers this result for line bundles. It is this corollary that will be used throughout our later work. 
\begin{lemma}\label{lemmalemon} 
Let $E\to M$ be a vector subbundle of the trivial bundle of rank $n$, and $P: \C^n\times M\to E$ be orthogonal projection. If $s$ is a local section of $E$, $$P\partial_i(P)\partial_j s = P\partial_i(P)\partial_j(P)s.$$
\end{lemma}
\begin{proof}
Applying $\partial_j$ to the equation $Ps = s$, we find $\partial_j s=\partial_j(P) s + P\partial_j s.$ Inputting this expression for $\partial_js$ into $P\partial_i(P)\partial_j s$, we obtain \begin{align}\label{eqn0} P\partial_i(P)\partial_js = P\partial_i(P)\partial_j(P)s + P\partial_i(P)P\partial_j s. \end{align} By applying $\partial_i$ twice to the equation $P^2 = P$, we find that $P\partial_i(P)P = 0$. Therefore the last term in equation (\ref{eqn0}) is zero and we obtain the required result. 
\end{proof}
With this lemma, we can prove the connection from Example \ref{remarkonprojectionconnection} has the claimed curvature.
	\begin{proposition} \label{linebundleconnection}
Let $E\to M$ be a vector subbundle of the trivial bundle of rank $n$. If $P: \C^n\times M\to E$ is orthogonal projection,  then the induced connection $\nabla = P\circ d$ on $E$ has curvature $$F_{\nabla}s = PdPdPs$$ for $s\in \Gamma(E)$.
	\end{proposition}
	\begin{proof}
	Note that, for a local section $s$ of $E$, $\nabla_i s = P\partial_i s$. Let $A_i$ and $F_i^j$ be the connection $1$-forms and curvature $2$-forms of $\nabla$ with respect to some trivialising open cover $\{U_i\}_{i\in I}$ of $M$. Since $[\partial_i, \partial_j] = 0$, we have \begin{align*}
		F_{\nabla}(\partial_i, \partial_j)s 
		&= \nabla_i\nabla_j s - \nabla_j \nabla_i s\\
		&= P\partial_i(P(\partial_j s)) - P\partial_j(P(\partial_i s))\\ 
		&= P\partial_i(P)\partial_j s -P\partial_j(P)\partial_i s \\
		&=  P\partial_i(P)\partial_j(P)s - P\partial_j(P)\partial_i(P)s \ \text{(Lemma \ref{lemmalemon})}\\
		& = (P[\partial_i P, \partial_j P])s.
		\end{align*}
Denote the matrix $(P[\partial_i P, \partial_j P])$ by $PdPdP$, so $F_\nabla s = PdPdPs$ for $s\in \Gamma(E)$. Note that if $Ps= 0$, $PdPdPs = 0$ also (this can be seen by differentiating $Ps=0$ twice). \end{proof}
We conclude this subsection by considering the previous proposition in the case when $E\to M$ is a line bundle. 
		\begin{corollary}\label{linebundlecurvature}
			Let $L\to M$ be a line bundle that is a subbundle of the trivial bundle of rank $n$. Let $P$ be orthogonal projection $\C^n\times M\to L$. Then the induced connection $\nabla = P\circ d$ on $L$ has curvature $F_\nabla = \textup{tr}(PdPdP). $
					\end{corollary}
		\begin{proof}	By Proposition \ref{linebundleconnection}, the induced connection $\nabla = P\circ d$ has curvature $F_\nabla$ satisfying $F_\nabla s = PdPdPs$	for all $s\in \Gamma(M, L)$. 
			Since $L\to M$ is a line bundle, $F_\nabla$ is a $2$-form and $PdPdP$ is a matrix of $2$-forms. If $\varphi \in \Gamma(M, L^\perp)$, then $P\varphi = 0$ so that $PdPdP\varphi = 0$. It follows that $F_\nabla = \text{tr}(PdPdP).$\end{proof}

\section{The first Chern class}\label{section the first chern class} \label{a4}
\indent \indent We will now consider the \textit{first Chern class}, $c_1(L)\in H^2(M, \Z)$, associated to a line bundle $L\to M$, which allows us to classify line bundles. In general, there are $n$ non-trivial Chern classes $c_1(E), ..., c_n(E)$ associated to a rank $n$ vector bundle $E\to M$. Since the general definition of the $n$-th Chern class would take up too much time to introduce, we focus only on the first Chern class here, which is all that is needed for our later work. A familiarity with sheaves and \v{C}ech cohomology is assumed, see \cite{botttu} for details. \\ 
\indent We construct the first Chern class, $c_1(L)$, of a line bundle $L\to M$ as follows.
 Denote the sheaves of smooth functions with values in $\R$, $\Z$, and $U(1)$ by $\underline{\R}$, $\underline{\Z}$, and $\underline{U(1)}$ respectively. Let $\{U_\alpha\}_{\alpha \in I}$ be a trivialising open cover of $M$ and $s_\alpha: U_\alpha \to L$ be unit sections. As before, the \textit{transition functions} $g_{\alpha \beta}: U_\alpha \cap U_\beta \to U(1)$ of $L$ are smooth maps defined by $s_\alpha = g_{\alpha \beta}\, s_\beta$. They satisfy the \v{C}ech cocycle condition $$g_{\beta\gamma} \, g_{\alpha \gamma}^{-1}\, g_{\alpha \beta}=1$$ on $U_\alpha \cap U_\beta \cap U_\gamma$, and hence determine a class in $H^1(M, \underline{U(1)})$. Now, the short exact sequence of sheaves $$0\xrightarrow{} \underline{\Z} \xrightarrow{} \underline{\R}\xrightarrow{} \underline{U(1)}\xrightarrow{} 0$$ induces a long exact sequence in cohomology \begin{align}\label{erw} \cdots \xrightarrow{} 0 = H^1(M, \underline{\R}) \xrightarrow{}H^1(M, \underline{U(1)})\xrightarrow{} H^2(M, \Z) \xrightarrow{} 0=H^2(M, \underline{\R}) \xrightarrow{} \cdots,\end{align} where $c_1$ denotes the connecting homomorphism. We remark that $H^k(M, \underline{\R}) = 0$ for all $k$ since $\underline{\R}$ is a \textit{fine sheaf}. With this information, we can define the first Chern class of $L$.
\begin{definition}
	Let $L\to M$ be a line bundle with transition functions $\{g_{\alpha \beta}\}_{\alpha, \beta \in I}$ over a trivialising open cover $\{U_\alpha\}_{\alpha \in I}$ of $M$. The \textit{first Chern class of $L\to M$}, denoted $c_1(L)\in H^2(M, \Z)$, is defined to be the image of $[g_{\alpha \beta}] \in H^1(M, \underline{U(1)})$ under the connecting homomorphism $H^1(M, \underline{U(1)}) \xrightarrow{\sim} H^2(M, \Z)$ from (\ref{erw}). 
\end{definition}
\begin{proposition}\label{ta}
	The first Chern class classifies line bundles up to isomorphism.
\end{proposition} Proposition \ref{ta} follows from exactness of (\ref{erw}) and the vanishing of $H^k(M, \underline{\R})$. The next standard result details relations obeyed by the first Chern class under the dual, product, and pullback operations.
\begin{proposition}\label{proposition chern class operations} Let $L\to M$ and $J\to M$ be line bundles and $f:N\to M$ be a smooth map. Then \begin{enumerate}[(1),font=\upshape]
		\item $c_1(L^*) = -c_1(L);$
		\item $c_1(L\otimes J) = c_1(L)+ c_1(J);$
		\item $c_1(f^{-1}(L)) = f^{*}c_1(L)$ where $f^*:H^2(M, \Z)\to H^2(N, \Z)$ is the induced map on cohomology.
	\end{enumerate}
\end{proposition}
We conclude this section by relating the curvature $F_\nabla$ of a connection $\nabla$ on a line bundle $L\to M$ to its first Chern class. Recall, by the fundamental theorem of finitely generated groups, that $H^2(M, \Z)$ can be decomposed into the sum of a free group, $\Z^b$, and torsion groups. It is a standard fact that $$H^2(M, \R)\cong H^2(M, \Z)\otimes_\Z \R\cong \R^b,$$ where $\otimes_\Z$ denotes the product of $\Z$-modules. Let $r: H^2(M, \Z)\to H^2(M, \R)$ be the map that includes the free part of $H^2(M, \Z)$ into $H^2(M, \R)$, and sends all torsion elements to zero. We call the image of the first Chern class $c_1(L)\in H^2(M, \R)$ in $H^2(M, \Z)$ under $r$ the \textit{real part} of $c_1(L)$, denoted $r(c_1(L))$. It is this quantity that we can relate to the curvature $F_\nabla$ of a connection $\nabla$, thereby concluding this section. 
\begin{proposition}\label{relationship between class and curvature} 
	Let $L\to M$ be a line bundle with connection $\nabla$ and curvature $F_\nabla$. Then the real part of $c_1(L)\in H^2(M, \R)$ is equal to the image of $[\tfrac{1}{2\pi i} F_\nabla]$ under the \v{C}ech-de Rham isomorphism $H^2_{\text{dR}}(M)\xrightarrow{\sim} H^2(M, \R)$.
\end{proposition}

\section{Line bundle holonomy} \label{a5}
\indent \indent In this section, we introduce \textit{line bundle holonomy}. In doing so, we hope to motivate the definition of \textit{bundle gerbe holonomy} to come later. Just as with the last section, we focus on line bundles, instead of vector bundles more generally, as to not over-complicate things.\\
\indent Let $L\to M$ be a line bundle with connection $\nabla$. Fix a curve $\gamma: [0,1]\to M$. The \textit{parallel transport} of $v\in L_{\gamma(0)}$ along $\gamma$ is the vector in $L_{\gamma(1)}$ which we obtain by translating $v$ along $\gamma$ without changing it (with respect to $\nabla$). That is, if $v(t) \in L_{\gamma(t)}$ is the translate of $v$ such that $$\nabla_{\dot{\gamma}}v \equiv 0$$ for $\dot{\gamma}$ the tangent vector field to $\gamma$, then  we call $P_\gamma(v):= v(1)$ the parallel transport of $v$. If $L$ is hermitian with hermitian connection $\nabla$, then $P_\gamma: L_{\gamma(0)}\to L_{\gamma(1)}$ preserves the inner product.
\begin{remark}
A connection on a vector bundle induces a notion of parallel transport, which in turn allows us to compare, or \textit{connect} vectors in different fibres of $L$ which would otherwise have no natural means of comparison. It is for this reason connections are named as such.
\end{remark}

It is a standard result that parallel transport defines an isomorphism $L_{\gamma(0)}\cong L_{\gamma(1)}$. In particular, if $\gamma$ is a closed loop, i.e. $\gamma(0)=\gamma(1)$, then it is possible that $v$ and $P_\gamma(v)$ are \textit{not} equal, but are complex (unit length) multiples of one another. A standard example of this occurs on the two-sphere, as described in \cite{TU}. With this, we can define line bundle holonomy, and relate the holonomy of a connection to its curvature.
\begin{definition}
Let $L\to M$ be a hermitian line bundle with hermitian connection $\nabla$. Let $\gamma:[0, 1]\to M$ be a loop. The \textit{holonomy of $\nabla$ along $\gamma$}, denoted $\text{hol}(\gamma, \nabla)\in U(1)$, is defined by $$P_\gamma(v) = \text{hol}(\gamma, \nabla) v$$ for $v\in L_{\gamma(0)}$.   
\end{definition}
\begin{remark}\label{hol of line bundle in terms of forms} 
If $L\to M$ is a line bundle with connection $\nabla$ and the loop $\gamma$ is the boundary of a compact submanifold $\Sigma$, then the holonomy is given by $$\text{hol}(\nabla, \gamma) = \text{exp}\left(-\int_\Sigma F_\nabla \right). $$
\end{remark}
\section{Homogeneous vector bundles}\label{homogeneous vector bundles} 
\indent \indent We conclude this chapter with the study of \textit{homogeneous vector bundles}. Here, we detail $G$-equivariant vector bundles for $G$ a Lie group, which allow us to consider group actions on vector bundles. Homogeneous vector bundles, which are certain kinds of  equivariant vector bundles, will then be introduced. After presenting some preliminary definitions, the rest of this section will be dedicated to proving the categorical equivalence of homogeneous vector bundles and linear representations. The bundle gerbes of interest to our research can naturally be defined in terms of $G$-equivariant vector bundles, and consequently, this section will be crucial to our later work. Let us begin by recalling the definition of a $G$-action, transitive $G$-action, and $G$-equivariant map. 
\begin{definition}
Let $G$ be a Lie group. Then a \textit{smooth left action by $G$} on a smooth manifold $M$ is a smooth map $G\times M\to M$, $(g, m)\mapsto g\cdot m$, such that \begin{enumerate}[(1)]
	\item $g\cdot(h\cdot m) = (gh)\cdot m$ for all $g, h \in G$ and $m\in M$;
	\item $e\cdot m = m$ for $e$ the identity of $G$ and $m\in M$.
\end{enumerate}
\end{definition}
\begin{definition}
A smooth left action by $G$ on $M$ is \textit{transitive} if, for all $x, y \in M$, there exists $g\in G$ such that $g\cdot x = y$.
\end{definition}
\begin{definition}
	Let $G$ be a Lie group and $M, N$ be smooth manifolds endowed with a smooth left action by $G$. Then a map $F:M\to N$ is $G$-\textit{equivariant} if, for all $g\in G$ and $m\in M$, $F(g\cdot m) = g\cdot F(m)$.
	\end{definition} 
\begin{definition}
	Let $M$ be a manifold endowed with a smooth left action by a Lie group $G$. The \textit{isotropy group} at a point $p\in M$ is the subgroup $$G_p := \{g\in G: g\cdot p = p\} < G.$$ 
\end{definition} \begin{definition}
Let $M$ be a smooth manifold and $G$ be a Lie group. If $G$ acts smoothly and transitively on $M$, we call $M$ a \textit{homogeneous} $G$-\textit{space}.
\end{definition}
It turns out that a homogeneous $G$-space can always be viewed as a quotient space of $G$ by some subgroup.
\begin{theorem}[{\cite[Theorem 7.19]{LEE}}]\label{theorem} 
Let $G$ be a Lie group and $M$ be a homogeneous $G$-space. Then for $p\in M$, the isotropy group $G_p$ is a closed subgroup of $G$, and the map $F: G/G_p \to M$, $gG_p\mapsto g\cdot p$, is an equivariant diffeomorphism.
\end{theorem}
We next define $G$-equivariant vector bundles. As one might expect, this is a vector bundle endowed with a $G$-action on both the total and base spaces, such that the action respects the vector bundle structure.
\begin{definition}\label{defn 1}
Let $G$ be a Lie group. Call a vector bundle $\pi: E\to X$ a $G$-\textit{equivariant vector bundle} if there are left actions of $G$ on $E$ and $X$ such that  \begin{enumerate}[(1),font=\upshape]
	\item $\pi$ is $G$-equivariant;
	\item the action of $G$ on $E$ restricts to a linear map on the fibres of $E$, i.e. the map $E_x\to E_{g\cdot x}, v \mapsto g\cdot v$ is linear. 
\end{enumerate}
\end{definition}

\begin{remark}
	Condition (2) of Definition \ref{defn 1} is equivalent to the action of $G$ being a linear isomorphism on fibres. The inverse of the action is defined by $v\mapsto g^{-1}\cdot v$, which is linear by definition of a group action and linearity of $v\mapsto g\cdot v$. 
\end{remark}
\begin{remark}
A $G$-equivariant bundle map of $G$-equivariant vector bundles is a bundle map $(\hat{f}, f)$ such that $\hat{f}$ and $f$ are $G$-equivariant with respect to the $G$-actions on the total and base spaces of the vector bundles. A $G$-equivariant bundle isomorphism can be defined similarly. 
\end{remark}
\begin{definition}
Let $G$ be a Lie group and $\pi:E\to X$ be a $G$-equivariant vector bundle. If $X$ is a homogeneous $G$-space, then $\pi:E\to X$ is a \textit{homogeneous vector bundle}.
\end{definition}
In the remainder of this section, we follow Segal \cite{SEGAL} to prove the categorical equivalence of homogeneous vector bundles over $G/H$ and linear representations of $H$ (Theorem \ref{thrmbundles}). We begin by constructing a vector bundle $G\times_HV\to G/H$, and show that any homogeneous vector bundle over $G/H$ is isomorphic to a vector bundle of this form. Let $G$ be a Lie group and $\pi:E\to X$ be a homogeneous vector bundle. By Theorem \ref{theorem}, $X\cong G/H$ where $H$ is the isotropy subgroup of some $p\in X$.  Let $V:= \pi^{-1}(p)$. Note that $H$ fixes $V$ by definition of $H$ and $G$-equivariance of $\pi$. Therefore there is a linear representation of $H$ defined by 
\begin{align*}
\chi: H&\to GL(V)\\
h &\mapsto (v\mapsto v\cdot h).
\end{align*}
  Let $G\times_HV$ denote the orbits of $G\times V$ under the $H$-action $$(g, v)\cdot h = (gh, \chi(h^{-1})v).$$ Define $\pi': G\times_HV\to G/H$ by $[(g, v)]\mapsto Hg$, which is well-defined and $G$-equivariant with respect to the $G$-action on $G\times_HV$ defined by $g'\cdot [(g, v)] = [(g'g, v)]$. 
  \begin{proposition}
Let $G\times_HV$ be the space defined above. Then $G\times_HV\to G/H$ is a $G$-equivariant vector bundle. 
  \end{proposition}
  \begin{proof}
We show this is a vector bundle and leave $G$-equivariance as an exercise. Local triviality follows by noting that $G\to G/H$ admits local sections. The fibre $(G\times_HV)_{Hg}$ can be equipped with a vector space structure as follows. Let $[p, v], [q, w] \in (G\times_HV)_{Hg}$. A simple computation shows that there exists $x\in H$ such that $[q, w] = [p, x].$ Let $k \in \C$. Then the scalar multiplication and addition operations defined by $$ [p, v] + k[p, x] := [p, vh + kx]$$ equip $(G\times_HV)_{Hg}$ with the structure of a vector space.\end{proof} The construction of $G\times_HV\to G/H$ from the linear representation $\chi: H\to GL(V)$ is an instance of the \textit{associated fibre bundle construction} (for more on this see \cite{SEGAL}). We can now prove the following key result.
\begin{proposition}
There is a $G$-equivariant isomorphism from $G\times_HV\to G/H$ to $E\to G/H$.
\end{proposition}
\begin{proof} Define $\alpha: G\times_HV \to E$ by 
$\alpha: [(g, v)]\mapsto v\cdot g.$ Then $\alpha$ is well-defined, $G$-equivariant, and smooth. It can be verified that $\alpha$ restricts to a linear map on fibres and makes 

\[
\begin{tikzcd}
G\times_HV\arrow{r}{\alpha} \arrow[swap]{d}{\pi'} & E \arrow{d}{\pi} \\
G/H \arrow[swap]{r}{id} & G/H 
\end{tikzcd}
\]
commute, hence is a $G$-equivariant bundle map. It follows that $\alpha$ is a $G$-equivariant bundle isomorphism by defining an inverse $\beta: E\to G\times_HV$ by $$\beta: v\mapsto [(g, vg^{-1})]$$ where $\pi(v) = gH$.  \end{proof} We conclude that any homogeneous vector bundle over $G/H$ is isomorphic to a vector bundle of the form $G\times_H V$ defined for some representation of $H$ on a finite-dimensional vector space $V$. 
Conversely, given a linear representation $H\to GL(V)$ for $V$ a finite-dimensional vector space, there is a standard construction of a homogeneous vector bundle associated to the linear representation. As a result, any homogeneous vector bundle $E\to G/H$ is described completely by the representation of $H$ on the fibre of $E$ over the identity coset $H$. By considering the space of all homogeneous vector bundles over $G/H$ and the space of linear representations of $H$ as categories, the preceding discussion can be summarised as follows. 
\begin{theorem}\label{thrmbundles} 
The category of homogeneous vector bundles over $G/H$ is equivalent to the category of linear representations of $H$.
\end{theorem}

%% file: chapter2/CHP2.tex
\chapter{Bundle gerbes\label{ch:two}}
\begin{chapquote}{M. K. Murray}
I recall that [Brylinski's] book took some months to makes its away across the sea to Australia during which time I pondered the advertising material I had which said that gerbes were fibrations of groupoids. Trying to interpret this lead to a paper on bundle gerbes...
\end{chapquote}
In 1971,  J. Giraud developed a general theory of  \textit{gerbes} \cite{giraud}, defined as certain kinds of sheaves of groupoids on $M$. 
 The study of gerbes attracted the attention of the international mathematical community upon publication of J.-L. Brylinski's 1993 book \cite{brylinski}. Drawing on Giraud's work, Brylinski detailed a bijection between cohomology classes in $H^3(M, \Z)$ and equivalence classes of abelian gerbes. Brylinski's ambition was best summarised in his own words: \\ \text{}\\
\textit{It would of course be highly desirable to have a theory for integer cohomology of any degree, which is as clear, geometric, and as explicit as the Weil-Kostant theory for line bundles. This book intends to be a first step in that direction.} -- J.-L. Brylinski.\\ \text{}\\   
\indent Brylinski refers here to the classical Weil-Kostant theory of  Proposition \ref{relationship between class and curvature}. In 1996, Brylinski's wish would be fulfilled in degree three. Motivated only by a summary of \cite{brylinski}, M. K. Murray developed the highly geometric, sheaf-free theory of \textit{bundle gerbes} \cite{1996bundlegerbes}. Murray showed that any bundle gerbe gives rise to a gerbe (though the converse need not hold), and that many interesting examples of gerbes could be realised as bundle gerbes. This paper also provided an explicit description of the cohomology class in $H^3(M, \Z)$ (called the \textit{Dixmier-Douady class}) associated to a bundle gerbe. In 2000, once a suitable notion of equivalence was developed (namely \textit{stable isomorphisms}), it was shown that equivalence classes of bundle gerbes are in bijective correspondence with $H^3(M, \Z)$. Remarkably, almost all of the fundamental aspects of bundle gerbe theory were described in Murray's original publication, making it a landmark paper that remains foundational to this field.\\
\indent Roughly, a bundle gerbe is a triple of manifolds $(P, Y, M)$ such that $P$ is a line bundle over a submanifold $Y^{[2]}$ of $Y^{2}$ and $Y\to M$ is a surjective submersion. Moreover, we require $P\to Y^{[2]}$ to be equipped with a fibrewise multiplication called the \textit{bundle gerbe multiplication.} Bundle gerbes are, in many ways, akin to line bundles. Each of the topics pertaining to line bundles in the previous chapter, including constructions, the first Chern class, triviality, connections, and holonomy, have higher-dimensional analogues for bundle gerbes. Notably, the curvature of a bundle gerbe is a \textit{$3$-form}, and the holonomy of a bundle gerbe is defined by exponentiating the integral of the three-curvature over a \textit{surface}. We will study each of these topics in detail. \\
\indent Bundle gerbes and gerbes have a variety of applications in differential geometry, topology and mathematical physics. Gerbes were used by Hitchin to study mirror symmetry \cite{hitchenmirrorsymmetry} and generalised geometry \cite{hitchengeneralisedgeometry}. Applications of gerbes to quantum field theory were studied in \cite{carey2000, mickleson2006}. Brylinski worked extensively with gerbes throughout the later part of his academic career, even using holomorphic gerbes to understand the so-called {B}e\u{\i}linson regulator map \cite{brylinskihologerbes}. The benefit of working with bundle gerbes, as opposed to gerbes, is due to their explicit geometric structure. For example, bundle gerbes provide a geometric realisation of degree-three cohomology, the twist in twisted K-theory \cite{ktheorybg} and the $B$-field in mathematical physics \cite{bfield}. They have also become popular in the study of $D$-brane charges in string theory \cite{ktheorybg, dbranes1, dbranes2, dbranes3, dbranes4}.  More recently, bundle gerbes have been used to study topological insulators \cite{insulator2, bundlegerbesfortopologicalinsulators}. Even a certain well-known lifting problem of principal bundles can be phrased in terms of bundle gerbes. Given a central extension of Lie groups $$0\to U(1)\to \hat{{G}}\to {G} \to 0,$$ one may ask, when does a principal ${G}$ bundle $Y\to M$ lift to a principal $\hat{{G}}$ bundle? To answer this, one can use the \textit{lifting bundle gerbe}, defined in \cite{1996bundlegerbes}, whose Dixmier-Douady class is precisely the topological obstruction to the existence of a lift. In particular, this problem can be used to understand \textit{string structures} when the central extension is the \textit{Kac-Moody central extension}, see \cite{INTROTOBG}. Bundle gerbes also arise naturally in the study of Wess-Zumino-Witten models \cite{INTROTOBG, konrad}, as we will see later, and in the study of Chern-Simons theory \cite{chernsimmons, mickler}. These are just a few of the many examples in mathematics and physics that utilise gerbes and bundle gerbes.  \\
\indent It should be noted that there has been some progress made to Brylinski's original question of defining a geometric representative of degree-$p$ cohomology. These are called bundle $p$-gerbes which (for historical reasons) define degree-$(p+2)$ cohomology. D. Stevenson motivates the problem of defining bundle $p$-gerbes by noting that it would be helpful if physicists could understand $p$-forms as the curvature of some geometric object \cite{dannybundle2gerbes}. In this work,  Stevenson defines bundle 2-gerbes as certain kinds of bundles of \textit{bigroupoids}, based on the original definition presented in \cite{higherbgs}. Although there is a notion of a bundle $p$-gerbe, there is not, as of yet, a straightforward geometric interpretation of this object when $p>2$.

We begin this chapter with some preliminaries needed to define bundle gerbes. Our discussion thereafter progresses similarly to the previous chapter. After defining bundle gerbes, we describe several key bundle gerbe constructions. Morphisms of bundle gerbes will then be detailed, along with the theory of connections and curvature for bundle gerbes.  In the latter half of this chapter, we discuss the Dixmier-Douady class, holonomy and Deligne cohomology class of a bundle gerbe. Assume throughout that all spaces are smooth manifolds and all maps (and in particular sections) are smooth. 
\section{Preliminaries}\label{b1} \indent \indent In this section, we present the preliminary definitions needed to define bundle gerbes, including \textit{surjective submersions} and \textit{fibre products.} The reader familiar with bundle gerbes may wish to skip this section. We assume all maps to be smooth throughout. Let us begin with surjective submersions. 
\begin{definition}
	A map $\pi:X\to N$ is a \textit{submersion} if $d\pi|_x$ is surjective for all $x\in X$. 
\end{definition}

\begin{definition}\label{morphismsurjsub}
	Let $\pi_X: X\to N$ and $\pi_Y: Y\to M$ be surjective submersions. If $f:N\to M$ is a smooth map, a \textit{(smooth) morphism of surjective submersions covering $f$} from $\pi_X: X\to N$ to $\pi_Y: Y\to M$, denoted $(\hat{f}, f)$, is a smooth map $\hat{f}: X\to Y$ making \[\begin{tikzcd}
	X\arrow[d, "\pi_X", swap] \arrow{r}{\hat{f}} & Y\arrow[d, "
	\pi_Y"] \\
	N \arrow{r}{f} &M
	\end{tikzcd}\] commute. A morphism of surjective submersions $\hat{f}:X\to Y$ covering $f$ is an \textit{isomorphism of surjective submersions covering $f$} if $f$ and $\hat{f}$ are diffeomorphisms. If $N=M$ and $f = id_M$, we call $\hat{f}$ a \textit{morphism of surjective submersions (over $M$).}
\end{definition}
 \begin{definition}\label{fprod} Let $\pi_X: X\to M$ and $\pi_Y: Y\to M$ be surjective submersions.  \begin{enumerate}[(1),font=\upshape]
\item The \textit{fibre product of $X$ and $Y$ over $M$}, denoted $X\times_M Y$, is the submanifold whose underlying space is defined by $$X\times_M Y:= \left\{(x,y) \in X\times Y \ \rvert\  \pi_X(x) = \pi_Y(y)\right\}.$$
\item For $p\in \N$, the \textit{p-fold fibre product} of $\pi_Y:Y\to M$ is the submanifold $Y^{[p]}$ whose underlying space is defined by $$Y^{[p]} = \left\{(y_1, ..., y_p) \in Y^p \ \rvert \  \pi_Y(y_1) = \cdots = \pi_Y(y_p)\right \}.$$
\item  For $i = 1, ..., p-1$, the \textit{i-th projection map} $\pi_i: Y^{[p]}\to Y^{[p-1]}$ is defined by $$(y_1, ..., y_p)\mapsto (y_1, ..., \hat{y}_i, ..., y_p).$$ 
\end{enumerate}  \end{definition}

\begin{example}\label{fibreproductexample} 
	Suppose $Y$ is the product manifold $Y := M\times N$ and consider the surjective submersion $Y\to M$ given by projection onto the first factor. Then $Y^{[2]} \cong M\times N^2$. More generally, $Y^{[p]} \cong M\times N^p$. \end{example}
Given surjective submersions $\pi_1:X\to M$, $\pi_2:Y\to M$, it is not difficult to see that $X\times_MY$ is a submanifold and hence a manifold, and $\pi_1:X\times_MY\to M$ is a surjective submersion. Further, $Y^{[p]}$ is a submanifold of $Y^p$ and $Y^{[p]}\to M$ is a surjective submersion (Corollary \ref{do}). 

\begin{proposition}\label{prop submersion} 
	Suppose $f:N\to M$ is smooth and $\pi: Y\to M$ is a surjective submersion. Then $f^{-1}(Y)\subseteq N\times Y$ is a smooth submanifold and $f^{-1}(Y)\to N$ is a surjective submersion. 
\end{proposition}

\begin{corollary}\label{do}
Let $\pi:Y\to M$ be a surjective submersion. Then for all $p \in \N$, $Y^{[p]}$ is a smooth submanifold of $Y^p$, and $Y^{[p]}\to M$ is a surjective submersion.
\end{corollary}
\begin{proof} By setting $N = Y$ and $f = \pi$ in Proposition \ref{prop submersion}, we see that $\pi^{-1}(Y) \cong Y^{[2]}$ is a smooth submanifold of $Y^2$ and $\pi_2:Y^{[2]}\to Y$ is a surjective submersion. Since the composition of surjective submersions is a surjective submersion, $\pi\circ \pi_2:Y^{[2]}\to M$ is a surjective submersion. Repeating this argument $p-1$ times shows that $Y^{[p]}$ is a smooth submanifold of $Y^p$ and $Y^{[p]}\to M$ is a surjective submersion for all $p\in \N$. \end{proof} 
\begin{remark}
	From this proof, we see that $\pi_i: Y^{[p]}\to Y^{[p-1]}$ is a (smooth) surjective submersion for all $i$. 
\end{remark}
\begin{remark}\label{submersionnotation} 
	If $f:X\to Y$ is a morphism of surjective submersions over $M$ then there is an induced morphism $f^{[p]}: X^{[p]}\to Y^{[p]}$ of surjective submersions over $M$ defined in the obvious way.

\end{remark}
We conclude this section by defining $\delta$ on smooth maps, line bundles, and sections. 
\noindent \noindent \begin{definition}\label{deltadefinition} Let $P\to Y^{[p-1]}$ be a line bundle, $\pi:Y\to M$ be a surjective submersion and $A$ be an abelian Lie group. 
	\begin{enumerate}[(1),font=\upshape]
		\item Let $g:Y^{[p-1]}\to A$ be a smooth map. Define $\delta(g):Y^{[p]}\to A$ by $$\delta(g) := (g\circ \pi_1) - (g\circ\pi_2)+(g\circ \pi_3)+\cdots +(-1)^{p+1}g\circ \pi_p.$$
		\item The line bundle $\delta(P)\to Y^{[p]}$ is defined by $$\delta(P) := \pi_1^{-1}(P)\otimes \pi_2^{-1}(P)^*\otimes \pi_{3}^{-1}(P)\otimes \cdots $$ where the final term in this product is $\pi_p^{-1}(P)$ if $p$ is odd and $\pi_p^{-1}(P)^*$ if $p$ is even.
		\item Let $s\in \Gamma(Y^{[p-1]}, P)$ be a unit section, and define $s^*(y)\in P_y$ by $s^*(y)(s(y)) = 1$. Then $$\delta(s) := (s\circ \pi_1)\otimes (s^*\circ \pi_2)\otimes \cdots $$ where the final term is $(s\circ \pi_p)$ if $p$ is odd and $(s^*\circ \pi_p)$ if $p$ is even.
	 \end{enumerate} 
\end{definition} \begin{remark}\label{deltaistrivial} 
	It is easily verified that $\delta(\delta(P))$ is canonically trivial, with canonical section denoted by $1$. It can be shown that $1 = \delta(\delta(s))$ for $s$ a non-zero section of $P$. 
\end{remark} 
\section{Definition} \label{b2} \indent \indent In this section, we present the long awaited definition of a \textit{bundle gerbe.}
We shall also remark upon several equivalent descriptions of bundle gerbes, including those in terms of \textit{groupoids} and \textit{\v{C}ech cocycles}. These remarks are provided for the reader's enrichment, and are the first of many comparisons to be made between bundle gerbes and line bundles throughout this chapter. We conclude by presenting three standard examples of bundle gerbes, in an effort to solidify the theory from this section.\\ 
 \indent As remarked earlier, bundle gerbes can be viewed as `higher' versions of line bundles, in the sense that bundle gerbes are classified by degree-three cohomology, while line bundles are classified by degree-two cohomology. With this in mind, it is perhaps not surprising that a bundle gerbe consists of a line bundle with some additional (algebraic) structure.

\begin{definition}[\cite{INTROTOBG}] \label{bgdefn} 
	A \textit{bundle gerbe} $(P, Y, M)$ consists of a surjective submersion $Y\to M$, a hermitian line bundle $P\to Y^{[2]}$, and a \textit{bundle gerbe multiplication} $$m: \pi_3^{-1}(P)\otimes \pi_1^{-1}(P)\to \pi_2^{-1}(P)$$  which is a smooth isomorphism of hermitian line bundles over $Y^{[3]}$. Moreover, the bundle gerbe multiplication is \textit{associative}, that is, for any $(y_1, y_2, y_3, y_4)\in Y^{[4]}$, the bundle gerbe multiplication makes \begin{center} 
		\begin{tikzcd}
		P_{(y_1, y_2)}\otimes P_{(y_2, y_3)}\otimes P_{(y_3, y_4)}  \arrow[r] \arrow[d]
		& P_{(y_1, y_3)}\otimes P_{(y_3, y_4)} \arrow[d] \\
		P_{(y_1, y_2)}\otimes P_{(y_2, y_4)} \arrow[r]
		& P_{(y_1, y_4)}
		\end{tikzcd} \end{center}
	commute.
\end{definition} 
\begin{remark} In Murray's original 1996 paper \cite{1996bundlegerbes}, bundle gerbes were defined in terms of a \textit{fibration} $Y\to M$, i.e. a locally trivial fibre bundle. It was quickly observed by way of example that this was too restrictive a condition. However, the progression from fibrations to surjective submersions in the literature appears somewhat non-linear. For instance, the definition was refined in terms of (surjective) submersions in Carey--Murray's 1996 paper \cite{fadeevanomoly}, but fibrations reappear in the 1997 paper \cite{higherbgs}. In other early papers, including \cite{fadeevanomoly}, bundle gerbes are defined in terms of submersions, which are implicitly assumed to be surjective. The 2005 paper of Murray et~al. \cite{chernsimmons} appears to be the first instance of the explicit use of the term `surjective submersion'. We should also mention that in the 2000 paper \cite{stableiso}, in place of surjective submersions, Murray--Stevenson consider `locally split' maps, i.e. maps $\pi:Y\to M$ that admit local sections. This however was only required in a topological setting and is unrelated to the above discussion.
\end{remark}
When the base space $M$ is understood, we denote a bundle gerbe $(P, Y, M)$ by $(P, Y)$, and say $(P, Y)$ is a bundle gerbe \textit{over} $M$. At times, we will need to emphasise the surjective submersion $\pi:Y\to M$ present in a bundle gerbe. In these instances, we denote the bundle gerbe by $(P, Y, M, \pi)$. A pictorial description of a bundle gerbe is presented in Figure \ref{fig:bundlegerbe}. Here, the arrows from $Y^{[2]}$ to $Y$ depict the first and second projection maps $\pi_1$ and $\pi_2$ (Definition \ref{fprod}). 

\begin{figure}[bth!]
	$$ \begin{array}{ccc}
	P      &          &        \\
	\downarrow &         &         \\
	Y^{[2]}  & \stackrel{}{\dra} &  Y\\
	&               &      \downarrow \\
	&               &       M      \\
	\end{array} $$ 
	\caption{A bundle gerbe $(P, Y, M)$}
\label{fig:bundlegerbe}
\end{figure}
\indent The definition of a bundle gerbe presented above is, in the author's opinion, the easiest to digest at a first encounter. Of course, it is instructive to consider equivalent descriptions, which we explore in the following remarks.
\begin{remark} Let $(P, Y, M)$ be a bundle gerbe. The existence and associativity requirements for bundle gerbe multiplication in Definition \ref{bgdefn} can be restated more concisely in terms of a unit section of $\delta(P)\to Y^{[3]}$ (Definition \ref{deltadefinition}) as follows. Define a section $s$ of $\delta(P)\to Y^{[3]}$ by the bundle gerbe multiplication $m$. That is, for $(y_1, y_2, y_3)\in Y^{[3]}$, set $$s(y_1, y_2, y_3) = p_{23}\otimes m(p_{12}\otimes p_{23})^*\otimes p_{12}\in P_{(y_2, y_3)}\otimes P_{(y_1, y_3)}^*\otimes P_{(y_1, y_2)} $$ for any $p_{12} \in P_{(y_1, y_2)}$ and $p_{23} \in P_{(y_2, y_3)}$ with $p_{12}, p_{23}\neq 0$. Since $\delta(\delta(P))$ is canonically trivial (Remark \ref{deltaistrivial}), we can ask that $\delta(s)$ is the canonical section $1$. A simple computation shows that $\delta(s) = 1$ if and only if associativity of bundle gerbe multiplication holds. Therefore the existence of the bundle gerbe multiplication in Definition \ref{bgdefn} is equivalent to the existence of a non-zero section $s$ of  $\delta(L)\to Y^{[3]}$, and associativity of this multiplication is equivalent to $\delta(s) = 1$.
\end{remark} 
\begin{remark} \label{principalbundle} 
	Let $(P, Y, M)$ be a bundle gerbe. In the literature, $P\to Y^{[2]}$ is often defined to be a \textit{principal $U(1)$-bundle}, rather than a hermitian line bundle. A principal $U(1)$-bundle is a type of \textit{principal $G$-bundle} defined for a Lie group $G$. Principal $G$-bundles are conceptually similar to vector bundles, but are enriched by a group action of $G$ on the total space (see \cite{TU} for more on this). The choice of structure on $P\to Y^{[2]}$ results in two definitions of a bundle gerbe, which can be shown to be equivalent. This is a consequence of the equivalence between the categories of principal $U(1)$-bundles and hermitian line bundles, which can be proved via the \textit{associated bundle construction}. 
\end{remark}

\begin{remark}
	A \textit{groupoid} is a category in which every morphism is invertible. In particular, a \textit{$U(1)$-groupoid} is a transitive groupoid (i.e. a groupoid in which any two objects are connected by at least one morphism) such that the automorphism groups of objects are identified with $U(1)$. A bundle gerbe over a point $m\in M$ can be thought of as a $U(1)$-groupoid. A bundle gerbe over $M$ is then a `bundle of $U(1)$-groupoids', with $U(1)$-groupoids parametrised by the base manifold $M$ (more precisely, this is a certain kind of sheaf of $U(1)$-groupoids on $M$).  For more on this, we refer the reader to \cite{1996bundlegerbes}. 
\end{remark}

\begin{remark}\label{cocycles bundle gerbes} 
In Section \ref{section dd class}, we will show that associated to any bundle gerbe $(P, Y, M)$ are \v{C}ech $2$-cocycles $g_{\alpha \beta \gamma}:U_\alpha \cap U_\beta\cap U_\gamma \to U(1)$ for $\{U_\alpha\}_{\alpha \in I}$ an open cover of $M$. The cocycles $\{g_{\alpha \beta \gamma}\}_{\alpha, \beta, \gamma \in I}$ determine $(P, Y, M)$ (up to \textit{stable isomorphism}), hence can be used to define a bundle gerbe. This is analogous to the description of line bundles in terms of transition functions $g_{\alpha \beta}:U_\alpha \cap U_\beta \to U(1)$. 
\end{remark}
These remarks give us a first insight into the rich field of bundle gerbe theory. We have, however, neglected one fundamental question: what does a bundle gerbe \textit{look like}? We provide four straightforward examples of bundle gerbes below, thereby concluding this section.
\begin{example}\label{almosttrivial}
	This example aims to demystify the associativity condition required of bundle gerbe multiplication. Let $M$ be a smooth manifold and consider the simplest of surjective submersions, the identity map ${id}: M\to M$. Then $M\times \C \to M^{[2]}\cong M$ is a trivial line bundle and $(M\times \C, M, M, id)$ is a bundle gerbe, with multiplication defined fibrewise by \begin{gather*} m_x: (M\times \C)_{x}\otimes (M\times \C)_x \to (M\times \C)_x \\
	((x, u), (x, v)) \mapsto (x, uv).
	\end{gather*} In this instance, commutativity of the diagram in Definition \ref{bgdefn} (2) holds if and only if $(x, (uv)w) = (x, u(vw))$ for all $x\in M, u, v, w\in \C$. That is, associativity of bundle gerbe multiplication is precisely associativity of multiplication in $\C$ in this case.
\end{example}
The reader might expect the bundle gerbe in Example \ref{almosttrivial} to be called `the trivial bundle gerbe', because it appears to generalise the trivial line bundle $M\times \C\to M$. It is, indeed, part of a class of bundle gerbes that we call trivial, however it is not `\textit{the} trivial bundle gerbe'. The trivial bundle gerbe defined by a line bundle $R\to Y$ is constructed as follows.
\begin{example} \textbf{(Trivial bundle gerbes)} \label{trivialbundlegerbe} 
	Let $Y\to M$ be a surjective submersion and $R\to Y$ be a line bundle. Consider the line bundle $\delta(R)\to Y^{[2]}$ (Definition \ref{deltadefinition}). There is a natural multiplication map $$(R_{y_2}\otimes R_{y_1}^*)\otimes (R_{y_3}\otimes R_{y_2}^*)\to R_{y_3}\otimes R_{y_1}^*$$ for all $(y_1, y_2, y_3)\in Y^{[3]}$ which can be shown to be associative. This makes $(\delta(R), Y, M)$ into a bundle gerbe which we call \textit{the trivial bundle gerbe defined by $R$}. 
\end{example} 
Next, we define \textit{Hitchin-Chatterjee gerbes}, or \textit{local bundle gerbes}, in terms of cocycles $g_{\alpha \beta \gamma}$. This is analogous to the construction of a line bundle from its transition functions.
\begin{example}[\cite{INTROTOBG}]\label{hitchinchatconstruction}
Consider an open cover $\mathcal{U}= \{U_\alpha\}_{\alpha \in I}$ of a manifold $M$. Define $Y_{\mathcal{U}} = \{(x, \alpha) \ | \ x\in U_\alpha\},$ the disjoint union of these open sets. Then $Y_{\mathcal{U}}\to M$ is a surjective submersion under the obvious projection, and $Y_\mathcal{U}^{[2]}$ consists of triples $(x, \alpha, \beta)$ with $x\in U_\alpha \cap U_\beta.$ Let $P = Y_\mathcal{U}^{[2]} \times U(1)$ be the trivial $U(1)$-bundle over $Y_\mathcal{U}^{[2]}$. There is a multiplication map $\pi_3^{-1}(P)\otimes \pi_1^{-1}(P)\to \pi_2^{-1}(P)$ given by $$(x, \alpha, \beta, v)\otimes (x, \beta, \gamma, w)\mapsto (x, \alpha, \gamma, g_{\alpha \beta \gamma}(x) vw) $$ for $(x, \alpha, \beta, \gamma)\in Y_\mathcal{U}^{[3]}$, $v, w\in U(1)$ and some $g_{\alpha \beta \gamma}: U_{\alpha}\cap U_{\beta}\cap U_\gamma \to U(1)$. This multiplication will be associative if and only if $g_{\alpha \beta \gamma}$ is a cocycle. When this is the case, $(P, Y_{\mathcal{U}}, M)$ is a bundle gerbe (in the sense of Remark \ref{principalbundle}), called a \textit{Hitchin-Chatterjee gerbe} or a \textit{local bundle gerbe}.
\end{example}
Finally, we will consider a non-trivial bundle gerbe called the \textit{lifting bundle gerbe.} The following example assumes familiarity with the concept of a principal $G$-bundle, see for example \cite[p.$\,$244]{TU}. 
\begin{example}[{\cite[p.$\,$13]{INTROTOBG}}]\label{exampleliftingbundlegerbe} 
Let $$1 \xrightarrow{} U(1) \xrightarrow{}\hat{\mathcal{G}} \xrightarrow{} \mathcal{G}\xrightarrow{} 1$$ 
be a central extension of Lie groups and $P\to M$ be a (principal) $\mathcal{G}$ bundle. If $\hat{P}\to M$ is a $\hat{\mathcal{G}}$ bundle, then $\hat{P}/{U(1)}$ is a $\mathcal{G}$ bundle and we call $\hat{P}$ a \textit{lift} of $P$ if $\hat{P}/{U(1)}\cong P$ as $\mathcal{G}$ bundles. Define $\tau:P^{[2]}\to \mathcal{G}$ by $p_1\tau(p_1, p_2) = p_2$. Then as $\hat{\mathcal{G}}\to \mathcal{G}$ is a $U(1)$ bundle, $\tau^{-1}(\hat{\mathcal{G}}) \to P^{[2]}$ is a $U(1)$ bundle. There is a natural bundle gerbe multiplication on this $U(1)$ bundle induced by the group product on $\hat{\mathcal{G}}$. Hence $(\tau^{-1}(\hat{\mathcal{G}}), P)$ is a bundle gerbe that we call the \textit{lifting bundle gerbe}.  It is shown in \cite{brylinskimcglaughlin} that the lifting bundle gerbe $(\tau^{-1}(\hat{\mathcal{G}}), P)$ is trivial if and only if $P$ has a lift.
\end{example}
\section{Constructions}\label{sectionconstructions} 
\indent \indent We now study how to make `new bundle gerbes from old'. The reader may notice that we are deviating from the structure of the previous chapter here, by leaving bundle gerbe morphisms for a later section. We do this in an effort to further familiarise the reader with bundle gerbes, before moving on to the more complicated topic of their morphisms.\\
\indent  Expectedly, the usual constructions for line bundles (products, duals, pullbacks, etc.) can be adapted to bundle gerbes. Several subtleties arise in these constructions (particularly in the pullback and product bundle gerbes). As a result, we approach this section gingerly, and point out some commonly ignored details in the literature. Each of the constructions described in the examples below will be used extensively in this work. Let us begin with the \textit{dual bundle gerbe,} based on the dual of line bundles (Example \ref{example dual line bundle}).
\begin{example} \textbf{(Dual)} Let $(P, Y, M)$ be a bundle gerbe. By taking the dual line bundle $P^{*}\to Y^{[2]}$, we obtain a bundle gerbe $(P, Y, M)^* := (P^{*}, Y, M)$, depicted in Figure \ref{figbanana}, called the \textit{dual bundle gerbe of $(P, Y, M)$}. The bundle gerbe multiplication on $P$ induces a bundle gerbe multiplication on $P^*$ on passing to duals.	
	
		\begin{align*} \begin{array}{ccc}
		P^*     &          &        \\
		\downarrow &         &         \\
		Y^{[2]} & \stackrel{}{\dra} &  Y\\
		&               &      \downarrow \\
	&               &       M     \\
	\end{array} \end{align*}
		
\captionof{figure}{A dual bundle gerbe}
			\label{figbanana}
\end{example}

Next, we detail the \textit{pullback bundle gerbe} by a morphism of surjective submersions. This appears to be the first construction of its kind in the literature, as we will remark in a moment. 

\begin{example}\textbf{{(Pullback)}}\label{example pullback} 
Let $(P, Y, M)$ be a bundle gerbe, $X\to N$ be a surjective submersion, and $(\hat{f}:X\to Y, f:N\to M)$ be a morphism of the surjective submersions. Pulling back $P\to Y^{[2]}$ by $\hat{f}^{[2]}$ we obtain a line bundle $(\hat{f}^{[2]})^{-1}(P)\to X^{[2]}$. The triple $$(\hat{f}, f)^{-1}(P, Y, M):=( (\hat{f}^{[2]})^{-1}(P), X, N)$$ is the \textit{pullback bundle gerbe by $\hat{f}$ covering $f$}, or the \textit{pullback bundle gerbe by $(\hat{f}, f)$}, depicted in Figure \ref{fig:pullback}. It is straightforward to check that the associative bundle gerbe multiplication on $(P, Y, M)$ endows the pullback bundle gerbe with an associative multiplication. 
\end{example} 
	\begin{center} 
		\begin{tikzcd}[row sep = 2em, column sep=3em]
			(\hat{f}^{[2]})^{-1}(P) \arrow[r]  \arrow[d]		  &P \ \  \arrow[d, shift right = .7ex]   \\
			\ 	X^{[2]} \arrow[r, "\hat{f}^{[2]}"] \arrow[d, shift right=1.2ex] \arrow[d,shift left=.1ex]& \ Y^{[2]}  \arrow[d, shift right=1.2ex] \arrow[d,shift left=.1ex] \\
			X \ \arrow[r, "\hat{f}"] \arrow[d] & \ Y \  \  \arrow[d, shift right = .7ex] \\ 
			N \arrow[r, "  f"] & M \ \ 
		\end{tikzcd}  \captionof{figure}{A pullback bundle gerbe\\}	\label{fig:pullbackexample33} 
		\label{fig:pullback} \end{center} 
	 \begin{remark}\label{remark 2} In the literature, the pullback of a bundle gerbe is almost always defined along a map of base spaces, rather than a map of surjective submersions as we have done. This is a special instance of our definition. To see this, consider a map $f:N\to M$ and a bundle gerbe $(P, Y, M)$. By Proposition \ref{prop submersion}, $f^{-1}(Y)\to N$ is a  surjective submersion. Let ${pr}_2: f^{-1}(Y)\to Y$ be projection onto the second factor. Then ${pr}_2$ covers $f$, and the \textit{pullback of $(P, Y, M)$ by $f:N\to M$}, denoted $f^{-1}(P, Y, M)$, is defined as the pullback of $(P, Y, M)$ by $({pr}_2,f)$. That is, $f^{-1}(P, Y, M) := ({pr}_2, f)^{-1}(P, Y, M)$. It is not difficult to check that this definition agrees with the definition in the literature. 
\end{remark}

\begin{remark}\label{remark 1} \label{remark on pullbacks of sub} 
The base spaces in Example \ref{example pullback} need not be different. Let $(P, Y, M)$ be a bundle gerbe and $X\to M$ be a surjective submersion. Consider a morphism of surjective submersions $\hat{f}:X\to Y$ over $M$ (Definition \ref{morphismsurjsub}). The \textit{pullback of $(P, Y, M)$ by $\hat{f}$}, denoted $\hat{f}^{-1}(P, Y, M)$, is defined to be the pullback of $(P, Y, M)$ by $\hat{f}$ covering the identity. That is, $\hat{f}^{-1}(P, Y, M) := (\hat{f}, id_M)^{-1}(P,Y, M)$.
\end{remark}
\begin{remark}
The two previous remarks encompass all possible pullback bundle gerbes. That is, the pullback of any bundle gerbe $(P, Y, M)$ by $\hat{f}:X\to Y$ covering $f:N\to M$ can be written as a composition of pullback bundle gerbes of the forms in Remarks \ref{remark 2} and \ref{remark 1}. This follows by commutativity of the diagram in Figure \ref{fig:pullbackexample34} and by noting that $\hat{f} = {pr}_2 \circ (\pi\times \hat{f})$ for ${pr}_2:f^{-1}(Y)\to Y$ projection onto the second factor. 
	\begin{center} 
		\begin{tikzcd}
			X \arrow[r, "\pi \times \hat{f}"] \arrow[d] & f^{-1}(Y)  \arrow[d]  \arrow[r, "{pr}_2"] & Y\arrow[d] \\ 
			N \arrow[r, "id_N"] & N \arrow[r, "f"] & M
		\end{tikzcd} \captionof{figure}{Composition of surjective submersions\\}	\label{fig:pullbackexample34} 
\end{center}  \end{remark}

Next, we study the \textit{reduced product bundle gerbe} and \textit{product bundle gerbe} constructions. The reduced product bundle gerbe can be defined for bundle gerbes with the same surjective submersion, and is a straightforward generalisation of Example \ref{product line bundle} (products of line bundles). However, to define the product of bundle gerbes $(P, Y, M)$ and $(Q, X, M)$ more generally requires some extra effort.  We detail these constructions below. 
\begin{example}\textbf{{(Reduced product)}}\label{examplereducedproducts} 
	Let $(P, Y, M)$ and $(Q, Y, M)$ be bundle gerbes. Recall that, since $P\to Y^{[2]}$ and $Q\to Y^{[2]}$ are line bundles, so too is $P\otimes Q \to Y^{[2]}$ (Example \ref{product line bundle}). The \textit{reduced product bundle gerbe} is the triple $$(P, Y, M)\otimes_{\text{red}} (Q, Y, M) := (P\otimes Q, Y, M). $$ This is clearly a bundle gerbe, with an associative multiplication induced from the associative multiplication on $(P, Y, M)$ and $(Q, Y, M)$. 
\end{example}

\begin{example} \textbf{(Product)}
	Consider bundle gerbes $(Q, X, M)$ and $(P, Y, M)$. The surjective submersion present in the product bundle gerbe will be $X\times_MY\to M$. We construct a line bundle over $(X\times_MY)^{[2]}$ via a two-step process. First, we pullback the line bundles $Q\to X^{[2]}$ and $P\to Y^{[2]}$ by the projections $$\pi_X^{[2]}: X^{[2]}\times_MY^{[2]}\to X^{[2]}, \ \ \pi_Y^{[2]}: X^{[2]}\times_MY^{[2]}\to Y^{[2]}$$ respectively, to obtain line bundles $(\pi_X^{[2]})^{-1}(Q)$ and $(\pi_Y^{[2]})^{-1}(P)$ over $X^{[2]}\times_MY^{[2]}$. Since these line bundles are defined over the same base space, we can take their product in the sense of Example \ref{product line bundle}, to define the line bundle \begin{align}\label{eqneee} Q\otimes P :=  (\pi_X^{[2]})^{-1}(Q)\otimes (\pi_Y^{[2]})^{-1}(P)\end{align} over $X^{[2]}\times_MY^{[2]}\cong (X\times_MY)^{[2]}$. It follows that the triple $$(Q, X, M) \otimes (P, Y, M) := (Q\otimes P, X\times_MY, M)$$ is a bundle gerbe, depicted in Figure \ref{figpear}, called the \textit{product bundle gerbe of $(Q, X, M)$ and $(P, Y, M)$.} Here, the bundle gerbe multiplication is taken to be the tensor product of the multiplication maps on $(Q, X, M)$ and $(P, Y, M)$.
	
	$$ \begin{array}{ccc}
	Q\otimes P     &          &        \\
	\downarrow &         &         \\
	X^{[2]}\times_M Y^{[2]} & \stackrel{}{\dra} &  X\times_MY\\
	&               &      \downarrow \\
	&               &       M     \\
	\end{array} $$
	\captionof{figure}{A product bundle gerbe\\}
	\label{figpear}
	
\end{example}
\begin{remark}\label{mug}
	The product and reduced product of bundle gerbes can be related as follows. Consider bundle gerbes $(P, Y, M)$ and $(Q, Y, M)$ and let $\pi_Y: Y\to Y\times_MY$ be the map $y\mapsto (y, y)$. Then $$(P, Y, M)\otimes_{\text{red}} (Q, Y, M) \cong \pi_Y^{-1}( (P, Y, M)\otimes (Q, Y, M))$$ where we pullback the product bundle gerbe in the sense of Remark \ref{remark 1}. This is similar to the situation for vector bundles in Remark \ref{choccake}. Alternatively, if $\pi_X: X\times_MY\to X$ and $\pi_Y:X\times_MY\to Y$ are the natural projections, then it is not difficult to see that $(P, Y, M)\otimes (Q, X, M) = \pi_Y^{-1}(P, Y, M)\otimes_{\text{red}} \pi_X^{-1}(Q, X, M)$. 
\end{remark}
This concludes our study of bundle gerbe constructions for the time being. The reader may have noticed that the \textit{determinant} and \textit{function power} constructions (Examples \ref{determinant line bundle }, \ref{tereq}--\ref{function power}) did not give rise to a bundle gerbe construction in this section. We will revisit these constructions later and see that they are key to defining the cup product bundle gerbe and basic bundle gerbe.
\section{Morphisms and equivariance}\label{section morphisms} 
\indent \indent In this section, we detail morphisms of bundle gerbes, which are foundational to the rest of this chapter. In contrast to line bundles, there are several kinds of bundle gerbe morphisms, including \textit{bundle gerbe isomorphisms,} \textit{stable isomorphisms} and \textit{$D$-stable isomorphisms}. In particular, stable isomorphisms will be used to classify bundle gerbes, and $D$-stable isomorphisms will be used to classify bundle gerbes \textit{with connective data}.\\
\indent We begin by defining bundle gerbe morphisms and isomorphisms. Trivial bundle gerbes are then introduced, enabling us to define stable isomorphisms. 
We then define \textit{$K$-equivariant bundle gerbes}, for $K$ a compact Lie group, together with $K$-equivariant bundle gerbe isomorphisms and stable isomorphisms. To conclude this section, we will show that $K$-equivariance is preserved under bundle gerbe products and pullbacks, which will be relevant to our later work. The definition of $D$-stable isomorphisms will be postponed until we discuss \textit{connective data}. \\
\indent Loosely, a bundle gerbe morphism $(P, Y, M) \to (Q, X, N)$ is a collection of maps $P\to Q$, $Y\to X$ and $M\to N$ that \textit{respect} the bundle gerbe structure in an appropriate way. We make this definition precise below. 
\begin{definition}[\cite{1996bundlegerbes}] \label{bundle gerbe morphism}
	A \textit{morphism of bundle gerbes} $(P, Y, M)\to(Q, X, N)$ is  a triple of maps $(\doublehat{f}, \hat{f}, f)$ such that: \begin{enumerate}[(1)]
		\item the map $\hat{f}:Y\to X$ is a morphism of surjective submersions $Y\to M$ and $X\to N$ covering $f: M\to N$;
		\item the map $\doublehat{f}:P\to Q$ is a morphism of hermitian line bundles covering $\hat{f}^{[2]}$ that commutes with the bundle gerbe product. That is, if $m_P$, $m_Q$ are the bundle gerbe products on $P$ and $Q$ respectively, then for all $v\otimes w \in \pi_3^{-1}(P)\otimes \pi_1^{-1}(P)$, $$\doublehat{f}(m_P(v\otimes w)) = m_Q(\doublehat{f}(v)\otimes \doublehat{f}(w)).$$
 	\end{enumerate}
\end{definition} 
\begin{definition} A morphism $(\doublehat{f}, \hat{f}, f)$ of bundle gerbes is an \textit{isomorphism} if each of $\doublehat{f}, \hat{f},$ and $f$ are diffeomorphisms. If $(P, Y, M)$ and $(Q, X, N)$ are isomorphic we write $(P, Y, M)\cong (Q, X, N)$. \end{definition}
\indent  We now divert our attention to \textit{trivial bundle gerbes}, which will be required to define stable isomorphisms of bundle gerbes. Equipped with the definition of a bundle gerbe isomorphism, and having defined the so-called trivial bundle gerbe defined by $R$ (Example \ref{trivialbundlegerbe}), the definition of triviality is not surprising. 
\begin{definition}[\cite{INTROTOBG}]A bundle gerbe $(P, Y, M)$ is \textit{trivial} if there is a line bundle $R\to Y$ such that $(P, Y, M)$ is isomorphic to $(\delta(R), Y, M)$. The choice of $R$ together with the bundle gerbe isomorphism $(P, Y, M)\xrightarrow{\sim} (\delta(R), Y, M)$ is called a \textit{trivialisation} of $(P, Y, M)$. We will sometimes call $R\to Y$ a \textit{trivialising line bundle} of $(P, Y,  M)$.
\end{definition}
A bundle gerbe trivialisation is not unique. This bears the question, how do two bundle gerbe trivialisations differ? It turns out that two trivialisations of a bundle gerbe differ by a line bundle. We state this formally below without proof.
\begin{proposition}[{\cite[Proposition 5.2]{INTROTOBG}}]\label{difference of trivialisations} 
 If $(P, Y, M, \pi)$ is a trivial bundle gerbe with trivialising line bundles $R$ and $R'$, then there exists a line bundle $Q\to M$ for which $R = R'\otimes \pi^{-1}(Q).$
\end{proposition} 
To reinforce the definition of bundle gerbe triviality, we provide the following example, which is analogous to a well-known result for line bundles.
\begin{example}[{\cite[Example 5.1]{INTROTOBG}}]\label{teaq}
	Let $(P, Y, M)$ be a bundle gerbe and suppose $Y\to M$ admits a global section $s$. We claim $(P, Y, M)$ is trivial. Consider the map $f:Y\to Y^{[2]}$ defined by $f(y) = (s(\pi(y)), y)$. Then $R:=f^{-1}(P)\to Y$ is a line bundle, and for any $(y_1, y_2)\in Y^{[2]}$, \begin{align*}
		\delta(R)_{(y_1, y_2)} &\cong P_{(s(\pi(y_1)), y_1)}^*\otimes P_{(s(\pi(y_2)), y_2)} \\
		&= P_{(y_1, s(\pi(y_1)))}\otimes P_{(s(\pi(y_2)), y_2)}  \\
		&= P_{(y_1, y_2)},
	\end{align*} using that $s(\pi(y_1)) = s(\pi(y_2))$ and $P_{(s(\pi(y_1)), y_1)} = P_{(y_1, s(\pi(y_1)))}^*$. It can be verified this isomorphism preserves the bundle gerbe multiplication, hence is a bundle gerbe isomorphism. Therefore $R$ together with this isomorphism is a trivialisation of $(P, Y, M)$.
\end{example}
\begin{remark}
Recall the bundle gerbe $(M\times \C, M, M, id)$ from Example \ref{almosttrivial}. At the time, we cryptically called this bundle trivial, but not `the trivial bundle gerbe'. By Example \ref{teaq}, we see that $(M\times \C, M, M)$ is a trivial bundle gerbe, hence is isomorphic to a bundle gerbe of the form described in Example \ref{trivialbundlegerbe}.
\end{remark}

As we have alluded, the definition of a bundle gerbe isomorphism presented above is too restrictive to classify bundle gerbes. This motivates the notion of a \textit{stable isomorphism.} We will see later why stable isomorphisms are the correct notion of equivalence for bundle gerbes.
\begin{definition}\label{stableisodefn} 
	Two bundle gerbes $(P, Y, M)$ and $(Q, X, M)$ are \textit{stably isomorphic} if $(P, Y)^*\otimes (Q, X)$ is trivial. We write $(P, Y) \cong_{\text{stab}} (Q, X)$. A choice of trivialisation of $(P, Y)^*\otimes(Q, X)$ is called a \textit{stable isomorphism from $(P, Y)$ to $(Q, X)$}. 
\end{definition}
\begin{remark} Stable isomorphisms of bundle gerbes were developed in 2000 by Murray--Stevenson \cite{stableiso}, four years after Murray's original bundle gerbe paper \cite{1996bundlegerbes}. The former paper details the significance of stable isomorphisms to the study of \text{local bundle gerbes} and to the classification of bundle gerbes in terms of $B\C^*$ bundles.
	\end{remark}

\begin{remark}\label{transitivityofstableiso}
	It is shown in  \cite[Proposition 3.9]{dannythesis} that stable isomorphisms can be composed.
\end{remark}
\begin{remark}\label{remark isoimplies stable iso} 
	If two bundle gerbes are isomorphic, they are stably isomorphic. The converse is not in general true. 
\end{remark}

\begin{remark}\label{remarkexample2} 
Consider a bundle gerbe $(P, Y, M, \pi)$ and a submanifold $X$ of $Y$. If $\pi|_X:X\to M$ is a surjective submersion, then $(P, Y, M) \cong_{\text{stab}}(P|_{X^{[2]}}, X, M)$. This is a special case of \cite[Proposition 3.4]{stableiso}. 
\end{remark}
\begin{remark}\label{fibreproductsimplifies}
Consider bundle gerbes $(P_i, Y, M)$ for $i=1,..., n$. By taking their product we obtain the bundle gerbe $\left(\Motimes_{i=1}^nP_i, Y\times_M\cdots \times_MY, M\right).$ Since $Y$ is a submanifold of $Y\times_M\cdots \times_MY$, and $Y\to M$ is a surjective submersion, by Remark \ref{remarkexample2} $$\bigotimes_{i=1}^n\left(P_i, Y, M\right) \cong_{\text{stab}} \sideset{}{_\mathrm{red}}\bigotimes_{i=1}^n \left(P_i, Y, M\right).$$
\end{remark}

The bundle gerbes central to our research come endowed with a natural group action. To make sense of this, we extend the definition of a $K$-equivariant line bundle to define a $K$-equivariant bundle gerbe for $K$ a compact Lie group. Equivariant bundle gerbes first appeared in the 2003 paper by Meinrenken \cite{Meinrenken}.
\begin{definition} [\cite{53UNITARY}]\label{equivariantbundlegerbedefn} 
	Let $(P, Y, M, \pi)$ be a bundle gerbe and $K$ be a compact Lie group that acts smoothly on $M$. Then $(P, Y, M, \pi)$ is a (strongly) $K$-\textit{equivariant bundle gerbe} if the following conditions hold:
	\begin{enumerate}[(1)]
		\item there is an action of $K$ on $Y$ making $\pi$ a $K$-equivariant map;
		\item the line bundle $P\to Y^{[2]}$ is $K$-equivariant for the induced $K$-action on $Y^{[2]}$;
		\item the section $s$ of $\delta(P)\to Y^{[2]}$ induced by bundle gerbe multiplication is $K$-equivariant.
	\end{enumerate}
\end{definition}

\begin{remark}
There is also the notion of a \textit{weakly $K$-equivariant bundle gerbe.} We will not define this as the two bundle gerbes central to our research are strongly equivariant. For more information, see \cite{equivariantbundlegerbes}.
\end{remark}
\begin{remark}
	Condition $(3)$ in Definition \ref{equivariantbundlegerbedefn} can be interpreted as the requirement that the bundle gerbe multiplication map is $K$-equivariant. Therefore a bundle gerbe is $K$-equivariant if each of the geometric \textit{and} algebraic structures present in the bundle gerbe are $K$-equivariant in the appropriate sense.
	\end{remark}
\begin{remark}\label{greenapple} 
	The standard constructions of pullbacks, products, and duals of bundle gerbes can be extended to the $K$-equivariant setting. In particular, if $R\to Y$ is a $K$-equivariant line bundle, and $\pi:Y\to M$ is $K$-equivariant, then the trivial bundle gerbe $(\delta(R), Y, M, \pi)$ is $K$-equivariant. 
\end{remark}
Naturally, an (isomorphism/stable isomorphism) of $K$-equivariant bundle gerbes is an (isomorphism/stable isomorphism)  that preserves the group action in the appropriate sense. We formalise this below. 
\begin{definition}
	Let $(P, Y, M)$ and $(Q, X, M)$ be $K$-equivariant bundle gerbes for $K$ a compact Lie group.	\begin{enumerate}[(1)]
		\item An isomorphism $(\doublehat{f}, \hat{f}, f)$ is a \textit{$K$-equivariant isomorphism}, or \textit{$K$-isomorphism}, if each of $\doublehat{f}$, $\hat{f}$, and $f$ are $K$-equivariant. If $(P, Y, M)$ and $(Q, X, M)$ are $K$-isomorphic, we write $(P, Y, M)\cong_K(Q, X, N)$.
		\item A stable isomorphism $(\doublehat{f}, \hat{f}, f)$ is a \textit{$K$-equivariant stable isomorphism}, or \textit{$K$-stable isomorphism}, if each of $\doublehat{f}$, $\hat{f}$, and $f$ are $K$-equivariant, the trivialising line bundle $R\to X\times_M Y$ is $K$-equivariant, and the isomorphism $\delta(R)\cong (P, Y)^*\otimes (Q, X)$ is $K$-equivariant. If $(P, Y, M)$ and $(Q, X, M)$ are $K$-stably isomorphic we write $(P, Y, M)\cong_{K\text{-stab}}(Q, X, M)$.\end{enumerate}
\end{definition}

\indent We conclude this section by showing that $K$-equivariance is preserved under bundle gerbe products and pullbacks, thereby partially proving Remark \ref{greenapple}. We focus on these constructions in particular, since they will be used later in this work. 
\begin{proposition}\label{equivariantbundlegerbeproduct}
Let $K$ be a compact Lie group	and $(P_i, Y_i, M, \pi_i)$ be $K$-equivariant bundle gerbes for $i=1, ..., n$. The product bundle gerbe $\Motimes_{i=1}^n( P_i, Y_i, M)$ is $K$-equivariant.
\end{proposition}
\begin{proof} Consider the induced surjective submersion $\pi: Y_1\times_M\cdots \times_MY_n \to M$. Certainly the action of $K$ on each $Y_i$ induces an action of $K$ on their fibre product, and equivariance of $\pi$ will follow by equivariance of each $\pi_i$. We next verify the line bundle \begin{align}\label{linebundle1} \Motimes_{i=1}^n P_i\to (Y_1\times_M\cdots \times_M Y_n)^{[2]} \cong Y_1^{[2]}\times_M\cdots \times_MY_n^{[2]}\end{align} is $K$-equivariant (Definition \ref{defn 1}). Once again, the action of $K$ on each $P_i$ and $Y_i^{[2]}$ induces an action on the product line bundle. Equivariance of each projection $P_i\to Y_i^{[2]}$ implies $K$-equivariance of the projection in (\ref{linebundle1}).
	 Now, the $K$-action restricts to a linear map $(P_i)_{(y_i, y_i')}\to (P_i)_{k\cdot (y_i, y_i')}$ for each $i$, so the $K$-action also restricts to a linear map in the product $$(P_1)_{(y_1, y_i')} \otimes \cdots \otimes (P_n)_{(y_n, y_n')}\to (P_1)_{k\cdot (y_1, y_i')} \otimes \cdots \otimes (P_n)_{k\cdot (y_n, y_n')}.$$ This shows that the product line bundle (\ref{linebundle1}) is $K$-equivariant. It remains to show part $(3)$ of Definition \ref{equivariantbundlegerbedefn}, i.e. that the section defined by bundle gerbe multiplication is $K$-equivariant. By equivariance of each $(P_i, Y_i, M)$, the section $s_i:\delta(P_i)\to Y_i^{[2]}$ induced by the bundle gerbe multiplication on $(P_i, Y_i, M)$ is $K$-equivariant. Therefore their product $s_1\otimes \cdots \otimes s_n$ is also $K$-equivariant. The section $s_1\otimes \cdots \otimes s_n$ is precisely the section of $$\delta(\otimes_{i=1}^n P_i)\to Y_1^{[2]}\times_M\cdots \times_M Y_n^{[2]}$$ defined by the bundle gerbe multiplication on the product bundle gerbe, hence the product bundle gerbe is $K$-equivariant. \end{proof}
 \begin{proposition}\label{proposition pullback equivariant bg is equivariant} 
 	Let $K$ be a compact Lie group and $(P, Y, M)$ be a $K$-equivariant bundle gerbe. Suppose $K$ acts smoothly on a manifold $N$, and $f:N\to M$ is a $K$-equivariant map. Then the pullback bundle gerbe $f^{-1}(P, Y, M)$ is $K$-equivariant. 
 \end{proposition}
\begin{proof}
By $K$-equivariance of $f$, the $K$-action on $Y$ pulls back to a $K$-action on $f^{-1}(Y)$ covering the action on $N$. It follows that the induced map $\hat{f}: f^{-1}(Y)\to Y$ is $K$-equivariant. Since the pullback of a $K$-equivariant line bundle by a $K$-equivariant map is a $K$-equivariant line bundle, $(\hat{f}^{[2]})^{-1}(P)\to p^{-1}(Y)^{[2]}$ is a $K$-equivariant line bundle. The remaining details can be checked easily. 
\end{proof}

 \section{Connections and curvature}\label{b3}
 \indent \indent Just as with line bundles, a bundle gerbe $(P, Y, M)$ can be endowed with a \textit{bundle gerbe connection}. As one might expect, a bundle gerbe connection is an extension of a line bundle connection, which we are well acquainted with from Chapter 1. In this section, we shall define bundle gerbe connections, curvings, and three-curvatures. The induced connective data on the dual, pullback, and product bundle gerbes will then be considered. We will conclude this section by considering connective data on the trivial bundle gerbe, along with the relationship between the connective data of stably isomorphic bundle gerbes. Let us begin with a preliminary proposition that will be used throughout the remainder of this chapter.
  \begin{proposition}[\cite{INTROTOBG}] \label{deltamap} 
 	Let $\Omega^q\left(Y^{[p]}\right)$ be the space of differential $q$-forms on $Y^{[p]}$. Define $\delta:\Omega^q\left(Y^{[p-1]}\right)\to \Omega^q\left(Y^{[p]}\right)$ by $$\delta = \sum_{i=1}^p (-1)^{i+1}\pi_i^*. $$ Then $\delta^2 = 0$ and the so-called \textit{fundamental complex} \begin{align}\label{thefundamentalcomplex}
 	0\xrightarrow{} \Omega^q(M) \xrightarrow{\pi^*} \Omega^q(Y) \xrightarrow{\delta} \Omega^q\left(Y^{[2]}\right) \xrightarrow{\delta} \Omega^q\left(Y^{[3]}\right)\xrightarrow{} \cdots 
 	\end{align} is exact. In particular, $\pi^*$ is injective.
 \end{proposition} 
Consider a bundle gerbe $(P, Y, M)$. It is natural that we should consider a bundle gerbe connection in terms of a line bundle connection $P\to Y^{[2]}$. Recall that bundle gerbes are enriched with an algebraic structure via the bundle gerbe multiplication map $m$. We call a line bundle connection $\nabla$ on $P\to Y^{[2]}$ a \textit{bundle gerbe connection} if it respects the bundle gerbe multiplication in an appropriate way. We make this precise in the proceeding definition. \\ 
  \indent Let $\pi^{*}_i \nabla$ be the induced pullback connection on $\pi_i^{-1}(P)$ for $i= 1, 2, 3$ (Example \ref{line bundle pullback}), and $\pi_3^*\nabla\otimes id_{1}+ id_3\otimes \pi_1^*\nabla$ be the induced tensor product connection on $\pi_3^{-1}(P)\otimes \pi_1^{-1}(P)$ (Example \ref{product line bundle}). Here, $id_1$ and $id_3$ denote the identity maps on $\pi_1^{-1}(P)$ and $\pi_3^{-1}(P)$, respectively. 
 \begin{definition}\label{bundlegerbeconnection} 
 	Consider a bundle gerbe $(P, Y, M)$ and a line bundle connection $\nabla$ on $P\to Y^{[2]}$. Let $m: \pi_3^{-1}(P)\otimes \pi_1^{-1}(P)\to \pi_2^{-1}(P)$ be the bundle gerbe multiplication map. Then $\nabla$ is a \textit{bundle gerbe connection} if  \begin{align}\label{respectbgmultiplication}  \pi_2^{*}\nabla m(\alpha\otimes \beta)= m\left((\pi_3^*\nabla\otimes id_1+ id_3\otimes \pi_1^*\nabla) (\alpha\otimes \beta)\right)\end{align} for all $\alpha \in \Gamma\left(\pi_3^{-1}(P)\right)$ and $\beta\in \Gamma\left(\pi_1^{-1}(P)\right)$.
 \end{definition}
 \begin{remark}[{\cite[p.$\,$9]{INTROTOBG}}]\label{bundle gerbe connections exist} 
Every bundle gerbe $(P, Y, M)$ has a bundle gerbe connection. 
 \end{remark}
 \begin{remark}
 	 Recall that a bundle gerbe can alternatively be defined by asking that $P\to Y^{[2]}$ be a principal $U(1)$-bundle (Remark \ref{principalbundle}). In this case, we would not have a line bundle connection on $P\to Y^{[2]}$, but instead, a \textit{connection $1$-form} $A$. The $1$-form $A$ induces a connection $1$-form on $\delta(P)\to Y^{[3]}$ that we denote by $\delta(A)$. We call $A$ a bundle gerbe connection if $s^{*}(\delta(A))=0$ for $s$ the section of $\delta(P)\to Y^{[3]}$ defined by the bundle gerbe multiplication. This is equivalent to Definition \ref{bundlegerbeconnection}. In fact, it is much easier to prove Remark \ref{bundle gerbe connections exist} with this definition. For more on this, see \cite{INTROTOBG}. 
 \end{remark} 
Before considering the connective data on a bundle gerbe, we present one elementary result that will be used later. Recall that two line bundle connections differ by a $1$-form on the base space (Proposition \ref{proposition two line bundle connections differ}). We have an analogous proposition for bundle gerbes. 
\begin{proposition} \label{differenceofbundlegerbeconnections} 
	Let $\nabla, \nabla'$ be bundle gerbe connections on $(P, Y, M)$. Then there exists $\alpha \in \Omega^1(Y^{[2]})$ such that $\delta(\alpha)=0$ and $$\nabla - \nabla' = \alpha.$$ 
\end{proposition}
\begin{proof}
	Since a bundle gerbe connection is a line bundle connection, $\nabla - \nabla' = \alpha$ for some $\alpha \in \Omega^1(Y^{[2]})$. As $\nabla$ is a bundle gerbe connection, $\nabla' + \alpha$ must also be a bundle gerbe connection, i.e. $\nabla' + \alpha$ must satisfy equation (\ref{respectbgmultiplication}). Using that $\nabla'$ is a bundle gerbe connection, we see that the latter holds if and only if $\delta(\alpha) =0$. 
\end{proof}

 We now move on to define the connective data and three-curvature of a bundle gerbe. Although the connection and curvature of bundles gerbes and line bundles appear superficially similar, we will find that the construction of the bundle gerbe curvature is quite different to what we have seen in Chapter \ref{ch:intro}. 

Consider a bundle gerbe connection $\nabla$ on $(P, Y, M, \pi)$. We construct the connective data and three-curvature associated to $\nabla$ as follows. By (\ref{respectbgmultiplication}), the connection $1$-form of $\pi_2^{*}\nabla$ will be the sum of connection $1$-forms of $\pi_3^{-1}(P)$ and $\pi_1^{-1}(P)$. Therefore the two-curvatures satisfy $F_{\pi_1^*\nabla} + F_{\pi_3^*\nabla} = F_{\pi_2^*\nabla},$ i.e. $\delta(F_{\nabla}) = 0$. By exactness of the fundamental complex, there exists $f\in \Omega^2(Y)$ for which $F_\nabla = \delta(f)$. Now, since $\delta$ commutes with $d$, $$\delta(df) = d\delta(f) = dF_\nabla = 0.$$ By exactness of the fundamental complex, this implies $df = \pi^*(\omega)$ for a unique $\omega \in \Omega^3(M)$ with $d\omega = 0$. We summarise this data below, and depict it in Figure \ref{figbanana2}.

\begin{definition}
	Let  $(P, Y, M, \pi)$ be a bundle gerbe with connection $\nabla$. \begin{enumerate}[(1)]
		\item A choice of $f \in \Omega^2(Y)$ for which $F_{\nabla} = \delta(f)$ is called a \textit{curving} for $\nabla$; 
		\item The pair $(\nabla, f)$ is called \textit{connective data} for $(P, Y, M, \pi)$; 
		\item The unique $\omega \in \Omega^3(M)$ for which $df = \pi^*\omega$ is called the \textit{three-curvature} of $(\nabla, f)$. 
	\end{enumerate}  \begin{figure}[thb!]	\begin{align*} \begin{array}{ccc}
\ \ \ \ \ \ \ \ \ \ \ \ \ \ \ \ P     &          &        \\
\ \ \ \ \ \ \ \ \ \ \ \ \ \  \ \ \downarrow &         &         \\
F_\nabla = \delta(f) \ \ \ \ Y^{[2]}  & \stackrel{}{\dra} &  Y\ \ \ \  f, \ df = \pi^*\omega\\
&               &      \downarrow \ \ \ \ \ \ \ \ \ \ \ \  \ \ \ \ \ \ \ \ \ \\
&               &       M \ \ \ \  \omega  \ \ \ \ \ \ \ \ \ \ \ \ \ \ \  \\
\end{array} \end{align*}
%
\captionof{figure}{Connective data on a bundle gerbe\\}
\label{figbanana2}\end{figure}
\end{definition}

\begin{remark}
We advise the reader to be wary of different terminology to ours in the bundle gerbe literature. In some papers, authors refer to the pair $(\nabla, f)$ as the \textit{connective structure} of a bundle gerbe. Our convention was chosen to remain compatible with Brylinski \cite{brylinski}, where connections on gerbes were called \textit{connective structures}. 
\end{remark}
\begin{remark}\label{remarkonequiariantconnectivestructure}
 There is a notion of \textit{equivariant connective data} on an equivariant bundle gerbe. We will not discuss this here, and refer the reader to \cite{Meinrenken} for details. 
\end{remark}
\begin{remark}
We will sometimes use the phrase `the three-curvature of the bundle gerbe' instead of `the three-curvature of the connective data of the bundle gerbe' when the connective data is understood. 
\end{remark}
With this definition in mind, we can consider the induced connective data and three-curvature on the dual, pullback, and product bundle gerbes. We do so via a series of propositions. Let us begin with the simplest construction, the dual bundle gerbe. 
\begin{proposition}\label{connectivestructureondualbg} \textbf{\textup{(Connective data on duals)}}
	Let $(P, Y, M)$ be a bundle gerbe with connective data $(\nabla, f)$ and three-curvature $\omega$. Then $(\nabla, f)^* := (\nabla^*, -f)$ is the canonical connective data on $(P, Y, M)^*$, with three-curvature $-\omega$.
\end{proposition}
\begin{proof}
	It can be shown easily from Definition \ref{bundlegerbeconnection} that the dual connection $\nabla^*$ is a bundle gerbe connection on the dual bundle gerbe. By Example \ref{example dual line bundle}, the curvature $F_{\nabla^*}$ is equal to $- F_{\nabla}$. Therefore a curving $f^*$ on $(P, Y, M)^*$  must satisfy $\delta(f^*) = -F_{\nabla}$, hence we can choose $f^*= -f$. The claim about the three-curvature follows easily. 
\end{proof}
In the next proposition, we consider the connective data on the pullback of a bundle gerbe by a morphism of surjective submersions (Example \ref{example pullback}). Of course, this proposition could be adapted easily to define the connective data on the pullback of a bundle gerbe by a morphism of surjective submersions over $M$, or by a map on the base space.
\begin{proposition}\label{curving of pullback}\textbf{\textup{(Connective data on pullbacks)}}
Let $(P, Y, M)$ be a bundle gerbe with connective data $(\nabla, f)$ and three-curvature $\omega$. Suppose $X\to N$ is a surjective submersion, and $(\hat{g}: X\to Y, g:N\to M)$ is a morphism of surjective submersions. Then there is canonical connective data $(\hat{g}, g)^*(\nabla, f) := ((\hat{g}^{[2]})^*\nabla, \hat{g}^*f)$ on $(\hat{g}, g)^{-1}(P, Y, M)$, with three-curvature $g^*\omega$. 
\end{proposition}
\begin{proof}
	Let $(\hat{g}^{[2]})^*\nabla$ denote the induced pullback connection on $(\hat{g}^{[2]})^{-1}P\to X^{[2]}$. It is straightforward to check that this is a bundle gerbe connection, and by Example \ref{line bundle pullback}, the curvature of this connection is $(\hat{g}^{[2]})^*F_{\nabla}$. It follows by naturality of the coboundary map $\delta$ that $\hat{g}^*f$ is a curving for the pullback connection. Hence $((\hat{g}^{[2]})^{*}\nabla, \hat{g}^*f)$ is connective data on $(\hat{g}, g)^{-1}(P, Y, M)$. The claim about the three-curvature follows easily. 
\end{proof}

Next, we study the connective data on reduced products of bundle gerbes {with the same surjective submersion}. Since the curvings of our bundle gerbes will be forms on the same space, they can be added together. We will use this proposition to find the connective data on the product of bundle gerbes with different surjective submersions in a moment.
\begin{proposition}\label{curving of product} \textbf{\textup{(Connective data on reduced products)}}
Consider bundle gerbes $(P, Y, M)$ and $(Q, X, M)$ with connective data $(\nabla_P, f_P)$ and $(\nabla_Q, f_Q)$ respectively. Denote the associated three-curvatures by $\omega_P$ and $\omega_Q$. Then there is canonical connective data $(\nabla_{P\otimes Q}, f_P+f_Q)$ on $(P\otimes Q, Y, M)$, with three-curvature $\omega_P+\omega_Q$.
\end{proposition}
\begin{proof}
It is straightforward to check that the connection $\nabla_{P\otimes Q}$ from Example \ref{product line bundle} is a bundle gerbe connection.	By Example \ref{product line bundle}, $\nabla_{P\otimes Q}$ has curvature $F_{\nabla_P}\otimes id_Q+id_P\otimes F_{\nabla_Q}$. Since $\delta$ is a homomorphism it follows that $f_P+f_Q$ is a curving for $\nabla_{P\otimes Q}$. The claim about the three-curvature follows using that $\pi^*$ is a homomorphism.
\end{proof}

\begin{proposition}\label{connectivestructureofproductbg} \textbf{\textup{(Connective data on products)}}
Consider bundle gerbes $(P, Y, M)$ and $(Q, X, M)$ with connective data $(\nabla_P, f_P)$ and $(\nabla_Q, f_Q)$ respectively. Denote the associated three-curvatures by $\omega_P$ and $\omega_Q$.  Let $\pi_Y: X\times_MY\to Y$ and $\pi_X: X\times_MY\to X$ be the canonical projections. Then there is canonical connective data $(\nabla_{P\otimes Q}, \pi_Y^*f_P+\pi_X^*f_Q)$ on $(P, Y, M)\otimes (Q, X, M)$, with three-curvature $\omega_P+\omega_Q$.
\end{proposition}
\begin{proof}
Recall that $(P, Y, M)\otimes (Q, X, M) = \pi_Y^{-1}(P, Y, M)\otimes_{\text{red}} \pi_X^{-1}(Q, X, M)$ (Remark \ref{mug}). By Proposition \ref{curving of pullback}, $\pi_Y^{-1}(P, Y, M)$ has connective data $((\pi_Y^{[2]})^* \nabla_P, \pi_Y^* f_P)$ and similarly for $\pi_X^{-1}(Q, X, M)$. Let \begin{align*}\nabla_{P\otimes Q}:= ({\pi_Y^{[2]}})^* \nabla_P\otimes id_Q+ id_P\otimes ({\pi_X^{[2]}})^*\nabla_Q\end{align*} where we abuse notation and let $id_P, id_Q$ denote the identity maps on $({\pi_Y^{[2]}})^{-1}(P)$ and $({\pi_X^{[2]}})^{-1}Q$, respectively. By Proposition \ref{curving of product}, $\nabla_{P\otimes Q}$ is a bundle gerbe connection on the reduced product of these pullback bundle gerbes with the claimed curving and three-curvature. 
\end{proof}
In the remainder of this section, we study the connective data on trivial and stably isomorphic bundle gerbes. We will be reminded of the following proposition in Section \ref{section holonomy} when we define the \textit{holonomy} of a bundle gerbe.
\begin{proposition}\label{bgconnectioninducedbylinebundleconnection}\textbf{\textup{(Trivial connective data)}}
Let $R\to Y$ be a line bundle and consider a trivial bundle gerbe	$(\delta(R), Y, M, \pi)$. If $\nabla_R$ is a line bundle connection on $R\to Y$, then $(\delta(\nabla_R), F_{\nabla_R})$ is connective data on $(\delta(R), Y, M, \pi)$ with trivial three-curvature. \end{proposition}
\begin{proof} 
The proof that $\delta(\nabla_R)$ is a bundle gerbe connection is a trivial computation using equation (\ref{bundlegerbeconnection}). Since $\delta$ is defined by pullbacks, $\delta(F_{\nabla_R}) = F_{\delta(\nabla_R)}$ and $F_{\nabla_R}$ is a curving of $\delta(\nabla_R)$. Finally, the pullback of the three-curvature by $\pi$ is equal to $d F_{\nabla_R} = 0$. By injectivity of $\pi^*$ (Proposition \ref{deltamap}), the three-curvature is trivial. 
\end{proof} 

\begin{definition}\label{definition of trivial connective structure}
	Consider a trivial bundle gerbe $(P, Y, M)\cong (\delta(R), Y, M)$ for some line bundle $R\to Y$. We say $(P, Y, M)$ has \textit{trivial connective data} if its connective data arises as $(\delta(\nabla_R), F_{\nabla_R})$ for some connection $\nabla_R$ on $R\to Y$ with curvature $F_{\nabla_R}$. 
\end{definition}
The next proposition will be crucial for defining bundle gerbe holonomy.
\begin{proposition}\label{connectionsontrivialbundlegerbes} Consider the trivial bundle gerbe $(\delta(R), Y, M)$ with bundle gerbe connection $\nabla$. There exists a line bundle connection $\nabla_R$ on $R\to Y$ such that $\delta(\nabla_R) = \nabla$. 
\end{proposition}
\begin{proof}
	Let $\nabla_R$ be a line bundle connection on $R\to Y$. Recall, by Proposition \ref{bgconnectioninducedbylinebundleconnection}, that $\delta(\nabla_R)$ is a bundle gerbe connection on $(\delta(R), Y, M)$. By Proposition \ref{differenceofbundlegerbeconnections}, there exists $\alpha \in \Omega^1(Y^{[2]})$ such that $$\delta(\nabla_R) - \nabla = \alpha $$ and $\delta(\alpha) = 0$. By exactness of the fundamental complex, there exists $\alpha' \in \Omega^1(Y)$ with $\alpha = \delta(\alpha')$. So $\delta(\nabla_R -\alpha') = \nabla$. Replacing $\nabla_R$ with $\nabla_R - \alpha'$ we obtain the result. 
\end{proof}
\begin{remark}
Upon presenting Proposition \ref{connectionsontrivialbundlegerbes}, it is tempting to say that trivial bundle gerbes admit trivial connective data. However, $\nabla = \delta(\nabla_R)$ does not imply the curving $f$ of $\nabla$ equals $F_{\nabla_R}$ (since there is not, in general, a unique choice of curving).
\end{remark}
It is natural to ask how the connective data of stably isomorphic bundle gerbes are related. We will explore this in greater depth once we have studied Deligne cohomology. For now, we present the following proposition, thereby concluding this section.
\begin{proposition}\label{outofnamesforthings} 
	Let $(P, Y, M)$ and $(Q, X, M)$ be stably isomorphic bundle gerbes with connective data $(\nabla_P, f_P)$ and $(\nabla_Q, f_Q)$ respectively. Denote the associated three-curvatures by $\omega_P$ and $\omega_Q$. Let $R\to X\times_MY$ be a trivialising line bundle with connection $\nabla_R$. Suppose the stable isomorphism preserves connections, i.e. $\nabla_P + \nabla_Q^* = \delta(\nabla_R)$. Then there exists $\gamma \in \Omega^2(M)$ such that \begin{align}\label{caramel1} f_P -f_Q = F_{\nabla_R}+\pi^* \gamma\end{align} and \begin{align}\label{caramel2}\omega_P - \omega_Q= d\gamma\end{align} for $\pi:X\times_M Y\to M$ the natural projection.
\end{proposition}
\begin{proof}
Since $\nabla_P + \nabla_Q^* = \delta(\nabla_R)$, the two-curvatures satisfy $F_{\nabla_P}  -F_{\nabla_Q} = \delta(F_{\nabla_R})$. Note that we are abusing notation here, since really these connections and forms should be pulled back to $(X\times_MY)^{[2]}$. By the previous equation, $\delta(f_P - f_Q-F_{\nabla_R}) = 0$, which, by exactness of the fundamental complex, implies there exists $\gamma\in \Omega^2(M)$ satisfying $f_P - f_Q = F_{\nabla_R} +\pi^*\gamma$, thereby proving (\ref{caramel1}). Differentiating this equation, we find $\pi^*(\omega_P - \omega_Q) = \pi^*(d\gamma)$, so by injectivity of $\pi^*$, $\omega_P-\omega_Q = d\gamma$, proving (\ref{caramel2}).
\end{proof}

\begin{remark}
Proposition \ref{outofnamesforthings} will be a crucial component in the proof of our main result in Chapter \ref{ch:four}. Namely, by inputting the connective data of the cup product bundle gerbe and the pullback of the basic bundle gerbe by the Weyl map into equations (\ref{caramel1}) and (\ref{caramel2}), we will be able to find a trivialising line bundle $R$ for our bundle gerbes, and hence prove they are stably isomorphic.
\end{remark}

\section{The Dixmier-Douady class}\label{section dd class} 
\indent \indent  As we have mentioned, bundle gerbes are classified by degree-three cohomology with integer coefficients. We will now formalise this claim. There are many parallels between this section and Section \ref{section the first chern class} on the first Chern class of a line bundle. Just as line bundles are determined by $1$-cocycles $g_{\alpha \beta}\in H^1(M, \underline{U(1)})$, we will show that there exist $2$-cocycles $g_{\alpha \beta \gamma} \in H^2(M, \underline{U(1)})$ which determine a bundle gerbe (up to stable isomorphism). We call the cohomology class of $g_{\alpha \beta \gamma}$ the \textit{Dixmier-Douady class}, which is in bijective correspondence with {stable} isomorphism classes of bundle gerbes. \\
 \indent After constructing the Dixmier-Douady class of a bundle gerbe, we will present some standard results, leading us to the main classification result in this section, Theorem \ref{main prop of this section}. We will conclude our study of the Dixmier-Douady class with some remarks provided for general interest. Let us begin by constructing the cocycles $g_{\alpha \beta \gamma}$ associated to a bundle gerbe. We follow the construction in \cite{INTROTOBG}.  
\begin{example} \label{DDconstruction} Consider a bundle gerbe $(P, Y, M)$ and local sections $s_\alpha:U_\alpha\to Y$ defined on a good open cover $\mathcal{U} = \{U_\alpha\}_{\alpha \in I}$ of $M$ (i.e. finite intersections of sets in $\mathcal{U}$ are contractible). Then $$(s_\alpha, s_\beta): U_\alpha\cap U_\beta \to Y^{[2]}$$ is a section of $Y^{[2]}\to M$ over $U_\alpha\cap U_\beta$. Define $P_{\alpha \beta} := (s_\alpha, s_\beta)^{-1}P$, a line bundle over $U_\alpha\cap U_\beta$. Choose a unit section $\sigma_{\alpha \beta}$ of $P_{\alpha \beta}$, i.e. $||\sigma_{\alpha \beta}(x)||^2=1$ (Definition \ref{unitsection}). There exists $g_{\alpha \beta \gamma}: U_\alpha \cap U_\beta \cap U_\gamma \to U(1)$ such that, over $U_\alpha \cap U_\beta \cap U_\gamma$, \begin{align} \label{shorts} m(\sigma_{\alpha \beta}(x), \sigma_{\beta \gamma}(x)) = g_{\alpha \beta \gamma}(x) \sigma_{\alpha \gamma}(x). \end{align} It can be verified that $g_{\alpha \beta \gamma}$ is a \v{C}ech $2$-cocycle that is independent of all choices made above and $g_{\beta \gamma \delta}\, g_{\alpha \gamma \delta}^{-1}\, g_{\alpha \beta \delta}\, g_{\alpha \beta \gamma}^{-1}=1$. The resulting class $$[g_{\alpha \beta \gamma}] \in H^2(M, U(1))$$ defines the {characteristic class} of the bundle gerbe $(P, Y, M)$.\end{example} 
We saw in Section \ref{section the first chern class} that the short exact sequence of sheaves $$0\xrightarrow{} \underline{\Z} \xrightarrow{} \underline{\R}\xrightarrow{} \underline{U(1)}\xrightarrow{} 0$$ induces a long exact sequence in cohomology $$\cdots\xrightarrow{} 0 = H^2(M, \underline{\R}) \xrightarrow{}H^2(M, \underline{U(1)})\xrightarrow{} H^3(M, \Z) \xrightarrow{} 0=H^3(M, \underline{\R}) \xrightarrow{} \cdots,$$ so $H^2(M, \underline{U(1)})\cong H^3(M, \Z)$. With this in mind, together with the above example, we can define the Dixmier-Douady class of a bundle gerbe.
\begin{definition}[\cite{INTROTOBG}]
 The \textit{Dixmier-Douady class} of a bundle gerbe $(P, Y, M)$, denoted $DD(P,Y, M)$ or $DD(P, Y)$, is the image of the characteristic class $[g_{\alpha \beta \gamma}]\in H^2(M, U(1))$ constructed in Example \ref{DDconstruction} under the isomorphism $H^2(M, U(1))\xrightarrow{\sim} H^3(M, \Z)$.
\end{definition} 
\begin{remark}
The mathematicians J. Dixmier and A. Douady, after which the Dixmier-Douady class is named, never worked with bundle gerbes directly. Their namesake here can be attributed to Brylinski's use of the term `Dixmier-Douady sheaves of groupoids' to describe the gerbes in \cite{brylinski}, which generalised the work of Dixmier and Douady in \cite{dixmierdouady}.
\end{remark}
For a simple example let us compute the Dixmier-Douady class of the trivial bundle gerbe defined by $R$ (Example \ref{trivialbundlegerbe}). As one would hope, its Dixmier-Douady class is zero. 
\begin{example}
	Consider a trivial bundle gerbe $(\delta(R), Y, M)$ for $R\to Y$ a line bundle and $Y\to M$ a surjective submersion. Let $\{U_\alpha\}_{\alpha \in I}$ be a good open cover of $M$ with local sections $s_\alpha:U_\alpha \to Y$. Choose local sections $\nu_\alpha$ of $s_\alpha^{-1}(R)\to U_\alpha$. Then for $m\in U_\alpha \cap U_\beta$, $$(s_\alpha, s_\beta)^{-1}\delta(R)_{m} = R_{s_\beta(m)} \otimes R^*_{s_\alpha(m)}, $$ and $\sigma_{\alpha \beta} = \nu_\beta\otimes \nu^*_\alpha$ defines a section of $\delta(R)_{\alpha \beta}$. It follows immediately from equation (\ref{shorts}) that $g_{\alpha \beta \gamma} = 1$ and $DD(\delta(R), Y) = 0$. 
\end{example} 
It is a standard fact that the Dixmier-Douady class of the dual, product, and pullback bundle gerbes obeys similar relations to those in Proposition \ref{proposition chern class operations} for the first Chern class.
\begin{proposition}[{\cite[p.$\,$8]{INTROTOBG}}] \label{DDclass}
	Let $(P, Y, M)$ and $(Q, X, M)$ be bundle gerbes and $f:N\to M$ be a smooth map. The Dixmier-Douady class satisfies \begin{enumerate}[(1),font=\upshape]
		\item $DD((P, Y)^*) = - DD(P, Y);$
		\item $DD((P, Y)\otimes (Q, X)) = DD(P, Y)+ DD(Q, X);$
		\item $f^*(DD(P, Y)) = DD(f^{*}(P, Y)),$ where $f^*:H^3(M, \Z)\to H^2(N, \Z)$ is the induced map on cohomology. 
	\end{enumerate}
\end{proposition}  
As we have seen, trivial bundle gerbes have trivial Dixmier-Douady class. The converse of this statement is also true. We present this result without proof in an effort to streamline this discussion, and refer the reader to \cite{INTROTOBG} for details. It will then be used immediately to prove Proposition \ref{time}.
\begin{proposition}[{\cite[Proposition 5.2]{INTROTOBG}}]\label{triviffclassiszero} A bundle gerbe $(P, Y, M)$ is trivial if and only if $DD(P, Y)=0$.
\end{proposition}  
\begin{proposition}\label{time} 
	Two bundle gerbes $(P, Y, M)$ and $(Q, X, M)$ are stably isomorphic if and only if $DD(P, Y) = DD(X, Q)$. 
\end{proposition} 
\begin{proof}
By definition, $(P, Y) \cong_{\text{stab}} (Q, X)$ if and only if $(P, Y)\otimes (Q,X)^*$ is trivial. By Propositions \ref{DDclass} and \ref{triviffclassiszero},  this is true if and only if $DD(P, Y) =DD(Q, X)$. 
\end{proof}
\indent The reader may recall from Section \ref{section morphisms} that we regard stable isomorphisms as the `correct' notion of equivalence for bundle gerbes. We are now in a position to formalise this claim in the following theorem. 
\begin{theorem}[{\cite[Proposition 5.4]{INTROTOBG}}]\label{main prop of this section}
	The Dixmier-Douady class defines a bijection between stable isomorphism classes of bundle gerbes on $M$ and $H^3(M, \Z)$. 
\end{theorem}
\begin{proof}[Sketch of Proof] We provide an outline of its proof and refer the reader to \cite{INTROTOBG} for details. In Example \ref{DDconstruction}, we saw how to construct a cohomology class in $H^3(M, \Z)$ from a bundle gerbe. It is a consequence of the \textit{Hitchin-Chatterjee construction} (Example \ref{hitchinchatconstruction}) that any bundle gerbe can be constructed in this way (see \cite{chatthesis, INTROTOBG} for details). Therefore the assignment of a stable isomorphism class of a bundle gerbe to a class in $H^3(M, \Z)$ is surjective. This assignment is well-defined and injective by Proposition \ref{time}. \end{proof} 

The results in the remainder of this section are provided for interest, and will not be needed for our research. The reader will recall that, given a line bundle $L\to M$ with connection $\nabla$ and curvature $F_\nabla$, the image of $[\tfrac{1}{2\pi i}F_\nabla]$ under the \v{C}ech-de Rham isomorphism $H^2_{\text{dR}}(M)\xrightarrow{\sim} H^2(M, \R)$ equals the real part of the first Chern class, $r(c_1(L))$ (Proposition \ref{relationship between class and curvature}). We have an analogous result for bundle gerbes. Consider the map $r:H^3(M, \Z)\to H^3(M, \R)$ defined similarly to that in Section \ref{section the first chern class}. We call $r(DD(P, Y))$ the \textit{real} Dixmier-Douady class. Then we have the following.
\begin{proposition}\label{curvatureisrealpartofclass} 
	Let $(P, Y, M)$ be a bundle gerbe with connective data and three-curvature $\omega\in \Omega^3(M)$. The image of $\left[\frac{1}{2\pi i}\omega \right]$ under the \v{C}ech-de Rham isomorphism $H^3_{\text{dR}}(M)\xrightarrow{\sim} H^3(M, \R)$ equals $r(DD(P, Y)) \in H^3(M, \R)$.

\end{proposition}
A brief explanation of this fact will be given in Remark \ref{remark real part of DD class} once we have discussed Deligne cohomology. We conclude this section with a brief comment on the types of bundle gerbe surjective submersions we observe in practice.  
\begin{remark}
The space $Y$ in a bundle gerbe $(P, Y, M)$ is commonly infinite dimensional. This is a consequence of the following result: if $Y\to M$ is a fibre bundle with finite dimensional,  $1$-connected fibres, then the Dixmier-Douady class of $(P, Y, M)$ is torsion \cite{INTROTOBG}. A bundle gerbe with torsion Dixmier-Douady class is said to be a \textit{torsion bundle gerbe.} Of course, by definition of $r: H^3(M, \Z)\to H^3(M, \R)$, torsion bundle gerbes have $r(DD(P, Y, M)) = 0$. Therefore by Proposition \ref{curvatureisrealpartofclass}, the three-curvature $\omega$ of a torsion bundle gerbe is exact. In our examples, the fibres of $Y$ are disconnected so this problem does not arise. 
%
\end{remark}

\section{Holonomy} \label{section holonomy}
\indent \indent We dedicate the penultimate section of this chapter to \textit{bundle gerbe holonomy}. In the the next section, we will see that bundle gerbe holonomy is a useful tool in determining when two bundle gerbes are equivalent as bundle gerbes \textit{with connective data}. Like line bundle holonomy, bundle gerbe holonomy is defined in terms of the exponential of the integral of a form. The key difference here is that line bundle holonomy is defined in terms of the integral of a $1$-form over a curve, whereas bundle gerbe holonomy is defined in terms of the integral of a $2$-form over a surface. While bundle gerbe holonomy is a useful tool in physics, and particularly the Wess-Zumino-Witten models (Remark \ref{WZWremark}), the higher dimensions make bundle gerbe holonomy more difficult to motivate geometrically than line bundle holonomy (which was based on parallel transport).\\
\indent The format of this section proceeds as follows. First, we define bundle gerbe holonomy and show it is well-defined. We will then compute the holonomy of the dual, product, and trivial bundle gerbes. These propositions will be used to prove the main result in this section, Proposition \ref{holonomyequalforstabisobg}, which states that stably isomorphic bundle gerbes have the same holonomy. We opt to prove each of these results in detail as they are commonly neglected in the literature. All of our proofs use similar techniques, so the reader could read the first few and skim the rest, if they wish. Let us begin with the definition of bundle gerbe holonomy, which requires a brief construction.\\
\indent Consider a bundle gerbe $(P, Y, \Sigma, \pi)$ for $\Sigma$ a surface. Suppose $(\nabla_P, f)$ is connective data on $(P, Y, \Sigma)$. It is a standard fact that $H^3(\Sigma, \Z) = 0$, so by Proposition \ref{triviffclassiszero}, $(P, Y, \Sigma)$ is trivial. By Proposition \ref{connectionsontrivialbundlegerbes}, there exists a line bundle $R\to Y$ with connection $\nabla_R$ such that $\delta(R) = P$, and $\delta(\nabla_R) = \nabla_P$. Now consider the difference $$f-F_{\nabla_R} \in \Omega^2(Y).$$ Since $\delta(f-F_{\nabla_R}) = F_{\nabla_P} - F_{\delta(\nabla_R)} = 0$, exactness of the fundamental complex implies there exists a unique $\mu\in \Omega^2(\Sigma)$ satisfying $$f - F_{\nabla_R} = \pi^*(\mu). $$ With this construction of the $2$-form $\mu$, we can define bundle gerbe holonomy as follows.
\begin{definition}\label{definition of holonomy on bundle gerbe} 
Consider a bundle gerbe $(P, Y)$ over a surface $\Sigma$ with connective data $(\nabla_P, f)$. Let $\mu\in \Omega^2(\Sigma)$ be the $2$-form constructed above. Then the \textit{holonomy of $(\nabla_P, f)$ over $\Sigma$}, denoted $\text{hol}((\nabla_P, f), \Sigma)$, is defined as $$\text{hol}((\nabla_P, f), \Sigma) := \text{exp}\left(\int_{\Sigma} \mu \right). $$ 
\end{definition}
\begin{remark}
We may at times use the phrase `the holonomy of the bundle gerbe', instead of `the holonomy of the connective data on the bundle gerbe', when the connective data we are working with is clear. 
\end{remark}
\begin{remark}
More generally, consider a bundle gerbe $(P, Y, M)$ with connective data $(\nabla_P, f)$. Let $\Sigma$ be a surface, and suppose there is a map $\varphi:\Sigma \to M$. We can then define $$\text{hol}( (\nabla_p, f), \Sigma) := \text{hol}(\varphi^{-1}(\nabla_P, f), \Sigma) $$ for $\varphi^{-1}(\nabla_P, f)$ the pullback of the connective data (Proposition \ref{curving of pullback}).
\end{remark}
\begin{remark}\label{alternative defn of holn} 
If there exists a three-dimensional submanifold $X\subset M$ such that the boundary $\partial X$ also a submanifold, we can define the holonomy of connective data $(\nabla_P, f_P)$ on $(P, Y, M)$ by $$\text{hol}((\nabla_P, f_P), \partial X) := \text{exp}\left(\int_X \omega \right)$$ for $\omega$ the three-curvature of $(\nabla_P, f_P)$. To see that this agrees with Definition \ref{definition of holonomy on bundle gerbe}, first note that $(P, Y, M)$ can be pulled back and trivialised over all of $X$. Then, for a trivialising line bundle $R\to Y$ of $(P, Y,X)$, there exists $\mu \in \Omega^2(X)$ satisfying $f_P - F_{R} = \pi^*\mu.$ By injectivity of $\pi^*$ and definition of $\mu$, $d\mu = \omega$ over $X$. The result then follows by Stokes' theorem. For more on this, see \cite{INTROTOBG}.

\end{remark}
\begin{remark}[{\cite[p.$\,$415]{1996bundlegerbes}}] 
	Suppose $X$ is a ball with $\partial X = S^2$ in Remark \ref{alternative defn of holn}. Then $\text{hol}((\nabla_P, f_P), S^2) \in U(1)$. In this way we see that the holonomy of a bundle gerbe generalises the holonomy of a line bundle. 
\end{remark}
\begin{remark}\label{WZWremark}
One interesting application of bundle gerbe holonomy arises in the context of Wess-Zumino-Witten (WZW) models. In \cite{witten1984}, Witten was searching for a conformally invariant action relating to his work on linear sigma models, leading him to define the \textit{WZW-term} as follows. Consider a smooth map $g:\Sigma \to G$ for $\Sigma$ a surface and $G$ a compact, simple and simply connected Lie group. Let $X$ be a three-manifold with $\partial X = \Sigma$ and extend $g$ to $\hat{g}:X\to G$. The $\text{WZW}$-term is defined by $$\text{WZW}(g) = \text{exp} \int_X \hat{g}^*(\omega) $$ for $\omega$ the basic $3$-form on $G$ \cite{53UNITARY}. It is not difficult to show that this is well-defined under our choice of extension $\hat{g}$, see, for example, \cite{INTROTOBG}. So $\text{WZW}(g)\in U(1)$ is a well-defined invariant associated to the map $g$.\\
\indent In \cite{1996bundlegerbes}, Murray pointed out that the $\text{WZW}$-term can be described in terms of bundle gerbe holonomy, thereby removing the topological restrictions on $X$. To do this, one can choose connective data $(\nabla, f)$ on the basic bundle gerbe so that the three-curvature of this connective data is the basic form on $G$. It then follows by Remark \ref{alternative defn of holn} that $$\text{WZW}(g) = \text{hol}(g^*(\nabla, f), \Sigma).$$ This removes the condition of simple-connectedness of $G$ needed in Witten's original definition of $\text{WZW}$, since the basic bundle gerbe can be defined even if $G$ is not simply connected (see \cite{bbgovernonsimpleconnectedgroups}). Moreover, this definition is advantageous in that the local formulae used to compute the holonomy of a bundle gerbe (see, for example, \cite{INTROTOBG}) is made available to compute the $\text{WZW}$-term.
\end{remark}

We next prove bundle gerbe holonomy is well-defined. Although we will not explicitly reference this result later, we choose to include its proof here as it is commonly left out of the student literature. Throughout the rest of this section let $\Sigma$ be a surface.
\begin{proposition}\textbf{\textup{(Holonomy is well-defined)}}
	Let $(P, Y, \Sigma, \pi)$ be a bundle gerbe with connective data $(\nabla_P, f)$. Then $\textup{hol}((\nabla_P, f), \Sigma)$ is well-defined, i.e. independent of the choice of trivialising line bundle $R$ and connection $\nabla_R$. 
\end{proposition}
\begin{proof}
		 We first show independence of $\nabla_R$. Suppose $\widetilde{\nabla}_R$ is another connection on $R$ satisfying $\delta(\widetilde{\nabla}_R) = \nabla_P$. By Remark \ref{proposition two line bundle connections differ}, the difference $\widetilde{\nabla}_R - \nabla_R = \alpha$ for some $\alpha \in \Omega^1(Y)$. Then $\delta(\widetilde{\nabla}_R - \nabla_R) =\delta(\alpha)=0$, and by exactness of the fundamental complex, $\alpha = \pi^*(\alpha')$ for some $\alpha'\in \Omega^1(\Sigma)$.  Therefore $\widetilde{\nabla}_R := \nabla_R + \pi^*(\alpha')$, and $F_{\widetilde{\nabla}_R} = F_{\nabla_R} + d\pi^*(\alpha')$. So $f-F_{\widetilde{\nabla}_R} = \pi^*(\mu_R + d\alpha')$ and $$ \exp\left(\int_\Sigma \mu_R + d\alpha' \right) = \exp\left(\int_\Sigma \mu_R  \right) $$ since $\int_\Sigma d\alpha'= 0$ by Stokes' Theorem.
	
 It remains to verify that the holonomy is independent of the choice of $R$. Let $R$ and $R'$ be two trivialising line bundles of $P$. Let $\nabla_{R'}$ be a connection on $R'$ and $\mu_{R'}$ be the unique $2$-form on $\Sigma$ satisfying $f-F_{\nabla_R'} =\pi^*(\mu_{R'})$.  By Proposition \ref{difference of trivialisations}, there exists a line bundle $Q\to \Sigma$ for which $R = R'\otimes \pi^{-1}(Q)$. Choose a connection $\nabla_Q$ on $Q$. Since we know the holonomy is independent of the choice of connection, we can define a connection $\nabla_R$ on $R$ by $\nabla_R:= \nabla_{R'} +\pi^*\nabla_Q$. Then $F_{\nabla_R} = F_{\nabla_{R'}} + \pi^{*}F_{\nabla_Q}$, so $$f - F_{\nabla_R} = (f- F_{\nabla_{R'}}) - \pi^{*}F_{\nabla_Q} = \pi^*(\mu_{R'} - F_{\nabla_Q}). $$ Therefore the unique $2$-form $\mu_R$ satisfying $f-F_{\nabla_R} = \pi^*(\mu_R)$ is $\mu_R := \mu_{R'} - F_{\nabla_Q}$. It is a standard fact that the integral of the two-curvature of a line bundle over a closed surface takes values in $2\pi i \Z$. Therefore $$\text{exp}\left(\int_\Sigma \mu_R \right) = \text{exp}\left(\int_\Sigma \mu_{R'}\right)\text{exp}\left(\int_\Sigma F_Q\right) = \text{exp}\left(\int_\Sigma \mu_{R'} \right). $$
We conclude the holonomy is well-defined. 
\end{proof}
Before studying the holonomy of the dual, product, and trivial bundle gerbes, we present one more proposition, which will be used later in this work. 
\begin{proposition}\label{proposition holonomy for different curvings}\textbf{\textup{(Holonomy under change of curving)}}
	Let $(P, Y, \Sigma, \pi)$ be a bundle gerbe with connective data $(\nabla, f)$. Let $\varphi \in \Omega^2(\Sigma)$. Then $f + \pi^*\varphi$ is also a curving for $\nabla$ and $$\textup{hol}((\nabla, f + \pi^* \varphi ), \Sigma) = \textup{exp}\left(\int_\Sigma \varphi \right)\textup{hol}((\nabla, f), \Sigma).$$
\end{proposition}
\begin{proof}
	Recall that $f$ is a curving for $\nabla$ if $\delta(f) = F_{\nabla}$. Since $\ker \delta = \text{im} \pi^*$, $\delta(f+\pi^*\varphi) = F_{\nabla}$, so $f+\pi^*\varphi$ is indeed a curving for $\nabla$. Let $R$ be a trivialisation for $(P, Y, \Sigma)$ with connection $\nabla_R$.  Define $\mu \in \Omega^2(Y)$ by $f - F_{\nabla_R} = \pi^* \mu$. Then $f + \pi^*\varphi - F_{\nabla_R} = \pi^*(\mu + \varphi)$, so $$\text{hol}((\nabla, f+\pi^*\varphi), \Sigma) = \text{exp}\left(\int_\Sigma \mu  + \varphi \right) = \textup{exp}\left(\int_\Sigma \varphi \right)\textup{hol}((\nabla, f), \Sigma). $$
\end{proof}
We now discuss the holonomy of the trivial, dual, and product bundle gerbes. Although we did not mention it in Chapter \ref{ch:intro}, there are analogous results for line bundle holonomy (to see this, just consider the curvature of the trivial, dual, and product line bundle connections). We begin with the holonomy of the trivial bundle gerbe, followed by the holonomy of the dual and product bundle gerbes.
\begin{proposition}\label{hol of trivial bundle gerbe}\textbf{\textup{(Holonomy of a trivial bundle gerbe)}}
	Consider the trivial bundle gerbe $(\delta(R), Y, \Sigma)$ with trivial connective data $(\delta(\nabla_R), F_{\nabla_R})$. Then $$\textup{hol}((\delta(\nabla_R), F_{\nabla_R}), \Sigma) = 1. $$
\end{proposition}
\begin{proof}
	Since the curving of the trivial bundle gerbe is equal to $F_{\nabla_R}$, injectivity of $\pi^*$ implies $\mu = 0$ and the result follows. 
\end{proof}
\begin{proposition}\textbf{\textup{(Holonomy of duals)}}
	Let $(P, Y, \Sigma, \pi)$ be a bundle gerbe with connective data $(\nabla_P, f_P)$. Then the holonomy of the connective data $(\nabla_P^*, -f_P)$ on the dual bundle gerbe $(P, Y, \Sigma, \pi)^*$ from \textup{Proposition \ref{connectivestructureondualbg}} satisfies $$\textup{hol}((\nabla_P^*, -f_P), \Sigma) = \textup{hol}((\nabla_P, f_P), \Sigma)^{-1}. $$
\end{proposition}
\begin{proof}
Let $R$ be a trivialising line bundle of $(P, Y, \Sigma)$ with connection $\nabla_R$ and curvature $F_{\nabla_R}$. Then $R^*$ is a trivialising line bundle of $(P, Y, \Sigma)^*$ with connection $\nabla_R^*$ and curvature $-F_{\nabla_R}$. If $\mu\in \Omega^2(\Sigma)$ satisfies $f_P - F_{\nabla_R} = \pi^*(\mu)$, then $-f_P -F_{\nabla_R^*} = -f_P + F_{\nabla_R} = \pi^*(-\mu)$, and $$\textup{hol}((\nabla_P^*, -f_P), \Sigma) = \exp\left(- \int_\Sigma \mu \right) = \textup{hol}((\nabla_P, f_P), \Sigma)^{-1}.$$
\end{proof}

\begin{proposition}\label{productholonomy} \textbf{\textup{(Holonomy of products)}}
	Let $(P, Y, \Sigma, \pi_1)$ and $(Q, X, \Sigma, \pi_2)$ be bundle gerbes with connective data $(\nabla_P, f_P)$ and $(\nabla_Q, f_Q)$ respectively. Consider the projections $\pi_X: X\times_\Sigma Y\to X$ and $\pi_Y:X\times_\Sigma Y\to Y$. Then the holonomy of the connective data $(\nabla_{P\otimes Q}, \pi_X^*f_P+\pi_Y^*f_Q)$ on the product bundle gerbe $(P\otimes Q, X\times_\Sigma Y, \Sigma, \pi)$ from \textup{Proposition \ref{connectivestructureofproductbg}} satisfies $$\textup{hol}((\nabla_{P\otimes Q}, \pi_X^*f_P+ \pi_Y^*f_Q), \Sigma) = \textup{hol}((\nabla_P, f_P), \Sigma) \cdot \textup{hol}((\nabla_Q, f_Q), \Sigma). $$
\end{proposition}
\begin{proof}
Let $R$ and $R'$ be trivialising line bundles for $(P, Y)$ and $(Q, Y)$ respectively, with connections $\nabla_R$ and $\nabla_{R'}$. So $R\otimes R'$ is a trivialisation of the product bundle gerbe, and $\nabla_R, \nabla_{R'}$ induce a connection $\nabla_{R\otimes R'}$ on $R\otimes R'$ with curvature equal to $\pi_X^*F_{\nabla_R} + \pi_Y^*F_{\nabla_{R'}}$. Define $\mu_1, \mu_2 \in \Omega^2(\Sigma)$ by $f_P - F_{\nabla_R} = \pi_1^*(\mu_1)$ and $f_Q - F_{\nabla_{R'}} = \pi_2^*(\mu_2)$. Noting that $\pi_1\circ \pi_X = \pi_2\circ \pi_Y = \pi$, we have $$(\pi_X^*f_P + \pi_Y^*f_Q) - (\pi_X^*F_{\nabla_R} + \pi_Y^*F_{\nabla_{R'}}) = \pi_X^*\pi_1^*(\mu_1) +\pi_Y^*\pi_2^*(\mu_2) = \pi^*(\mu_1+\mu_2).$$ Therefore $$ \textup{hol}( (\nabla_{P\otimes Q}, \pi_X^*f_P+\pi_Y^*f_Q), \Sigma) = \exp\left(\int_\Sigma \mu_1+\mu_2 \right) =   \textup{hol}((\nabla_P, f_P), \Sigma) \cdot \textup{hol}((\nabla_Q, f_Q), \Sigma). $$ 
\end{proof}

For our final proposition, we require the notion of a stable isomorphism of bundle gerbes \textit{with connective data}, which we call a $D$-stable isomorphism for reasons that will become clear in the next section. Recall that a bundle gerbe is said to have trivial connection and curving if its connective data arises as $(\delta(\nabla_R), F_{\nabla_R})$ for some line bundle connection $\nabla_R$ on $R\to Y$ (Definition \ref{definition of trivial connective structure}).
\begin{definition}  A \textit{$D$-stable isomorphism} of bundle gerbes $(P, Y, M)$ and $(Q, X, M)$ with connective data is a stable isomorphism that preserves connections and curvings (where the trivial bundle gerbe $P^*\otimes Q$ is assumed to have trivial connection and curving). If $(P, Y, M)$ and $(Q, X, M)$ are $D$-stably isomorphic we write $(P, Y, M)\cong_{D\text{-stab}} (Q, X, M)$.\end{definition}
With this definition, we can now prove that $D$-stably isomorphic bundle gerbes have equal holonomy. The proof of this result will be a neat application of the previous propositions in this section. 

\begin{proposition}\label{holonomyequalforstabisobg} \textbf{\textup{(Holonomy of $D$-stably isomorphic bundle gerbes)}}
 	Let $(P, Y, \Sigma)$ and $(Q, X, \Sigma)$ be bundle gerbes with connective data $(\nabla_P, f_P)$ and $(\nabla_Q, f_Q)$ respectively. If $(P, Y, \Sigma)\cong_{\textup{D-stab}} (Q, X, \Sigma)$, then $\textup{hol}((\nabla_P, f_P), \Sigma) = \textup{hol}((\nabla_Q, f_Q), \Sigma)$. 
\end{proposition}
\begin{proof}
	Now, $(P\otimes Q^*, X\times_\Sigma Y) \cong (\delta(R), X\times_\Sigma Y)$, where the latter bundle gerbe is endowed with trivial connective data. Since this is a $D$-stable isomorphism, the connective data is equal, i.e. $(\nabla_{P\otimes Q^*}, \pi_X^*f_P - \pi_Y^*f_Q) = (\delta(\nabla_R), F_{\nabla_R})$, where $\pi_X$ and $\pi_Y$ are the projections from earlier. Hence the respective holonomies should be equal. By the previous propositions, this implies $$\textup{hol}((\nabla_P, f_P), \Sigma) \cdot \textup{hol}((\nabla_Q, f_Q), \Sigma)^{-1} = 1, $$ and the result follows. 
\end{proof}
\begin{remark}
The converse of the above result is most useful in practice. That is, if two bundle gerbes have different holonomies, they cannot be $D$-stably isomorphic. This is how we will show the cup product bundle gerbe and pullback of the basic bundle gerbe are not $D$-stably isomorphic in Chapter \ref{ch:four}.
\end{remark}

\section{Deligne cohomology}\label{deligne cohomology} 
\indent \indent We conclude this chapter with a powerful classification result for bundle gerbes. The reader may recall that the construction of the Dixmier-Douady class in Example \ref{DDconstruction} did not assume the bundle gerbe to have any connective data. This renders our classification result in Theorem \ref{main prop of this section} somewhat incomplete. In this section, we aim to understand how Deligne cohomology can be used to classify stable isomorphism classes of bundle gerbes \textit{with connective data}. Formally, the result we will prove is the following. \begin{theorem}[{\cite[Theorem 4.1]{stableiso}}] \label{delignetheorem} The Deligne class defines a bijection between $D$-stable isomorphism classes of bundle gerbes with connective data on $M$ and $H^3(M, \Z(3)_D)$.
%
\end{theorem}
For our purposes, it will suffice to define the Deligne cohomology groups in terms of isomorphic \v{C}ech hypercohomology groups. After doing so, we will detail how a Deligne class can be associated to any bundle gerbe with connective data. The main result in this section, Theorem \ref{delignetheorem}, will then be proved, and a summary of the classification results for bundle gerbes will be provided in Proposition \ref{bgclassification}. We conclude this section by remarking on several interesting consequences of Theorem \ref{delignetheorem}, including the unexpected relationship between the holonomy of a bundle gerbe and its Deligne class. A foundational understanding of sheaves, double complexes of sheaves, and \v{C}ech hypercohomology groups will be assumed throughout. For general information on these topics, we refer the reader to \cite{botttu, brylinski}, and for information on these topics in the context of bundle gerbes we recommend \cite{stuartthesis, stableiso}.

\begin{remark} In Murray's original 1996 publication \cite{1996bundlegerbes}, he detailed how to associate to any bundle gerbe (with connective data) its Dixmier-Douady (Deligne) class. At that time, however, stable isomorphisms had not been developed, so the bijective correspondence in Theorems \ref{main prop of this section} and \ref{delignetheorem} could not be stated. Instead, the Dixmier-Douady (Deligne) class was simply understood as the `obstruction to the bundle gerbe (with connective data) being trivial'. It was not until 2000 (when stable isomorphisms were developed) that Theorems \ref{main prop of this section} and \ref{delignetheorem} could be stated and proved formally by Murray--Stevenson in \cite{stableiso}.
	
\end{remark}

\indent We begin by defining the \v{C}ech hypercohomology groups $\check{H}^q(\mathcal{U}, \mathcal{D}^p)$ that will be used to define Deligne cohomology. Let $\underline{U(1)}_M$ be the sheaf of smooth $U(1)$-valued functions on $M$ and $\underline{\Omega}^p_M$ be the sheaf of real differential $p$-forms on $M$. We will abuse notation and write $\underline{U(1)}$ or $\underline{\Omega}^p$ when the space $M$ is understood. 
Define the complex of sheaves $\mathcal{D}^{p}$ by $$\mathcal{D}^{p} := \underline{U(1)}_M\xrightarrow{d\log} \underline{\Omega}^1_M \xrightarrow{d}\cdots \xrightarrow{d} \underline{\Omega}^{p}_M.$$ Associated to $\mathcal{D}^{p}$ is the \v{C}ech double complex depicted in Figure \ref{fig:cech1}. 
	\begin{center} 
		\begin{tikzcd}
			\vdots	  &\vdots  & &\vdots   \\
			C^2(\mathcal{U}, \underline{U(1)}) \arrow[r, "d\log"] \arrow[u, "\delta'"] & 	C^2(\mathcal{U}, \underline{\Omega}^1)  \arrow[r, "d"] \arrow[u, "\delta'"] & \cdots \arrow[r, "d"]& C^2(\mathcal{U}, \underline{\Omega}^{p})\arrow[u,"\delta'"] \\
			C^1(\mathcal{U}, \underline{U(1)})  \arrow[r, "d\log"] \arrow[u, "
			\delta'"] & C^1(\mathcal{U}, \underline{\Omega}^1) \arrow[r, "d"] \arrow[u, "\delta'"] & \cdots\arrow[r, "d"]  &C^1(\mathcal{U}, \underline{\Omega}^{p}) \arrow[u, "\delta'"] \\		C^0(\mathcal{U}, \underline{U(1)})  \arrow[r, "d\log"] \arrow[u, "\delta'"] &  C^0(\mathcal{U}, \underline{\Omega}^1) \arrow[r, "d"] \arrow[u, "\delta'"] &\cdots \arrow[r, "d"]& C^0(\mathcal{U}, \underline{\Omega}^{p}) \arrow[u, "\delta'"] 
		\end{tikzcd} 		 \captionof{figure}{\v{C}ech double complex associated to $\mathcal{D}^{p}$}
		\label{fig:cech1} \end{center}

	\text{}\\
\noindent \noindent	Here, $\mathcal{U}$ is a Leray cover of $M$ (i.e. finite intersection of sets in $\mathcal{U}$ have trivial cohomology) and the maps $\delta': C^p(\mathcal{U}, \cdot) \to C^{p+1}(\mathcal{U}, \cdot)$ are defined by $$\delta'(\alpha)_{i_0, ..., i_{p+1}} = \sum_{j=0}^{p+1}(-1)^j (\alpha_{i_0, ..., \hat{i}_j, ..., i_{p+1}})\big |_{U_{i_0, ..., i_{p+1}}}.$$ 
	With this, we can define the \v{C}ech hypercohomology groups with coefficients in $\mathcal{D}^p$.
\begin{definition}
Let $\mathcal{U}$ be a Leray cover of $M$. The \textit{$q$-th \v{C}ech hypercohomology group with coefficients in $\mathcal{D}^{p}$}, denoted $\check{{H}}^{q}(\mathcal{U}, \mathcal{D}^{p})$, is the $q$-th cohomology of the total complex of the double complex depicted in Figure \ref{fig:cech1}.
\end{definition}

To provide the general definition of Deligne cohomology groups with coefficients in $\Z(p)_D$ would be a lengthy process. Instead, we circumvent this with the following proposition. Let $\Z(p)=(2\pi \sqrt{-1})^p\cdot \Z$ and let $\Z(p)_M$ be the sheaf of locally constant functions on $M$. Define the complexes of sheaves $\Z(p)_D$ by $$\Z(p)_D := \Z(p)_M \xrightarrow{i} \underline{\Omega}^1_M \xrightarrow{d}\cdots \xrightarrow{d} \underline{\Omega}^{p-1}_M.$$ 

\begin{proposition}
Let $\mathcal{U}$ be a Leray cover of $M$. The \textit{$q$-th Deligne cohomology group with coefficients in $\Z(p)_D$}, denoted $H^{q}(M, \Z(p)_D)$, is isomorphic to the $(q-1)$-th \v{C}ech hypercohomology group with coefficients in $\mathcal{D}^{p-1}$, i.e. $H^{q}(M, \Z(p)_D)\cong \check{{H}}^{q-1}(\mathcal{U}, \mathcal{D}^{p-1}).$
\end{proposition}

\begin{proof} [Sketch of Proof] 
To streamline this discussion we will present the key elements of this proof and refer the reader to \cite{stuartthesis, brylinski, botttu} for details. The proof amounts to showing that there is a sequence of isomorphisms $$H^q(M, \Z(p)_D) \cong H^{q-1}(M, \mathcal{D}^{p-1}) \cong  \check{{H}}^{q-1}(\mathcal{U}, \mathcal{D}^{p-1})$$ for $q\geq 1$ and $\mathcal{U}$ a Leray cover of $M$. The first of these isomorphisms of Deligne cohomology and sheaf hypercohomology groups is induced from a quasi-isomorphism of sheaves, as shown in \cite{stuartthesis}. It is then shown in \cite{brylinski} that the latter groups are canonically isomorphic to the \v{C}ech hypercohomology groups $\check{{H}}^{q-1}(\mathcal{U}, \mathcal{D}^{p-1})$. 
\end{proof}
Throughout the remainder of this section we will take $H^{q}(M, \Z(p)_D):= \check{{H}}^{q-1}(\mathcal{U}, \mathcal{D}^{p-1})$ to be the definition of the Deligne cohomology groups. We next aim to understand what a general element in $H^3(M, \Z(3)_D)$ looks like.

\begin{proposition}
The triple $(h_{\alpha \beta \gamma}, \theta_{\alpha \beta}, \nu_\alpha)\in C^2(\mathcal{U}, \underline{U(1)})\oplus C^1(\mathcal{U}, \underline{\Omega}^1) \oplus C^0(\mathcal{U}, \underline{\Omega}^2)$ is a  cocycle, and hence representative of a class in $H^3(M, \Z(3)_D)$, if and only if \begin{align}
h_{\beta \gamma \delta} \ h^{-1}_{\alpha \gamma \delta} \ h_{\alpha \beta \delta} \ h^{-1}_{\alpha \beta \gamma} &= 1 \label{eqn11} \\
\theta_{\beta \gamma} - \theta_{\alpha \gamma}  + \theta_{\alpha \beta} &= - d\log(h_{\alpha \beta \gamma})\label{eqn12}  \\
\nu_\beta -  \nu_\alpha  &= d\theta_{\alpha \beta}\label{eqn13} 
\end{align} for $\mathcal{U} = \{U_\alpha\}_{\alpha \in I}$ a Leray cover of $M$. 
\end{proposition}
\begin{proof}
 Consider a segment of the total complex of the double complex in Figure \ref{fig:cech1} with $p = 2$. $$\cdots \xrightarrow{} C^2(\mathcal{U}, \underline{U(1)})\oplus C^1(\mathcal{U}, \underline{\Omega}^1) \oplus C^0(\mathcal{U}, \underline{\Omega}^2) \xrightarrow{D} C^3(\mathcal{U},\underline{U(1)})\oplus C^2(\mathcal{U}, \underline{\Omega}^1) \oplus C^1(\mathcal{U}, \underline{\Omega}^2) \xrightarrow{} \cdots $$ Here $D$ is defined by 
$$D(h_{\alpha \beta \gamma}, \theta_{\alpha \beta}, \nu_\alpha) = (\delta'(h_{\alpha \beta \gamma}), d\log h_{\alpha \beta \gamma} + \delta' (\theta_{\alpha \beta}), \delta'(\nu_\alpha) -d\theta_{\alpha \beta}).$$ The triple $(h_{\alpha \beta \gamma}, \theta_{\alpha \beta}, \nu_\alpha)$ is a cocycle if and only if $D(h_{\alpha \beta \gamma}, \theta_{\alpha \beta}, \nu_\alpha) = (1, 0, 0)$. By equating coordinates and using the above definition of $\delta'$ we obtain the result.
\end{proof}

\indent We now construct the Deligne cocycle associated to any bundle gerbe with connective data, and prove it is a cocycle using the previous proposition. Let $(P, Y, M)$ be a bundle gerbe with connective data $(\nabla, f)$ and $\{U_{\alpha}\}_{\alpha \in I}$ be a good cover of $M$. Following the construction from Example \ref{DDconstruction}, consider sections $s_\alpha: U_\alpha \to Y$ and define the line bundle $P_{\alpha \beta} := (s_\alpha, s_\beta)^{-1}P$ over $U_\alpha\cap U_\beta$.  Since $U_\alpha \cap U_\beta$ is contractible, $P_{\alpha \beta} \to U_\alpha \cap U_\beta$ is trivial, so there exists a unit section $\sigma_{\alpha \beta}$ on $U_\alpha \cap U_\beta$. Recall that we can define \v{C}ech cocycles $g_{\alpha \beta \gamma}$ by $\sigma_{\alpha \beta} \sigma_{\beta \gamma} = g_{\alpha \beta \gamma} \sigma_{\alpha \gamma}.$ Define $A_{\alpha \beta} \in \Omega^1(U_\alpha \cap U_\beta)$ implicitly by $$(s_\alpha, s_\beta)^{*}\nabla \sigma_{\alpha \beta} = A_{\alpha \beta}\sigma_{\alpha \beta}.$$ Finally, set $f_\alpha := s_{\alpha}^{-1}f$. We will prove that the triple $(g_{\alpha \beta \gamma}, -A_{\alpha \beta}, f_\alpha)$ defines a Deligne cocycle in the following proposition.
\begin{proposition}\label{classindeligne} 
	Let $(P, Y, M)$ be a bundle gerbe with connective data $(\nabla, f)$. Then the triple $(g_{\alpha \beta \gamma}, -A_{\alpha \beta}, f_\alpha)$ constructed above defines a class in $H^3(M, \Z(3)_D)$.
\end{proposition} 
\begin{proof}
We need only verify that equations (\ref{eqn11})--(\ref{eqn13}) are satisfied by $(g_{\alpha \beta \gamma}, -A_{\alpha \beta}, f_\alpha)$. Since we know $g_{\alpha \beta \gamma}$ to be a \v{C}ech cocycle, (\ref{eqn11}) is immediate. Next, observe that by pulling back equation (\ref{bundlegerbeconnection}) by $(s_\alpha, s_\beta, s_\gamma): U_\alpha \cap U_\beta \cap U_\gamma \to Y^{[3]}$ we obtain $$(s_\alpha, s_\gamma)^{*}\nabla m(\sigma_{\alpha \beta}, \sigma_{\beta \gamma} ) = m((s_\alpha,  s_\beta)^{*}\nabla \sigma_{\alpha \beta}, \sigma_{\beta \gamma}) + m(\sigma_{\alpha \beta}, (s_\beta, s_\gamma)^{*}\nabla \sigma_{\beta \gamma}).$$ Therefore $(s_\alpha, s_\gamma)^{*}\nabla m(\sigma_{\alpha \beta}, \sigma_{\beta \gamma} ) = (A_{\alpha \beta} + A_{\beta \gamma}) m(\sigma_{\alpha \beta}, \sigma_{\beta \gamma} ).$ Since $m(\sigma_{\alpha \beta}, \sigma_{\beta \gamma}) = g_{\alpha \beta \gamma} \sigma_{\alpha \gamma}$, by Remark \ref{connectiononeformsrelationship} we have $$A_{\alpha \beta} + A_{\beta \gamma} = A_{\alpha \gamma} + d\log g_{\alpha \beta \gamma}, $$ that is, $\delta'(A)_{\alpha \beta \gamma} = d\log g_{\alpha \beta \gamma}.$ So $-A_{\alpha \beta}$ satisfies equation (\ref{eqn12}). Finally, recall that the curving $f$ satisfies $\delta(f) = F_{\nabla},$ and $F_{(s_\alpha, s_\beta)^*\nabla} = dA_{\alpha \beta} = (s_\alpha, s_\beta)^*F_{\nabla}$, so  \begin{align*}
d A_{\alpha \beta} &= (s_\alpha, s_\beta)^*F_{\nabla} \\
&= (s_\alpha, s_\beta)^*\delta(f)\\
&= (s_\alpha, s_\beta)^*(\pi_1^*f - \pi_2^*f)\\
&= s_\beta^*f - s_\alpha^* f \\
&= \delta'(f)_{\alpha \beta}. 
\end{align*} This shows that $f_\alpha$ satisfies equation (\ref{eqn13}), hence the triple $(g_{\alpha \beta \gamma}, - A_{\alpha \beta}, f_\alpha)$ defines a class in $H^3(M, \Z(3)_D)$. 
\end{proof}
We present the next proposition without proof, and refer the reader to \cite{stableiso} for details. This result, combined with our above work, will allow us to present a sketch of the proof of Theorem \ref{delignetheorem}.
\begin{proposition}[{\cite[Proposition 4.2]{stableiso}}] \label{zerodeligneclass} 
Let $(P, Y, M)$ be a bundle gerbe with connective data and zero Deligne class. Then $(P, Y, M)$ is a trivial bundle gerbe with trivial connection and curving.
\end{proposition}
\noindent \noindent \textbf{Theorem \ref{delignetheorem}} ({\cite[Theorem 4.1]{stableiso}})\textbf{.} \textit{The Deligne class defines a bijection between $D$-stable isomorphism classes of bundle gerbes with connective data on $M$ and $H^3(M, \Z(3)_D)$.} 
\begin{proof} [Sketch of Proof] We follow the proof from \cite{stuartthesis}.
	Let $(P, Y, M)$ be a bundle gerbe with connective data $(\nabla, f)$. By Proposition \ref{classindeligne}, the triple $(g_{\alpha \beta \gamma}, A_{\alpha \beta}, f_{\alpha})$ associated to $(P, Y, M)$ defines a class in $H^3(M, \Z(3)_D)$. This is shown to be well-defined in \cite{stuartthesis}. The map which assigns to a bundle gerbe $(P, Y, M)$ the class $[(g_{\alpha \beta \gamma}, A_{\alpha \beta}, f_\alpha)]$  is clearly a homomorphism, and is injective by Proposition \ref{zerodeligneclass}. It remains to show how, given a Deligne class $(g_{\alpha \beta \gamma}, A_{\alpha \beta}, f_{\alpha}),$ one can construct a bundle gerbe. Define $$Y := \{(m, \alpha) \ | \ m\in U_\alpha\} \subseteq M\times I,$$ the disjoint union of sets in the open cover $\mathcal{U}$ of $M$. So $Y^{[2]}$ consists of triples $(m, \alpha, \beta)$ such that $m\in U_\alpha\cap U_\beta$. Let $P\to Y^{[2]}$ be the trivial bundle and define a bundle gerbe product by $$(m, \alpha, \beta, z) \cdot (m, \beta, \gamma, w) = (m, \alpha, \gamma, zwg_{\alpha \beta \gamma}(m))$$ for $z, w \in U(1)$. It can be shown that the resulting bundle gerbe $(P, Y, M)$ defines the desired Deligne class (in fact, $d+A_{\alpha \beta}$ defines a bundle gerbe connection on $(P, Y, M)$ with curving $f_\alpha$). 
\end{proof}
It can be shown by similar arguments to those above that bundle gerbes over $M$ with connection and no choice of curving are classified by the Deligne cohomology group $H^3(M, \Z(2)_D)$ \cite{stuartthesis}. This leads us to the following proposition, which summarises our classification results. The remarks following this proposition will conclude our preliminary study of bundle gerbes.
\begin{proposition} \label{bgclassification} Let $M$ be a smooth manifold. Then
\begin{enumerate}[(1),font=\upshape] 
\item  stable isomorphism classes of bundle gerbes over $M$ are in bijective correspondence with $H^3(M, \Z);$
\item stable isomorphism classes of bundle gerbes over $M$ with connection are in bijective correspondence with $H^3(M, \Z(2)_D);$
\item $D$-stable isomorphism classes of bundle gerbes over $M$ with connection and curving are in bijective correspondence with $H^3(M, \Z(3)_D)$. 
\end{enumerate}
\end{proposition}
\begin{remark}\label{equivalenceof1and2} 
Conditions $(1)$ and $(2)$ in Proposition \ref{bgclassification} are, in fact, equivalent. By Proposition \ref{connectionsontrivialbundlegerbes}, any two stably isomorphic bundle gerbes will have equal Deligne class in $H^3(M, \Z(2)_D)$ when they are each endowed with a connection. Conversely, stably isomorphic bundle gerbes with connection will certainly be stably isomorphic. Indeed, $H^3(M, \Z)\cong H^3(M, \Z(2)_D)$. In light of this, we see that the condition that the stable isomorphism in Proposition \ref{outofnamesforthings} preserve connections is redundant, since any stable isomorphism can be given a connection that does this.
\end{remark}
	\begin{remark}\label{remark deligne cohomology of line bundles and functions} 
	It is a standard fact that line bundles over $M$ with connection are classified by $H^2(M, \Z(2)_D)$. Therefore Proposition \ref{bgclassification} (2) formalises the notion of a bundle gerbe \textit{with connection} being a `higher' version of a line bundle \textit{with connection}.
\end{remark}
\begin{remark}\label{rrr}
	It is possible to use the theory of \textit{differential characters} \cite{chernsimmons} to relate the three-curvature and holonomy of a bundle gerbe to its Deligne class. \end{remark}
 \begin{remark} \label{remark real part of DD class} Consider a bundle gerbe $(P, Y, M)$ with connective data $(\nabla, f)$ and three-curvature $\omega$. Let $\{U_{\alpha}\}_{\alpha \in I}$ be a  Leray cover of $M$ and consider the data $g_{\alpha \beta \gamma}, A_{\alpha \beta}$ and $f_{\alpha}$ constructed above. Then
$$\left(\tfrac{1}{2\pi i}\omega|_{U_\alpha}, dA_{\alpha \beta}, d\log g_{\alpha \beta \gamma}, \delta' \log g_{\alpha \beta \gamma}\right) \in \Pi\, \Omega^3(U_\alpha)\oplus \Pi \, \Omega^2(U_{\alpha \beta}) \oplus\Pi\, \Omega^1(U_{\alpha \beta \gamma}) \oplus \Pi \, \Omega^0(U_{\alpha \beta \gamma \delta})$$ and by standard homological arguments this defines a an element in the total complex of the double complex seen in the proof of the \v{C}ech-de Rham isomorphism in \cite{botttu}. It is then a consequence of this isomorphism that the image of $\left[\tfrac{1}{2\pi i}\omega \right]\in H^3_{\text{dR}}(M)$ in $H^3(M, \R)$ is equal to $$[\delta' \log g_{\alpha \beta \gamma} ] = {r}(DD(P, Y, M))$$ as claimed in Remark \ref{curvatureisrealpartofclass}.

\end{remark}

%% file: chapter3/CHP3.tex
	\chapter{The cup product bundle gerbe}\label{cupproductchapter} 
	\begin{chapquote}{J.-L. Brylinski}
Given a line bundle $L$ over a manifold $X$ and a smooth function $f : X \to \C^*$, there is a gerbe attached
to $L$ and $f$... this gerbe is given by a cup-product construction.
	\end{chapquote}
		\indent Having discussed the prerequisite bundle gerbe theory, we can now focus on the central question of this thesis. The three remaining chapters are dedicated to the study of the \textit{cup product bundle gerbe} $\cupbg$, the \textit{basic bundle gerbe} $\basicbgnopi$, and the relationship between them. We remind the reader that this latter relationship is our claim that the pullback of the basic bundle gerbe by the Weyl map $$p:T\times \su\to SU(n)$$ is $SU(n)$-stably isomorphic to the cup product bundle gerbe. Our aim in this chapter is to introduce the cup product bundle gerbe, study its equivariance, and explicitly describe its connective data in such a way as to simplify our computations in the final chapter.\\
\indent Our first section, \textit{The geometry of $\su$,} is distinct from the rest of this chapter. There is, in fact, no mention of bundle gerbes in this section, and as such its placement here may appear somewhat mysterious. The necessity for Section $3.1$ will become clear by the end of Section $3.3$. In short, it offers a description of the space $\su$ in terms of $n$-tuples of orthogonal projections. This will, in turn, allow us to describe the connective data of our bundle gerbes very explicitly  in the language of projections. We encourage the reader to skim read this section, and return to it after its application in Section $3.3$. \\ 
\indent To define the cup product bundle gerbe and compute its connective data, we utilise a three-tiered approach, progressing from \textit{general cup product bundle gerbes}, to the \textit{$i$-th cup product bundle gerbes}, and finally, to the cup product bundle gerbe $(P_c, X)$ itself. In Section \ref{difference cup}, we define general cup product bundle gerbes (Subsection \ref{subsectiondefinitionofgeneral}), consider stable isomorphisms between general cup product bundle gerbes (Subsection \ref{subsectionstableisocups}), and calculate their connective data (Subsection \ref{subsectionconnectivedata}). Subsection \ref{subsectionstableisocups} will not be used in this chapter, but will be crucial to our later work. In Section \ref{ithbg}, we restrict our study to \textit{$i$-th cup product bundle gerbes} for $i=1, ..., n$ (Definition \ref{ex 1}), which are general cup product bundle gerbes over $T\times \su$. The results from Section \ref{difference cup} will be applied to calculate the connective data on the $i$-th cup product bundle gerbes (Proposition \ref{connective data on ith cup product bundle gerbe}). In the final section of this chapter, Section \ref{cupprodbgsection}, the cup product bundle gerbe $(P_c, X)$ is defined as the product of the $i$-th cup product bundle gerbes (Definition \ref{ex 1}), and results from Section \ref{ithbg} are used to compute its connective data (Proposition \ref{corollary data on cup prod}).\\
\indent The cup product bundle gerbe constructed in this chapter is, to our knowledge, the first explicit description of this bundle gerbe in the literature. It is based loosely on the {cup product bundle gerbe construction} due to S. Johnson \cite{stuartthesis}. Recall that a function $f:M\to \C^*$ is classified by its winding number in $H^1(M, \Z)$. Therefore, if $f, g:M\to \C^*$ have winding class $\alpha, \beta$ respectively, then $\alpha \cup \beta$ is a class in $H^2(M, \Z)$ that determines a line bundle. In  \cite{beilinson, delignecuponlinebundles}, Deligne and Be\u{\i}linson exploited this to describe a line bundle with connection over $M$ in terms of functions $f, g:M\to \C^*$, thereby providing a geometric interpretation of the cup product of functions. In 1994, Brylinski expanded upon this idea, using the fact that the cup product of $\alpha \in H^1(M, \Z)$ and $\beta \in H^2(M, \Z)$ is an element of $H^3(M, \Z)$, the group that classifies gerbes. Since any $\alpha \in H^1(M, \Z)$ is the winding class of some $g:M\to U(1)$, and $\beta$ is the first Chern class of some line bundle $P\to M$, Brylinski was able to define the \textit{cup product gerbe} from a function and a line bundle \cite{brylinski, brylinskimcglaughlin}. We recommend \cite{brylinskiarticle} for a refined definition of this gerbe. Notably to this work, in 2000, Brylinski considered a cup product bundle gerbe over the product $T\times {G}/{T}$, showing it pushed forward to a gerbe on $G$ by the Weyl map \cite{brylinskiarticle}. Later, in 2002, S. Johnson (under the supervision of M. K. Murray) defined the \textit{cup product bundle gerbe} of a smooth map $M\to S^1$ and line bundle $L\to M$ in his PhD thesis \cite{stuartthesis}. Our work will draw upon the ideas of Johnson and Brylinski to define the cup product bundle gerbe over $T\times \su$.

\section{The geometry of $\su$}\label{flagmanifold}
\indent \indent We will temporarily divert our attention away from bundle gerbes to consider the left coset space $\su$ for $T$ the space of diagonal matrices in $SU(n)$. We define $T$ formally below. 
\begin{definition}\label{definitionofourtorus}	Let $T$ be the subgroup of $SU(n)$ consisting of diagonal matrices. Explicitly, $T := \left\{\text{diag}(e^{2\pi i x_1}, ..., e^{2\pi i x_n}) \ |  \ x_1 + \cdots + x_n = 0\right\}.$\end{definition}
It is a standard fact that $\su$ is a manifold, see, for instance \cite[Proposition 4.2]{kobayashi}. In this section, we will introduce the manifold $\text{Proj}_n$ consisting of $n$-tuples of orthogonal projections, and show it is diffeomorphic to $\su$. We will then consider a line bundle $J_i\to \su$, and its isomorphism  with a line bundle $K_i\to \text{Proj}_n$. The line bundle $J_i\to \su$ will be key to defining the cup product bundle gerbe, and its isomorphism with $K_i\to \text{Proj}_n$ will be crucial to defining the connective data on the cup product bundle gerbe in terms of projections. 
This section can be skim-read, and returned to at a later time. Let us begin by defining $\text{Proj}_n$. 
\begin{definition}
	For $n\in \N$, let $\text{Proj}_n$ be the set of $n$-tuples of ordered orthogonal projections $(P_1, ..., P_n)$, where, for each $i$, $P_i:\C^n\to W_i$, and the spaces $W_i$ are mutually orthogonal one-dimensional subspaces of $\C^n$.
\end{definition}
\begin{remark}
	Clearly if $(P_1, ..., P_n)\in \text{Proj}_n$, then $\sum_{i=1}^n P_i = 1$ and $P_iP_j = 0$ if $i\neq j$. These facts will be used extensively to simplify later calculations.
\end{remark}
We next prove the first key proposition in this section. 

\begin{proposition}\label{SU(n)/T is Proj_n}
There is a bijection $\su\cong\textup{Proj}_n$.
\end{proposition}
\begin{proof}
	We will utilise Theorem \ref{theorem} (which can be applied to sets and functions to give us a bijection, rather than a diffeomorphism). Consider the smooth left action of $SU(n)$ on $\text{Proj}_n$ defined by \begin{align}\label{action} g \cdot (P_1, ..., P_n) = (gP_1g^{-1}, ..., gP_ng^{-1}).\end{align} We claim that this action is transitive. To see this, let $(P_1, ..., P_n), (Q_1, ..., Q_n)\in \text{Proj}_n$, and choose unit vectors $p_i \in \text{im}(P_i)$, $q_i\in \text{im}(Q_i)$ for each $i$. Then $(p_i)_{i=1}^n$ and $(q_i)_{i=1}^n$ are orthonormal bases of $\C^n$, hence are related by an element of $g\in U(n)$ which we can scale to make an element of $SU(n)$. It follows that $g\cdot (P_1, ..., P_n) =(Q_1, ..., Q_n)$. So this action is indeed transitive, and makes $\text{Proj}_n$ into a homogeneous $SU(n)$-space. \\ \indent For $i=1, ..., n$, let $O_i$ be the orthogonal projection onto the span of the $i$-th standard basis vector of $\C^n$. Set $e:= (O_1, ..., O_n)\in \text{Proj}_n$. It easy to verify that the isotropy group of $e$ is the set of all diagonal matrices in $SU(n)$, i.e. $(\text{Proj}_n)_e = T$. Therefore by Theorem \ref{theorem}, $\su \cong \text{Proj}_n$ via the map $gT \mapsto (gO_1g^{-1}, ..., gO_ng^{-1})$.\end{proof} This result gives rise to the following important proposition.
\begin{proposition}\label{treehouse}
	The space $\text{Proj}_n$ has a unique manifold structure making the $SU(n)$ action $(\ref{action})$ smooth.
\end{proposition}
\begin{proof}
This follows from \cite[Proposition 7.21]{LEE}, by noting that the action $(\ref{action})$ is transitive from the above proof and the isotropy subgroup of $e\in \text{Proj}_n$, namely $T$, is a closed Lie subgroup of $SU(n)$.
\end{proof}
\begin{remark}
The reader familiar with flag manifolds will recognise $SU(n)/T$ as the manifold of all full flags in $\C^n$, i.e.  the space of $n$-tuples $(V_1, ..., V_n)$ of subspaces (called \textit{flags}) with $V_1\subset \cdots \subset V_n = \C^n$ and dim$V_k = k$. Equivalently, we can identify each flag $(V_1, ..., V_n)$ with the $n$-tuple of mutually orthogonal $1$-dimensional subspaces of $\C^n$, $(V_1, V_2\cap V_1^{\perp}, \ldots, V_n\cap V_{n-1}^{\perp})$. Such an $n$-tuple can, in turn, be identified with the $n$-tuple of orthogonal projections onto each $1$-dimensional subspace, yielding the bijection in Proposition \ref{SU(n)/T is Proj_n}.
\end{remark}

The remainder of this section is dedicated to introducing the line bundles $J_i\to \su$, and proving they are isomorphic to line bundles $K_i\to \text{Proj}_n$ for each $i$. We will show this using an intermediary isomorphism with a line bundle $X_i\to \text{Proj}_n$. These results will be used implicitly throughout the later parts of this chapter. Let us begin by defining $J_i\to \su$.
\begin{definition}\label{salad} The homomorphism ${p}_i: T\to S^1$ is defined by $$p_i\left(\text{diag}(e^{2\pi i x_1}, ..., e^{2\pi i x_n})\right) =e^{2\pi i x_i}.$$\end{definition}
\begin{definition}\label{definitionofithcupproduclinebundle}
For $i=1, ..., n$, let $J_i  := \C\times_{p_i} SU(n)$ be the set of equivalence class in $\C\times SU(n)$ under the relation $(z, s) \sim_{p_i} (p_i(t^{-1})z, st)$ for all $t\in T$. 
\end{definition}
Note that $J_i\to SU(n)/T$ is the $SU(n)$-homogeneous vector bundle associated to the linear representation on $T$ defined by $p_i$, following the construction from Section \ref{homogeneous vector bundles}. Here the $SU(n)$ action on $SU(n)/T$ is defined by left multiplication.\\
\indent We next introduce line bundles $X_i\to \text{Proj}_n$, and show they are isomorphic to $J_i\to \su$. It will then be shown that $X_i\to \text{Proj}_n$ is isomorphic to $K_i\to \text{Proj}_n$, thereby acting as our intermediary isomorphism to showing that $J_i\to \text{Proj}_n$ is isomorphic to $K_i\to \text{Proj}_n$.
\begin{definition}\label{definitionxi}
For $i\in \{1, ..., n\}$, let $e_i$ be the $i$-th standard basis vector of $\C^n$ and $O_i$ be orthogonal projection onto $\text{span}(e_i)$. Define $$X_i= \bigcup_{g\in SU(n)} \text{span}(ge_i)\times \left\{\left(gO_1g^{-1}, ..., gO_ng^{-1}\right)\right\} \subset \C^n \times \text{Proj}_n.$$ 
\end{definition}
\begin{remark}
	Projection onto the second factor makes $X_i\to \text{Proj}_n$ into a line bundle. This line bundle is $SU(n)$-equivariant with respect to the conjugation action on $\text{Proj}_n$ and the $SU(n)$-action on $\C^n$ defined by (left) multiplication.
\end{remark}
\begin{proposition}\label{XandJ}
For each $i= 1, ..., n$, there is an $SU(n)$-equivariant line bundle isomorphism from $J_i\to \su$ to $X_i\to \textup{Proj}_n$.
\end{proposition}
\begin{proof}
It is easy to verify that the map $J_i\to X_i$, $
	[z, g]\mapsto \left(zge_i, \left(gO_1g^{-1}, ..., gO_ng^{-1}\right)\right)$
is a well-defined $SU(n)$-equivariant diffeomorphism covering $\su \xrightarrow{\sim} \text{Proj}_n$.
\end{proof}

Note that we can write $\text{span}(ge_i) = \text{im}(gO_ig^{-1})$ in Definition \ref{definitionxi}. This idea motivates us o the line bundles $K_i\to \text{Proj}_n$, which will be shown to be isomorphic to $J_i\to SU(n)/T$ using Proposition \ref{XandJ}. 

\begin{definition}
For each $i=1, ..., n$, let  $K_i \to \text{Proj}_n$ be the line bundle defined by $K_i = \{(w, P_1, ..., P_n) \ | \ w\in \text{im}P_i \}\subset \C^n\times \text{Proj}_n.$
\end{definition}
It follows that the fibre of $K_i$ at a point $(P_1, ..., P_n)$ is im$(P_i)$. Clearly $K_i\to \text{Proj}_n$ is $SU(n)$-equivariant with respect to the $SU(n)$-action on $K_i$ defined by $g\cdot (v, P_1, ..., P_n) = (gv, gP_1g^{-1}, ..., gP_ng^{-1})$.
\begin{proposition}
For each $i$, there is an $SU(n)$-equivariant line bundle isomorphism from $J_i\to \su$ to $K_i\to \textup{Proj}_n$. 
\end{proposition}
\begin{proof}
By setting $(O_1, ..., O_n)\in \text{Proj}_n$ to be the basepoint of $\text{Proj}_n$, it follows from Theorem \ref{thrmbundles} that $K_i$ and $X_i$ are isomorphic as line bundles over $\text{Proj}_n$. Therefore $J_i\to \su$ is isomorphic to $K_i\to \text{Proj}_n$ by Proposition \ref{XandJ}.
\end{proof}

Throughout the later parts of this chapter, we will continue to write $J_i\to \su$, but will in practice work with the line bundles $K_i\to \text{Proj}_n$. This is advantageous in that our computations will take place in the space of $n$-tuples of orthogonal projections, instead of the more abstract space $\su$. This concludes our discussion of $\su$ for the time being. 

\section{General cup product bundle gerbes}\label{difference cup}
%
\indent \indent In this section, we will consider a non-standard bundle gerbe construction, which, for lack of a better name, we call the \textit{general cup product bundle gerbe construction.} Just as the bundle gerbe constructions in Section \ref{sectionconstructions} relied on line bundle constructions from Section \ref{section constructions}, so too does the general cup product bundle gerbe. Namely, it relates to \textit{function powers} of line bundles (Example \ref{tereq}). \\
\indent General cup product bundle gerbes will have two key applications in this thesis. As we have mentioned, they will ultimately allow us to define the cup product bundle gerbe (via the $i$-th cup product bundle gerbes), and compute its connective data. The second, less obvious, application of general cup product bundle gerbes will be to the computations in our final chapter. As we shall see, the theory of general cup product bundle gerbes, and in particular the work from Subsection \ref{subsectionstableisocups}, will significantly simplify our central research question.
\subsection{Definition}\label{subsectiondefinitionofgeneral}
\indent \indent Let us begin with some motivation. Recall that smooth maps $f:M\to S^1$ are classified by $H^1(M, \Z)$, and line bundles with connection are classified by $H^{2}(M, \Z)$. We saw in the previous chapter that bundle gerbes are classified by $H^3(M, \Z)$. Therefore we expect, given a smooth map $f:M\to S^1$ and a line bundle $L\to M$, to be able to construct a bundle gerbe. This is achieved by Johnson's construction in \cite{stuartthesis}, and is the basis for our more general cup product bundle gerbe which we describe now.\\
\indent Consider a surjective submersion $Y\to M$, a line bundle $L\to N$, and a smooth map $g:Y^{[2]}\to \Z$ satisfying $\delta(g)=0$. By Example \ref{tereq}, $L^g\to Y^{[2]}\times N\cong (Y\times N)^{[2]}$ is a line bundle. We claim that the triple $(L^g, Y\times N, M\times N)$ is a bundle gerbe. To see this, we need only show that there is an associate multiplication on this triple. This follows by noting that $L^{g(y_1, y_2)}\otimes L^{g(y_2, y_3)}= L^{g(y_1, y_3)}$ for all $(y_1, y_2, y_3)\in Y^{[3]}$, since $\delta(g)= 0$, therefore this is indeed a bundle gerbe. 

\begin{definition}\label{definition of general cup product} 
Let $Y\to M$ be a surjective submersion, $L\to N$ be a line bundle, and $g:Y^{[2]}\to \Z$ be a smooth map satisfying $\delta(g) = 0$. The bundle gerbe $(L^g, Y\times N, M\times N)$ is called the \textit{general cup product bundle gerbe of $L\to N$ and $g$,} depicted in Figure \ref{figapple}. 	\begin{align*} \begin{array}{ccc}
	L^{g}      &          &        \\
	\downarrow &         &         \\
	Y^{[2]}\times N & \stackrel{}{\dra} &  Y\times N\\
	&               &      \downarrow \\
	&               &       M\times N     \\
	\end{array} \end{align*}
	\captionof{figure}{A general cup product bundle gerbe}
	\label{figapple}
\end{definition}

\begin{remark}
	The reader familiar with Johnson's cup product bundle gerbes will recall that they take the form $(f\cup L, f^{-1}(\R), M)$ for $f:M\to S^1$ a smooth map and $L\to M$ a line bundle. Fibrewise, the total space $f\cup L$ is defined by $(f\cup L)_{(m, x, y)} = L^{x-y}_{m}$. This is in fact the pullback of a general cup product bundle gerbe $(L^d, f^{-1}(\R)\times M, M\times M)$ by the map $M\to M\times M$, $m\mapsto (m, m)$, where $d$ is the difference map $d(x, y) = x-y$. 
\end{remark}
\subsection{Stable isomorphisms}\label{subsectionstableisocups}

\indent \indent Next, we will consider criterion for general cup product bundle gerbes to be stably isomorphic to one another. Although these results will not be referenced again in this chapter, they are critically important to our final chapter. As we shall see, both the cup product bundle gerbe and the pullback of the basic bundle gerbe are stably isomorphic to a reduced product of $n$ general cup product bundle gerbes. Therefore, in order to understand if these two bundle gerbes are stably isomorphic, we should consider when, more generally, reduced products of general cup product bundle gerbes are stably isomorphic. The reader should feel free to skip ahead to Subsection \ref{subsectionconnectivedata} if they wish, and return to these results after their application in the final chapter. \\
\indent This subsection consists of a key proposition followed by two corollaries. This proposition, Proposition \ref{tererererer}, describes a condition for a general cup product bundle gerbe to be trivial. Then, in Corollary \ref{pop}, we use this result to describe a condition for two general cup product bundle gerbes $(L^f, Y\times N, M\times N)$ and $(L^g, X\times N, M\times N)$ to be stably isomorphic. Finally, our main result of this subsection, Corollary \ref{cormaincupproductresult}, describes a condition for the reduced products of general cup product bundle gerbes to be stably isomorphic. 
\begin{proposition}\label{tererererer}
Consider a general cup product bundle gerbe $\left(L^g, Y\times N, M\times N\right)$. If there exists a smooth map $h: Y\to \Z$ such that \begin{align}\label{tree} g(y_1, y_2) = h(y_2)-h(y_1)\end{align} for all $(y_1, y_2)\in Y^{[2]}$, then the bundle gerbe $\left(L^g, Y\times N, M\times N\right)$ is trivial. 
\end{proposition}
\begin{proof}
Assume such an $h$ exists. Then $L^h\to Y\times N$ is a line bundle and $$\delta\left(L^h\right)_{(y_1, y_2, n)} = L_n^{h(y_2)-h(y_1)} = (L^g)_{(y_1, y_2, n)}$$ for all $(y_1, y_2, n)\in Y^{[2]}\times N$. So $L^g\to Y^{[2]}\times N$ and  $\delta\left(L^h\right)\to Y^{[2]}\times N$ are isomorphic. This line bundle isomorphism clearly preserves the bundle gerbe product. Therefore $\left(\delta\left(L^h\right), Y\times N, M\times N\right)$ is isomorphic to $\left(L^g, Y\times N, M\times N\right)$ and the general cup product bundle gerbe is trivial.
\end{proof}
\begin{remark}
	In Proposition \ref{tererererer}, it is crucial that the function $h$ is \textit{integer-valued}. Since $\delta(g) = 0$ by assumption, by exactness of the fundamental complex it will always be the case that there exists $f: Y\to \R$ satisfying $\delta(f) = g$. However, it is not possible to define the trivialising line bundle as we have done above unless this function is integer-valued, which in general will not be the case. 
\end{remark}
\begin{corollary}\label{pop}
	Consider two general cup product bundle gerbes $\left(L^f, Y\times N, M\times N\right)$ and $(L^g, X\times N, M\times N)$. If there exists a smooth map $h: (X\times_M Y)\times N\to \Z$ such that $$f(y_1, y_2)-g(x_1, x_2) = h(x_2, y_2, n)- h(x_1, y_1, n)$$ for $(x_1, y_1, x_2, y_2, n)\in (X\times_M Y)^{[2]}\times N$, then $$\left(L^f, Y\times N, M\times N\right) \cong_{\textup{stab}} (L^g, X\times N, M\times N)$$ with trivialising line bundle $\pi_N^{-1}(L)^h\to (X\times_M Y)\times N$ for $\pi_N: (X\times_MY)\times N\to N$ the natural projection.
\end{corollary}
%
%

\begin{corollary}\label{cormaincupproductresult}
	For each $i=1, ..., n$, let $(K_i^{f_i}, Y\times N, M\times N)$ and $(K_i^{g_i}, X\times N, M\times N)$ be general cup product bundle gerbes. If, for each $i= 1, ..., n$, there exists a smooth function $h_i: (X\times_M Y)\times N\to \Z$ such that \begin{align}\label{ddssss} f_i(y_1, y_2)-g_i(x_1, x_2) = h_i(x_2, y_2, n)-h_i(x_1, y_1, n)\end{align} for all $(x_1, y_1, x_2, y_2, n) \in (X\times_M Y)^{[2]}\times N$, then \begin{align}\label{fdsfss}\sideset{}{_\mathrm{red}}\bigotimes_{i=1}^n  (K_i^{f_i}, Y\times N, M\times N) \cong_{\textup{stab}} \sideset{}{_\mathrm{red}}\bigotimes_{i=1}^n (K_i^{g_i}, X\times N, M\times N)\end{align}with trivialising line bundle $$\Motimes_{i=1}^n \pi_N^{-1}(K_i)^{h_i} \to (X\times_M Y) \times N$$ for $\pi_N: (X\times_M Y)\times N\to N $ the natural projection. 
\end{corollary}
\begin{proof}
First, note that \begin{align}\label{chicken}
\bigotimes_{i=1}^n (K_i^{f_i}, Y\times N, M\times N) \cong_{\textup{stab}} \bigotimes_{i=1}^n (K_i^{g_i}, X\times N, M\times N) 
\end{align} if, for each $i$, $ (K_i^{f_i}, Y\times N, M\times N) \cong_{\textup{stab}}  (K_i^{g_i}, X\times N, M\times N)$. By Corollary \ref{pop}, this will hold if there exists smooth $h_i: (X\times_MY)\times N\to \Z$ satisfying (\ref{ddssss}). The result then follows by Remark \ref{fibreproductsimplifies}.
\end{proof}

\subsection{Connective data}\label{subsectionconnectivedata}
\indent \indent In this subsection, we will describe a natural connection on the general cup product bundle gerbe of $L\to N$ and $g:Y^{[2]}\to \Z$ induced by a line bundle connection $\nabla$ on $L$ (Proposition \ref{existsconnectiononcup}). A curving and associated three-curvature for this connection will then be calculated in terms of $g$ and $F_{\nabla}$ (Proposition \ref{curvaturetheorem1}). This result will be utilised to find connective data on the $i$-th cup product bundle gerbes introduced in the next section.  Throughout the following propositions, let $\pi_N: M\times N\to N$, $\pi_{Y^{[2]}}:Y^{[2]}\times N\to Y^{[2]}$ and $\pi_N^{[2]}:Y^{[2]}\times N\to N$ be the natural projections.
\begin{proposition}\label{existsconnectiononcup} \textbf{\textup{(A connection on a general cup product bundle gerbe)}}
	Consider a general cup product bundle gerbe $(L^g, Y\times N, M\times N)$. Let $\nabla$ be a connection on $L\to N$ with curvature $F_\nabla$. Then there is a bundle gerbe connection $\nabla^{g}$ on $(L^g, Y\times N, M\times N)$ with two-curvature given by $$F_{\nabla^{g}} = g\, (\pi_N^{[2]})^{*}F_{\nabla}.$$ 
\end{proposition}
%
\begin{proof}
	By Example \ref{tereq}, there is a connection $\nabla^g$ on $L^g\to Y\times N$ with the desired curvature. This can easily be verified to be a bundle gerbe connection.
	%
\end{proof}
We next describe the connective data associated to the connection in Proposition \ref{existsconnectiononcup}. This result will be key to the computations in the next section.
\begin{proposition} \label{curvaturetheorem1} \textbf{\textup{(Geometry of a general cup product bundle gerbe)}}
Consider a general cup product bundle gerbe $(L^g, Y\times N, M\times N, \pi)$. Let $\nabla$ be a connection on $L\to N$, and let $\nabla^{g}$ be the induced bundle gerbe connection from \textup{Proposition \ref{existsconnectiononcup}}. Then 
	\begin{enumerate}[(1),font=\upshape]
		\item there exists a smooth function $\varphi:Y\times N \to \R$ such that $\delta(\varphi) = \pi_{Y^{[2]}}^* g;$
		\item for any such $\varphi$, $d\varphi = \pi^*(\eta)$ for some $\eta \in \Omega^1(M\times N);$
		\item the $2$-form $f\in \Omega^2(Y\times N)$ defined by $$f = - \varphi \, (\pi_N\circ \pi)^*F_\nabla$$ satisfies $\delta(f) = F_{\nabla^{g}}$, so $f$ is a curving for the  connection $\nabla^{g};$
		\item the three-curvature $\omega \in \Omega^3(M\times N)$ of $(\nabla^{g}, f)$ is given by $$\omega =  - \eta \wedge \pi_N^*F_{\nabla}; $$ 
		\item the real Dixmier-Douady class of $(L^g, Y\times N, M\times N)$ is represented by
		$$-\frac{1}{2\pi i}\, \eta\wedge \pi_N^*F_{\nabla}.$$ 
	\end{enumerate}
\end{proposition}

\begin{proof} Let us first verify that the function $\varphi$ and $1$-form $\eta$ exist. By assumption, $$\delta(\pi_{Y^{[2]}}^*g) = \pi_{Y^{[2]}}^*\delta(g) = 0,$$ so exactness of the fundamental complex implies there exists $\varphi:Y\times N \to \R$ satisfying $\delta(\varphi) = \pi_{Y^{[2]}}^*g$. Similarly, $$\delta d(\varphi) = d\delta \varphi= d\pi_{Y^{[2]}}^*g = \pi_{Y^{[2]}}^*dg= 0,$$ so again by exactness of the fundamental complex, there exists $\eta\in \Omega^1(M\times N)$ such that $\pi^*(\eta) = d\varphi$. This proves $(1)$ and $(2)$. \\
\indent We compute $\delta(f)$ as follows. Let $\pi_1, \pi_2: Y^{[2]} \times N \to Y\times N $ be the first and second projection maps (Definition \ref{fprod}). At a point $(x, y, n) \in Y^{[2]}\times N,$
	\begin{align*}
	\delta(f)\big\rvert_{(x, y, n)} = -\varphi(y)\,(\pi_N\circ \pi \circ \pi_2)^* F_\nabla\big\rvert_{(x, y, n)}  + \varphi(x)\, (\pi_N\circ \pi \circ \pi_1)^* F_\nabla\big\rvert_{(x, y, n)}.
	\end{align*}
Now, $\pi_N\circ \pi \circ \pi_1 = \pi_N \circ \pi \circ \pi_2 = \pi_N^{[2]}$, and $\varphi(x)-\varphi(y) = g(x, y)$. Therefore \begin{align*}
	\delta(f) &= g\, (\pi_N\circ \pi\circ \pi_1)^{*}F_{\nabla} \\
	&= g\, (\pi_N^{[2]})^* F_{\nabla}\\
	&= F_{\nabla^{g}}.
	\end{align*} This proves $(3)$. To show $(4)$, recall that the three-curvature $\omega\in \Omega^3\left(M\times N\right)$ is defined uniquely by $df= \pi^{*}\omega.$ Then \begin{align*}
	df&= d\left(-\varphi \, (\pi_N\circ \pi)^{*}F_\nabla\right)\\
	&= -d\varphi \wedge (\pi_N\circ \pi)^{*}F_{\nabla}\\
	&= \pi^{*}\left(-\eta \wedge \pi_N^*F_{\nabla}\right). 
	\end{align*} This gives us the claimed three-curvature in $(4)$, and $(5)$ follows immediately. \end{proof}
We conclude this section by considering Proposition \ref{curvaturetheorem1} for a specific cup product bundle gerbe. In doing so, we hope to solidify the concepts in this section, and also demonstrate how the real Dixmier-Douady class of a general cup product bundle gerbe can be represented by a wedge product of forms. 
\begin{example}
Let $L\to N$ be a line bundle with connection $\nabla$ and two-curvature $F_\nabla$. Consider the surjective submersion $\R\to S^1$ defined as the quotient map by $\Z$. Define $d:\R^{[2]}\to \Z$ by $(x, y)\mapsto x-y$. Let $x\in (0,1)$ and $\theta\in (0, 2\pi)$ denote the standard coordinates on $\R$ and $S^1$ respectively, so $dx = \pi^*(\tfrac{d\theta}{2\pi})$. By Proposition \ref{curvaturetheorem1}, the natural connection on $(L^d, \R\times M, S^1\times M)$ has associated two-curvature $f = -x (\pi_N\circ \pi)^*F_\nabla$, and a representative of the real Dixmier-Douady class of this bundle gerbe is the pullback of the form $$ \frac{1}{2\pi}d\theta\wedge\frac{1}{2\pi i}F_\nabla $$ to $S^1\times N$. We can see that this is a representative of the cup product of the first Chern class of $L\to N$ with the winding class of $id:S^1\to S^1$. 
\end{example}

\section{The $i$-th cup product bundle gerbe} \label{ithbg}
	\indent \indent We will next introduce the $i$-th cup product bundle gerbes, which will be key to defining the cup product bundle gerbe.
%
By pulling back the surjective submersion $\R\to S^1, x \mapsto [x],$ by $p_i:T\to S^1$, we obtain the surjective submersion $p_i^{-1}(\R)\to T$. Here $ p_i^{-1}(\R) \cong \left \{(t, x)\in T \times \R \ \big \rvert \ p_i(t) = [x]\right\}.
$ Then $i$-th cup product bundle gerbe is then defined as follows.
\begin{definition}\label{ex 1}  Define $d_i: p_i^{-1}(\R)^{[2]}\to \Z$ by $(t, x, y)\mapsto y-x$. The \textit{$i$-th cup product bundle gerbe} over $T\times SU(n)/T$ is the general cup product bundle gerbe of $J_i\to \su$ and $d_i$, denoted $$\left(J_i^{d_i}, p_i^{-1}(\R)\times \su\right), $$ and depicted in Figure \ref{fig:cupprodbg21}.
		\begin{align*}
		\begin{array}{ccc}
		J_i^{d_i}    &              \\
		\downarrow &               \\
		p_i^{-1}(\R)^{[2]} \times \su  & \stackrel{}{\dra} \ \ p_i^{-1}(\R)\times \su \\
		&            \ \     \ \ \    \downarrow  \\
		&            \ \ \  \ \      \ \   T   \times \su
		\end{array}
		\end{align*}
		\captionof{figure}{The $i$-th cup product bundle gerbe}
		\label{fig:cupprodbg21} 
\end{definition} 

There is a a natural $SU(n)$-action on $T\times \su$ making the $i$-th cup product bundle gerbe $SU(n)$-equivariant, as we will see in the next proposition. This result will allow us to easily show in the next section that {the} cup product bundle gerbe is itself $SU(n)$-equivariant. 
\begin{proposition}\label{cupprodisequivariant} 
The $i$-th cup product bundle gerbe is $SU(n)$-equivariant for the action of $SU(n)$ on $T\times \su$ given by multiplication in the $\su$ factor. 
\end{proposition}
\begin{proof} This follows easily by $SU(n)$-homogeneity of $J_i\to \su$, and by noting that the $SU(n)$-action on each of the spaces in the bundle gerbe is given by multiplication on the $\su$ factor.
\end{proof}
\indent In the remainder of this section, we will describe the geometry of the $i$-th cup product bundle gerbe. First, we must define a connection $\nabla_{J_i}$ on $J_i\to \su$, and calculate its two-curvature. We will then construct a bundle gerbe connection $\nabla_{c_i}$ on the $i$-th cup product bundle gerbe and calculate its two-curvature using Proposition \ref{existsconnectiononcup}. Defining a curving and calculating the three-curvature of the connective data on the $i$-th cup product bundle gerbe will then be a straightforward application of Proposition \ref{curvaturetheorem1}. Throughout this work we abuse notation and view the line bundles $J_i\to \su$ as the line bundles $K_i\to \text{Proj}_n$ from Section \ref{flagmanifold}.
\begin{proposition}\label{curvaturetwoform}\textup{\textbf{(A connection on $J_i\to \su$)}}
	There is a canonical line bundle connection $\nabla_{{J}_i}$ on ${J}_i\to \su$ with curvature $F_{\nabla_{{J}_i}} = \textup{tr}(P_idP_idP_i)$ for $P_i$ orthogonal projection $\su\times \C^n\to J_i $.
\end{proposition}

\begin{proof} This is merely an application of Example \ref{linebundleconnection}. We need only show that ${J}_i$ is a subbundle of the trivial bundle of rank $n$. This follows by noting that ${J}_i$ is a subbundle of the $SU(n)$-homogeneous vector bundle $({\C^n\times SU(n)})/{T}\to \su$, which is isomorphic to the trivial bundle $\C^n \times \su \to \su$ by Theorem \ref{thrmbundles}. \end{proof}

\begin{remark}\label{remark curvatureofJ} 
	The connection $\nabla_{J_i}$ induces a connection on the pullback of $J_i$ by the projection $T\times \su\to \su$. By Proposition \ref{curvaturetwoform}, the two-curvature of this connection will be $\text{tr}(P_idP_idP_i)$ for $P_i: \su \times T\times \C^n \to J_i$ orthogonal projection. 
\end{remark}
Now that we have defined a line bundle connection on $J_i\to \su$ and computed its two-curvature, we can apply Proposition \ref{existsconnectiononcup} to define a bundle gerbe connection on the $i$-th cup product bundle gerbe.

\begin{proposition}\label{connectiononJ} \textbf{\textup{(A connection on the $i$-th cup product bundle gerbe)}}
	Let $\nabla_{J_i}$ be the connection on $J_i\to \su$ from \textup{Proposition \ref{curvaturetwoform}} with curvature $F_{\nabla_{J_i}}$. Let $$\pi_i^{[2]}: p_i^{-1}(\R)^{[2]}\times \su \to T\times \su$$ be projection. Then there is a bundle gerbe connection $\nabla_{c_i}$ on the $i$-th cup product bundle gerbe with two-curvature $$ F_{\nabla_{c_i}} = d_i\, (\pi_i^{[2]})^{*}\textup{tr}(P_idP_idP_i)$$ for $P_i: T\times \su\times \C^n\to J_i$ orthogonal projection. 
\end{proposition}

Finally, we can present the curving, three-curvature, and real Dixmier-Douady class representative of the $i$-th cup product bundle gerbes using Proposition \ref{curvaturetheorem1}. The next result concludes this section. 

\begin{proposition}\label{connective data on ith cup product bundle gerbe}  \textbf{\textup{(Geometry of the $i$-th cup product bundle gerbe)}}
 	Consider the $i$-th cup product bundle gerbe $\left( J_i^{d_i}, p_i^{-1}(\R)\times \su, T\times \su, \pi_i\right)$ with connection $\nabla_{c_i}$ from \textup{Proposition \ref{connectiononJ}}. Let $P_i:T\times \su\times \C^n\to J_i$ be orthogonal projection, and define a $2$-form $f_{c_i} \in \Omega^2\left(p_i^{-1}(\R)\times \su\right)$ by $$f_{c_i}(t, gT, x) = -x\, \pi_i^* \textup{tr}(P_idP_idP_i).$$ Abuse notation and denote the pullback of $p_i$ to $T\times \su$ by $p_i$. Then 	\begin{enumerate}[(1),font=\upshape]
		\item the $2$-form $f_{c_i}$ satisfies $\delta(f_{c_i}) = F_{\nabla_{c_i}}$, so $f_{c_i}$ is a curving for $\nabla_{c_i};$
		\item the three-curvature $\omega_{c_i} \in \Omega^3\left(T\times\su \right)$ of $(\nabla_{c_i}, f_{c_i})$ is given by $$\omega_{c_i} = -\frac{1}{2\pi i}\, p_i^{-1}dp_i \wedge \textup{tr}(P_idP_idP_i);$$
		\item the real Dixmier-Douady class of the $i$-th cup product bundle gerbe is represented by
		$$\frac{1}{4\pi^2}\, p_i^{-1}dp_i\wedge \textup{tr}(P_idP_idP_i). $$
	\end{enumerate}
\end{proposition}
\begin{proof}
To apply Proposition \ref{curvaturetheorem1}, we must compute the function $\varphi$ and $1$-form $\eta$ for the $i$-th cup product bundle gerbe. Let $\pi: p_i^{-1}(\R)^{[2]}\times \su \to p_i^{-1}(\R)^{[2]}$ be projection. Then  $\delta(x) = \pi^*d_i$, so $x$ is our `$\varphi$'. Since $\pi_i^{*}p_i=e^{2\pi i x}$, $$\pi_i^{-1}\left(\frac{p_i^{-1}dp_i}{2\pi i}\right) = x,$$ so ${p_i^{-1}dp_i}/{2\pi i}$ is our `$\eta$'.  Lastly, by Remark \ref{remark curvatureofJ}, the $\pi_N^*F_\nabla$ term in Proposition \ref{curvaturetheorem1} equals $\text{tr}(P_idP_idP_i)$ for $P_i:\su\times T\times \C^n\to J_i$ orthogonal projection. Inputting this data into Proposition \ref{curvaturetheorem1} gives us the desired result.
\end{proof}
\section{The cup product bundle gerbe} \label{cupprodbgsection}
\indent \indent We are now in a position to define the cup product bundle gerbe, and present the main result of this section, which describes the connective data of this bundle gerbe in terms of orthogonal projections. As we have mentioned, the cup product bundle gerbe is the product of the $i$-th cup product bundle gerbes, which we define formally now.
\begin{definition}\label{definitionofthefinalcupproductbundlegerbe} 
The \textit{cup product bundle gerbe} over $T\times SU(n)/T$ is the tensor product of all $i$-th cup product bundle gerbes (Definition \ref{ex 1}), denoted  $$\left(P_c, X\right) := \bigotimes_{i=1}^n \left( J_i^{d_i}, p_i^{-1}(\R)\times \su\right).$$
\end{definition}
By Propositions \ref{equivariantbundlegerbeproduct} and \ref{cupprodisequivariant}, we immediately obtain our first result for the cup product bundle gerbe.
\begin{proposition}\label{corproposition cup prod bg is equivariant} 
	The cup product bundle gerbe is $SU(n)$-equivariant for the action of $SU(n)$ on $T\times \su$ given by multiplication in the $\su$ factor. 
\end{proposition}
 The following proposition will be critical to our final chapter, realising the cup product product bundle gerbe as a \textit{reduced} product of general cup product bundle gerbes.
\begin{proposition}\label{redproductofcupproduct} There is an $SU(n)$-equivariant stable isomorphism $$\bigotimes_{i=1}^n \left( J_i^{d_i}, p_i^{-1}(\R)\times \su\right) \cong_{SU(n)\textup{-stab}}
	\sideset{}{_\mathrm{red}}\bigotimes_{i=1}^n 
	 \left( J_i^{d_i}, \R^{n-1} \times \su \right).$$
\end{proposition}
\begin{proof}
 It is not difficult to see that there are $SU(n)$-equivariant diffeomorphisms \begin{align*} p_1^{-1}(\R) \times_{T\times \frac{SU(n)}{T} }\times \cdots \times_{T\times \frac{SU(n)}{T}} p_n^{-1}(\R)&\cong \left\{(x_1, ..., x_n, t)\in \R^n \times T \ \big\rvert\ [x_i] = p_i(t)\ \forall \  i \right\}\\ 
&\cong \left\{(x_1, ..., x_n) \in \R^n \ \big\rvert \ \textstyle \sum_{i=1}^n x_i \in \Z\right\} \\ &=:W, \end{align*} using that $p_1(t)\cdots p_n(t) = 1$. Note that the surjective submersion $W\to T\times \su$ is given by $(x_1, ..., x_n, gT) \mapsto (\text{diag}([x_1], ..., [x_n]), gT).$ So $$\bigotimes_{i=1}^n \left( J_i^{d_i}, p_i^{-1}(\R)\times \su\right)\cong \left(\Motimes_{i=1}^n J_i^{d_i}, W\times \su \right).$$  Now, we can consider $\R^{n-1}$ as a submanifold of $W$ since \begin{align}\label{Rniso}\R^{n-1} \cong \left\{(x_1, ..., x_n) \in \R^n \ \big\rvert \  \textstyle\sum_{i=1}^n x_i = 0\right\}.\end{align} Clearly the restriction of the projection $W\times \su \to T\times \su$ to $\R^{n-1}\times \su$ is again a surjective submersion. Therefore by Remark \ref{remarkexample2}, $$\left(\Motimes_{i=1}^n J_i^{d_i}, W\times \su \right) \cong_{\textup{stab}} \left(\Motimes_{i=1}^n J_i^{d_i}, \R^{n-1} \times \su \right).$$ It is trivial to verify that this stable isomorphism is $SU(n)$-equivariant. The result follows by definition of the reduced product (Example \ref{examplereducedproducts}). 
\end{proof}

\begin{remark}\label{defndi}
 We abuse notation and understand $d_i: \left(\R^{n-1}\right)^{[2]}\to \Z$ in the bundle gerbes $\left(J_i^{d_i}, \R^{n-1}\times \su\right)$ to be the maps defined by $$d_i(x_1, ..., x_n, y_1, ..., y_n) =  y_i - x_i.$$ Clearly $\delta(d_i)=0$, so $\left(J_i^{d_i}, \R^{n-1}\times \su\right)$ are indeed general cup product bundle gerbes.
\end{remark}

\indent We next study the connective data on the cup product bundle gerbe, thereby concluding this chapter. Similar to the previous section, we begin by considering a bundle gerbe connection. This result follows immediately from Propositions \ref{product line bundle} and \ref{connectivestructureofproductbg}.

\begin{proposition}\label{jello} \textup{\textbf{(A connection on the cup product bundle gerbe)}} Let $$\pi: (\R^{n-1})^{[2]}\times \su \to T\times \su$$ be projection and $P_i: \su\times T \times \C^n \to {J}_i$ be orthogonal projection. The $i$-th cup product bundle gerbe connections $\nabla_{c_i}$ from \textup{Proposition \ref{connectiononJ}} induce a bundle gerbe connection $\nabla_c$ on  the cup product bundle gerbe with curvature  $$F_{\nabla_c}=  \sum_{i=1}^n d_i\, \pi^{*}\textup{tr}(P_idP_idP_i).$$
\end{proposition}
\begin{proposition}\label{corollary data on cup prod} \textup{\textbf{(Geometry of the cup product bundle gerbe)}}
Let $\nabla_{c}$ be the connection on $\left(P_c, X, T\times \su, \pi_c\right)$ from \textup{Proposition \ref{jello}}. Let  $P_i$ be orthogonal projection $T\times \su\times \C^n\to J_i$, and define a $2$-form $f_c \in \Omega^2\left(X\right)$ by $$ f_c(x_1, ..., x_n, gT) := -\sum_{i=1}^n x_i \, \pi_c^*\textup{tr}(P_idP_idP_i).$$ Abuse notation and denote the pullback of $p_i$ to $T\times \su$ by $p_i$. Then  \begin{enumerate}[(1),font=\upshape]
		\item the $2$-form $f_{c}$ satisfies $\delta(f_c) = F_{\nabla_c}$, so $f_c$ is a curving for the connection $\nabla_c;$
		\item the three-curvature $\omega_c \in \Omega^3\left(T\times \su\right)$ of $(\nabla_{c}, f_c)$ is given by $$\omega_c = -\frac{1}{2\pi i}\sum_{i=1}^n p_i^{-1} dp_i\wedge \textup{tr}(P_idP_idP_i);$$
		\item the real Dixmier-Douady class of the cup product bundle gerbe is represented by $$\frac{1}{4\pi^2}\sum_{i=1}^n p_i^{-1} d p_i\wedge \textup{tr}(P_idP_idP_i).$$ 
	\end{enumerate}
\end{proposition}
\begin{proof}
This follows immediately from Propositions \ref{connectivestructureofproductbg}, \ref{jello}, and \ref{connective data on ith cup product bundle gerbe}.
\end{proof}

%% file: chapter3-5/CHP3-5TEST.tex

	\chapter{The basic bundle gerbe and the Weyl map \label{ch:three}}
	\begin{chapquote}{E. Meinrenken}
	For the special unitary group $G=SU(d+1)$, the construction of the basic gerbe simplifies... in fact, the gerbe is presented as a Chatterjee-Hitchin gerbe... 
	\end{chapquote}
\indent The \textit{basic gerbe} is defined in full generality over a compact, simple, simply connected Lie group $G$. It is defined to be any gerbe whose Dixmier-Douady class is a generator of $H^3(G, \Z)$. There are a number of ways to construct the basic gerbe. The first of these constructions was due to Brylinski in 1993, who  utilised the path fibration $PG\to G$ in his construction \cite[Theorem 5.4.7]{brylinski}. A similar (infinite-dimensional) construction for bundle gerbes which also used path fibrations was provided in 1996 by Murray \cite{1996bundlegerbes}. Here, Murray defined the so-called \textit{tautological bundle gerbe} over a 2-connected manifold $M$, whose curvature was equal to a given integral, closed $3$-form. The special case when $M$ was a compact, simple, simply connected Lie group was later considered in the 1997 paper \cite{higherbgs}, and in more detail in \cite{dannythesis}, to define the basic bundle gerbe. The basic bundle gerbe over $G$ is an example of a \textit{lifting bundle gerbe} (Example \ref{exampleliftingbundlegerbe}). The central extension in this construction is taken to be $$ U(1) \to \widehat{\Omega G} \to \Omega G$$ for $\Omega G$ the group of based, smooth loops in $G$, and $\widehat{\Omega G}$ the Kac-Moody group.\\
\indent In this work, we will restrict our study to the basic bundle gerbe over $SU(n)$. This is advantageous in that there is a much simpler, finite-dimensional construction of this bundle gerbe due to Murray--Stevenson. There were several influential papers leading up to this construction. In 2002, while studying the WZW model, Gawedzki and Reis provided a representation-theoretic construction of a gerbe over $SU(n)$ \cite{branesandgerbes}. Meinrenken extended the work of Gawedzki and Reis in \cite{Meinrenken} to describe a finite-dimensional construction of an equivariant gerbe over $G$. In this paper, Meinrenken (and later Mickelsson \cite{mickelsson03}) considered an explicit example of a local bundle gerbe over $SU(n)$. In 2008, Murray--Stevenson generalised Meinrenken and Mickelsson's constructions to describe the aforementioned finite-dimensional construction of the globally defined \text{basic bundle gerbe} over $SU(n)$ \cite{53UNITARY}. In this work, Murray--Stevenson explicitly described a connection and curving on the basic bundle gerbe over $SU(n)$, and showed its three-curvature was $$-\frac{1}{24\pi^2}\, \text{tr}(g^{-1}dg)^3,$$the basic $3$-form on $SU(n)$. These results will be used extensively in our work. \\
\indent We begin this chapter with a brief study of the Weyl map $p:T\times {G}/{T}\to G$ (Section \ref{c2}), with an emphasis on the case when $G=SU(n)$. In Section \ref{section: definition of basic bg}, we define the basic bundle gerbe over $SU(n)$ following the construction in \cite{53UNITARY}. We show this bundle gerbe is $SU(n)$-equivariant (Proposition \ref{proposition basic bg is equivariant}), and present its connective data and three-curvature (Proposition \ref{curvature of basic bg}, Theorem \ref{theorem basic bundle gerbe data}). The pullback of the basic bundle gerbe by $p$ will then be constructed in Section \ref{c3}. Here, the description of the Weyl map on $\text{Proj}_n$ from Section \ref{c2} will be used to describe the induced connective data on the pullback bundle gerbe in terms of orthogonal projections (Propositions \ref{proposition connective structure of pullback}, \ref{curvature}). Doing so will enable us to easily compare the connective data on the cup and pullback basic bundle gerbes later. Finally, we present a series of computations in Section \ref{subsection stable isos}, culminating in Proposition \ref{pullbackstableisomainresult}. This result describes the pullback of the basic bundle gerbe by the Weyl map as a reduced product of general cup product bundle gerbes, and will be crucial to the proof of this thesis' main result.

	\section{The Weyl map} \label{c2}
	\indent \indent The Weyl map $p:T\times {G}/{T} \to G, (t, gT)\mapsto gtg^{-1}$ is defined for $G$ a compact, connected Lie group and $T$ a \textit{maximal torus} of $G$. This map first appeared in the fundamental works of H. Weyl (1925-1926) \cite{weyl1, weyl2, weyl3} to prove the \textit{Weyl Integral Formula}, which describes the integral of a continuous function over $G$ as a product of integrals over ${G}/{T}$ and $T$. This formula has a variety of applications in representation theory. Most notably, it was used to prove \textit{Weyl's character formula}, a famous formula which expresses the character of an irreducible representation in terms of its highest weight. We refer the reader to \cite{dieck, rossmann} for more on these topics. The Weyl map would continue to appear in proofs throughout the 20-th century. Notably, in 1965, M. F. Atiyah made use of the Weyl map to provide an alternative proof of a $K$-theoretic result on compact Lie groups due to L. Hodgkin \cite{atiyahusingweylmap}. These are just a few of the many applications of the Weyl map in Lie theory.\\
	\indent	In this section we briefly introduce maximal tori and formally define the Weyl map. A geometric interpretation of the Weyl map when $G=SU(n)$ will also be considered. We conclude this section by considering the image of the Weyl map on $T\times \text{Proj}_n\cong T\times \su$, which will be used implicitly to describe the connective data on the pullback of the basic bundle gerbe later. Although the remainder of this chapter requires little more than the definition of the Weyl map, we encourage the reader to explore this rich topic further in \cite{ADAMS, rossmann, repfinitecompactgroups}. Let us begin with the definition of tori in Lie groups.
	\begin{definition}
		Let $G$ be a compact connected Lie group.\begin{enumerate}[(1)]
			\item A \textit{torus} of $G$ is a subgroup $T$ of $G$ that is isomorphic to a product of $U(1)$ factors.
			\item A torus $T$ is called a \textit{maximal torus} of $G$ if there is no other torus in $G$ properly containing $T$. 
		\end{enumerate}  
	\end{definition}
	\begin{remark}[\cite{ADAMS}]
		A maximal torus of a compact connected Lie group $G$ is a maximal abelian subgroup of $G$. 
	\end{remark}
	\begin{example}
		Let $G= SU(n)$ and $T$ be the subgroup of $SU(n)$ consisting of all diagonal matrices. Clearly $T \cong U(1)^{n-1}$ and $T$ is a \textit{maximal} torus of $SU(n)$ (see \cite{ADAMS}). 
	\end{example}
	With this, we can define the Weyl map, and consider its image on $T\times \su$ for $T$ the group of diagonal matrices in $SU(n)$. As we have mentioned, the Weyl map in this case has a highly geometric interpretation. For more on this, we recommend \cite{rossmann}. 
	
	\begin{definition}
		Let $G$ be a compact, connected Lie group and $T$ be a maximal torus of $G$. Define the \textit{Weyl map} by \begin{align*}
		p: T\times {G}/{T} &\to G \\
		(t, gT)&\mapsto gtg^{-1}.
		\end{align*}
	\end{definition}
	\begin{remark}
		It is easy to verify the Weyl map is well-defined using the fact that $T$ is abelian. 
	
	\end{remark} 
\begin{remark}\label{weylmapisequiv}
The Weyl map is $G$-equivariant with respect to the $G$-action on $G/T$ defined by left multiplication and the $G$-action on $G$ defined by conjugation. To see this, let $t\in T$ and $g, h\in G$. Then $$p(t,hgT) = hgtg^{-1}h^{-1} = h p(t, gT) h^{-1}.$$
\end{remark}
	\begin{example}
		When $G= SU(n)$ and $T<SU(n)$ is the subgroup of diagonal matrices, the Weyl map has several interpretations. For instance, the image of the Weyl map on $\{t\}\times SU(n)/T$ for a fixed $t\in T$ is $$p(\{t\}\times SU(n)/T) = \{gtg^{-1} \ | \ g\in SU(n)\},$$ the set of all elements in $SU(n)$ whose eigenvalues are the diagonal entries of $t$. Equivalently, $p(\{t\}\times SU(n)/T)$ is the orbit of $t$ under the conjugation action in $SU(n)$. We can alternatively understand the restriction of the Weyl map to be a covering space over the dense open subset $SU(n)_{\text{reg}}\subset SU(n)$ consisting of matrices whose eigenvalues are distinct (see \cite{53UNITARY}).
	\end{example}
	\begin{remark}
		Throughout the remainder of this work, any use of the term `Weyl map' shall be referring the the Weyl map on $T\times \su$ for $T$ the maximal torus of $SU(n)$ consisting of diagonal matrices. 
	\end{remark}
\indent As we have seen in the previous chapter, the curving and connection on the cup product bundle gerbe can be described by orthogonal projections using that $\su \cong \text{Proj}_n$, where $\text{Proj}_n$ is the space of $n$-tuples of orthogonal projections with mutually orthogonal one-dimensional image. We wish to describe the connective data on the basic bundle gerbe and pullback of the basic bundle gerbe by the Weyl map similarly. To do so, we must understand the image of the Weyl map on $T\times \text{Proj}_n\cong T\times \su$. This is given to us in the next proposition, which concludes this section. For $i=1, ..., n$, let $p_i:T\to S^1$ be the homomorphism mapping $t\in T$ to its $i$-th diagonal (Definition \ref{salad}).
	\begin{proposition}\label{WEyl map on projections}
		Under the identification $\su \cong \text{Proj}_n$ in \textup{Proposition \ref{SU(n)/T is Proj_n}}, the Weyl map $p:T\times \text{Proj}_n\to SU(n)$ is given by \begin{align} \label{weylmapdescription} p: (t, P_1, ..., P_n)\mapsto \sum_{i=1}^np_i(t)P_i. \end{align}
	\end{proposition}
	\begin{proof}
		Let $(t, gT)\in T\times \su$. Under the diffeomorphism in Proposition \ref{SU(n)/T is Proj_n}, $(t, gT)$ maps to $(t, gO_1g^{-1}, ..., gO_ng^{-1})\in T\times \text{Proj}_n$. Since $O_i$ is the matrix with a $1$ in the $i$-th diagonal and zeros elsewhere, $t = \sum_{i=1}^n p_i(t)O_i$ and  $$p(t, gT) = gtg^{-1} = g\left(\sum_{i=1}^n p_i(t)O_i\right)g^{-1} = \sum_{i=1}^n p_i(t) \left(gO_ig^{-1}\right).$$ Therefore $p(t, gO_1g^{-1}, ..., gO_ng^{-1}) = \sum_{i=1}^n p_i(t) gO_ig^{-1}$. Since the action of $SU(n)$ on $\text{Proj}_n$ is transitive, any element $(P_1, ..., P_n)\in \text{Proj}_n$ can be written as $(gO_1g^{-1}, ..., gO_ng^{-1})$ for some $g\in SU(n)$, and the result follows.
	\end{proof}

	\section{The basic bundle gerbe}\label{section: definition of basic bg} 
	\indent \indent In this section, we will construct the basic bundle gerbe on $SU(n)$, prove it is $SU(n)$-equivariant, and study its connective data. Our construction of the basic bundle gerbe $\basicbgnopi$ follows closely that in \cite{53UNITARY}. First, the space $Y$ will be introduced. We will then consider an ordering on $Y^{[2]}$, which will provide us with a straightforward geometric description of the disconnected sets $Y_-^{[2]}, Y_0^{[2]}$ and $Y_+^{[2]}$ that cover $Y^{[2]}$. Finally, the line bundle $P_b\to Y^{[2]}$ will be defined over these components as the determinant of a vector bundle.\\
	\indent The space $Y$ is defined as follows. Let $Z := U(1)\backslash \{1\}$. This is an open subset of $U(1)$, hence a manifold.  Denote the set of eigenvalues of an operator $g\in SU(n)$ by spec$(g)$. Define $$Y:= \left\{(z, g) \in Z\times SU(n) \ | \ z\notin \text{spec}(g)\right\}.$$ This is an open subset of $Z\times SU(n)$, hence is a manifold. Let $\pi: Y\to SU(n)$ be projection onto the second factor. This is a surjective submersion, and clearly $$Y^{[2]} \cong \left\{(z_1, z_2, g) \in Z^2 \times SU(n) \  | \ z_1,z_2 \notin \text{spec}(g) \right\}.$$ More generally, $Y^{[p]}$ consists of $(p+1)$-tuples  $(z_1, ..., z_p, g)$ such that $z_1, ..., z_n \notin \text{spec}(g)$. \\ 
	\indent To define the line bundle $\Pbasic \to Y^{[2]}$, we must understand the aforementioned disconnected cover of $Y^{[2]}$. These sets can be described explicitly using the ordering on $Z$ induced from the natural ordering of elements in $(0, 2\pi)$. Namely, if $\text{arg}: Z \to (0, 2\pi)$ is the bijection satisfying $\text{exp}(i\text{arg}(z)) = z$ for all $z\in Z$, we can define an ordering on $Z$ by setting $z_1 < z_2$ if and only if $\text{arg}(z_1) < \text{arg}(z_2)$. The set in our cover containing $(z_1, z_2, g)\in Y^{[2]}$ will then be determined by the positions of the eigenvalues of $g$ relative to $z_1$ and $z_2$ with respect to this ordering. To better explain this, we introduce some terminology.

	\begin{definition}
		Let $(z_1, z_2, g)\in Y^{[2]}$ and $\lambda$ be an eigenvalue of $g$. Say that $\lambda \in Z$ is \textit{between} $z_1$ and $z_2$ if $z_1 < \lambda < z_2$ or $z_2 < \lambda < z_1$ (or, equivalently, if $\lambda$ is in the connected component of $U(1)\backslash\{z_1, z_2\}$ not containing $\{1\}$). Call a triple $(z_1, z_2, g)\in Y^{[2]}$ \textit{positive} if there exist eigenvalues of $g$ between $z_1 > z_2$, \textit{null} if there are no eigenvalues of $g$ between $z_1$ and $z_2$, and \textit{negative} if there exist eigenvalues of $g$ between $z_1<z_2$. 
	\end{definition}
Denote the set of all positive, null, and negative triplets in $Y^{[2]}$ by $Y^{[2]}_+, Y^{[2]}_0$ and $Y^{[2]}_{-}$. Note that $(z_1, z_2, g)\in Y_+^{[2]}$ if and only if $(z_2, z_1, g)\in Y_-^{[2]}$. Elements in each of these sets are depicted in Figure \ref{fig:circleimage2}, where we assume for simplicity that all eigenvalues of $g$ are in the connected component of $Z\backslash \{z_1, z_2\}$ containing $\lambda$. 
	
	\begin{figure} \centering 
		\begin{tikzpicture}[scale=0.8]  
		\draw (0, 0) circle (1.5 cm);
		\draw[fill = black] (1, 1.118) circle (2pt);
		\draw[fill = black] (0, -1.5) circle (2pt);
		\draw[fill = white, draw = black] (1.5, 0) circle (2pt);
		\node[] at (1.28, 1.3) {$\ z_1$};
		\node[] at (0, -1.85) {$z_2$};
		\draw[fill = black] (-1, 1.118) circle (2pt);
		\node[] at (-1.2, 1.35) {$\lambda \ \  $};
		\node[] at (0, -2.7) {$(z_1, z_2, g) \in Y_-^{[2]}$};

		\draw (6, 0) circle (1.5 cm);
		\draw[fill = black] (7, 1.118) circle (2pt);
		\draw[fill = black] (4.585, -.5) circle (2pt);
		\draw[fill = black] (5, 1.118) circle (2pt);
		\draw[fill = white, draw = black] (7.5, 0) circle (2pt);
		\node[] at (7.28, 1.3) {$\ z_2$};
		\node[] at (4.3, -.6) {$\lambda$};
		\node[] at (4.8, 1.35) {$z_1\ \ $};
		\node[] at (1, 1) {};
		\node[] at (6, -2.7) {$(z_1, z_2, g) \in Y_0^{[2]}$};
		
		\draw (12, 0) circle (1.5 cm);
		\draw[fill = black] (12, -1.5) circle (2pt);
		\draw[fill = black] (12, 1.5) circle (2pt);
		\draw[fill = black] (12.9, 1.2) circle (2pt);
		\draw[fill = white, draw = black] (13.5, 0) circle (2pt);
		\node[] at (12, 1.85) {$\lambda $};
		\node[] at (12, -1.85) {$\ z_1$};
		\node[] at (13.17, 1.37)  {$\ z_2$};
		\node[] at (1, 1) {};
		\node[] at (12, -2.7) {$(z_1, z_2, g) \in Y_+^{[2]}$};
		\end{tikzpicture}	\captionof{figure}{Components of $Y^{[2]}$ }
		\label{fig:circleimage2} \end{figure}
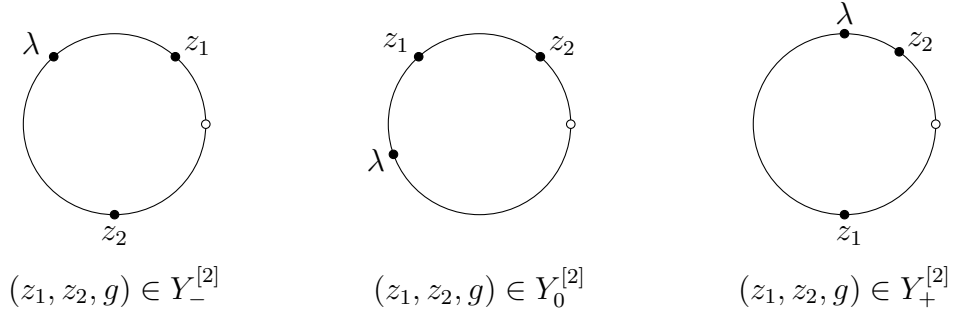

	\indent With this description of $Y^{[2]}$, we can define $P_b\to Y^{[2]}$ as follows. For $\lambda$ an eigenvalue of $g\in SU(n)$, let $E_{(g, \lambda)}$ denote the $\lambda$-eigenspace of $g$. Define the vector bundle $L\to Y^{[2]}_{+}$ fibrewise by $$L_{(z_1, z_2, g)} = \Moplus_{z_1>\lambda > z_2} E_{(g, \lambda)}.$$ 
	For a proof that this is indeed a vector bundle we refer the reader to \cite{53UNITARY}. Note that $L_{(z_1, z_2, g)}$ has finite dimension as a finite sum of finite-dimensional spaces. Therefore we can define $\Pbasic \to Y^{[2]}$ fibrewise by $$(\Pbasic)_{(z_1, z_2, g)} = \begin{cases} \text{det}(L_{(z_1, z_2, g)}) &\text{if } (z_1, z_2, g)\in Y^{[2]}_+\\ 
	\C&\text{if } (z_1, z_2, g)\in Y^{[2]}_0 \\
	\text{det}(L_{(z_2, z_1, g)})^* & \text{if } (z_1, z_2, g) \in Y^{[2]}_{-}. \end{cases} $$
	By \cite{53UNITARY}, $\Pbasic \to Y^{[2]}$ is a smooth locally trivial hermitian line bundle.
	
\begin{remark}
The realisation of $P_b$ as a determinant line bundle (Example \ref{determinant line bundle }) was our primary motivation to define bundle gerbes using line bundles rather than principal $U(1)$-bundles (Remark \ref{principalbundle}).
\end{remark}
	 It remains to be shown that there is an associative multiplication operation that will endow $(P_b, Y, SU(n))$ with a bundle gerbe structure. This is provided to us in the next proposition. 
	\begin{proposition}\label{teawoke}
		There exists an associative multiplication operation on $P_b\to Y^{[2]}$ making $(P_b, Y, SU(n))$ a bundle gerbe. 
	\end{proposition}
	\begin{proof} Let $(z_1, z_2, z_3, g) \in Y^{[3]}$ with $z_1>z_2>z_3$. Assume there are eigenvalues of $g$ between $z_1$ and $z_2$ and also between $z_2$ and $z_3$. Clearly $$\Moplus_{z_1 >\lambda > z_2 } E_{(g, \lambda)} \ \oplus \Moplus_{z_2>\lambda > z_3 } E_{(g, \lambda)} = \Moplus_{z_1 >\lambda > z_3 } E_{(g, \lambda)},$$ so for this choice of $(z_1, z_2, z_3, g)\in Y^{[3]}$ the fibres of $L$ satisfy \begin{align}\label{eqn3} L_{(z_1, z_2, g)}\oplus L_{(z_2, z_3, g)} = L_{(z_1, z_3, g)}.\end{align} It is a standard fact that, for vector spaces $V$ and $W$, the wedge product defines a canonical isomorphism $\text{det}(V)\otimes \text{det}(W)\to \text{det}(V\oplus W)$. Therefore the wedge product defines a multiplication map $$(\Pbasic)_{(z_1, z_2, g)}\otimes (\Pbasic)_{(z_2, z_3, g)}= (\Pbasic)_{(z_1, z_3, g)}$$ for this choice of $(z_1, z_2, z_3, g)$. By considering the other connected components of $Y^{[3]}$, we see that this equation holds for all $(z_1, z_2, z_3, g) \in Y^{[3]}$ (see \cite{53UNITARY} for details). Associativity of this multiplication follows by associativity of the wedge product.
	\end{proof}
	\begin{definition}
		Call the bundle gerbe $\basicbg$ constructed above the \textit{basic bundle gerbe over $SU(n)$}, or simply the \textit{basic bundle gerbe}, depicted in Figure \ref{figbbg}. 	\begin{figure}[bth!]
			$$ \begin{array}{ccc}
			P_b    &          &        \\
			\downarrow &         &         \\
			Y^{[2]}& \stackrel{}{\dra} &  Y\\
			&               &      \downarrow \\
			&               &      SU(n)   \\
			\end{array} $$
			
			\captionof{figure}{The basic bundle gerbe}
			\label{figbbg}
		\end{figure}
	\end{definition}
\begin{remark}
The construction of the basic bundle gerbe due to Murray--Stevenson in \cite{53UNITARY} that we have replicated here is defined more generally over a group $G$ equipped with a unitary action on a Hilbert space. The possible choices for $G$ include the group $U(H)$ of unitary operators on a finite dimensional Hilbert space $H$, the diagonal torus $T\subset U(n)$, and the Banach Lie groups $U_p(H)$ for $H$ an infinite dimensional separable complex Hilbert space. The limitation of the basic bundle gerbe construction to these groups is a consequence of the space $Y$ being defined in terms of eigenvalues of unitary operators. This choice of $Y$ by Murray--Stevenson in \cite{53UNITARY} was based on a space used in Meinrenken and Mickelsson's construction of local bundle gerbe over $SU(n)$ in \cite{Meinrenken, mickelsson03}. Although Murray--Stevenson's construction relies on $G$ being unitary, it is advantageous in that it is finite-dimensional, unlike the construction of the basic bundle gerbe over an arbitrary compact, simple, simply connected Lie group $G$ as in \cite{dannythesis}.
\end{remark}

	\indent We now begin to explore the properties of this bundle gerbe. By noting that $SU(n)$ acts on itself by conjugation, it is natural to ask if the basic bundle gerbe over $SU(n)$ is itself $SU(n)$-equivariant (Definition \ref{equivariantbundlegerbedefn}). This is indeed the case, leading us to our next proposition.
	\begin{proposition}\label{proposition basic bg is equivariant} 
		The basic bundle gerbe $\basicbg$ is an $SU(n)$-equivariant bundle gerbe for the action of $SU(n)$ on itself by conjugation. 
	\end{proposition}
	\begin{proof} We provide an outline of this proof and refer the reader to \cite{53UNITARY} for details. First, we see that the conjugation action of $SU(n)$ on itself lifts to an action on $Y$ defined by $g \cdot(z, s) :=
		(z, gsg^{-1})$. This is well-defined since $\text{spec}(s) = \text{spec}(gsg^{-1})$. Clearly $\pi:Y\to M$ is $SU(n)$-equivariant, so we need only verify that $\Pbasic \to Y^{[2]}$ is an $SU(n)$-equivariant line bundle. We show this is true for the restriction of $\Pbasic\to Y^{[2]}$ to $Y^{[2]}_+$ and claim the other cases proceed similarly. Let $(z_1, z_2, g) \in Y_+^{[2]}$ and $h\in SU(n)$. Consider an eigenvector $v\in L_{(z_1, z_2, g)}$ with eigenvalue $\lambda$. Then $hv$ is an eigenvector of $hgh^{-1}$ with eigenvalue $\lambda$, so multiplication by $h$ defines an action $L_{(z_1, z_2, g)}\to L_{(z_1, z_2, hgh^{-1})}$. Since $\Pbasic|_{Y_+^{[2]}} = \text{det}(L)$, this induces an action on $\Pbasic|_{Y_+^{[2]}}$ which is linear on fibres. It can be verified that the line bundle projection is equivariant, so this is an $SU(n)$-equivariant line bundle.
	\end{proof}
	\indent The remainder of this section will focus on the geometry of the basic bundle gerbe. This work will allow us to describe the induced geometry on the pullback of the basic bundle gerbe later in the chapter. As a first observation, note that a vector bundle connection $\nabla$ on $L$ will induce a line bundle connection det$(\nabla)$ on $\Pbasic$ (Example \ref{determinant line bundle }). We claim that there is a choice of $\nabla$ such that this induced connection on $P_b$ is a bundle gerbe connection on $\basicbgnopi$. Let $M_n$ be the space of $n\times n$ matrices. 
	
	\begin{proposition}[{\cite[p.$\,$1577]{53UNITARY}}]  \label{curvature of basic bg} \textbf{\textup{(A connection on the basic bundle gerbe)}}
	Consider the orthogonal projection $P: Y_+^{[2]}\to M_n$ associated to the line bundle $L\to Y^{[2]}_+$. There exists a connection $\nabla$ on $L \to Y^{[2]}_+$ that induces a bundle gerbe connection $\nablabasic := \textup{det}(\nabla)$ on the basic bundle gerbe $(\Pbasic, Y, SU(n))$. The two-curvature $F_{\nabla_b}$ of this connection satisfies $$F_{\nabla_b}\big|_{Y_+^{[2]}} = \textup{tr}(PdPdP), \ F_{\nabla_b}\big|_{Y_-^{[2]}} = -\textup{tr}(PdPdP),\ F_{\nabla_b}\big|_{Y_0^{[2]}} = 0.$$
	\end{proposition}
	
	\begin{proof}
		Let	$\nabla$ be the connection on $L\subset \C^n\times Y_+^{[2]}$ induced by the trivial connection. That is, $\nabla := P\circ d$ for $P:\C^n\to L$ orthogonal projection (Example \ref{remarkonprojectionconnection}). This gives rise to a connection det$(\nabla)|_{Y_+^{[2]}}$ on $\Pbasic|_{Y_+^{[2]}}$ (Example \ref{determinant line bundle }). The induced dual and trivial connections define connections on $\Pbasic|_{Y_-^{[2]}}$ and $\Pbasic|_{Y_0^{[2]}}$ respectively, thereby defining a global connection $\nablabasic := \text{det}(\nabla)$. We claim this is a bundle gerbe connection on $(\Pbasic, Y, SU(n))$. To see this, consider a connected component of $U\subset Y^{[3]}$ containing some $(z_1, z_2, z_3, g)$ with $z_1>z_2>z_3$. Further assume there exist eigenvalues of $g$ between $z_1>z_2$ and $z_2>z_3$. By equation (\ref{eqn3}), 
		$$\pi_2^{-1}(L) = \pi_3^{-1}(L)\oplus \pi_1^{-1}(L)$$ over $U$. Since this is an orthogonal splitting and $\nabla$ is defined by orthogonal projection, over $U$ we have $$\pi_2^* \nabla = \pi_3^*\nabla \oplus \pi_1^*\nabla.$$ Hence the wedge product induces an isomorphism $$\pi_2^* \nablabasic = \pi_3^* \nablabasic \otimes id_1 + id_3 \otimes \pi_1^* \nablabasic$$ over $U$, where $id_1, id_3$ denote the identity maps on $\pi_1^{-1}(L)$ and $\pi_3^{-1}(L)$, respectively. Similar arguments applied to the other connected components show this equation holds over all of $Y^{[3]}$. Therefore $\nablabasic$ is a bundle gerbe connection on $\basicbgnopi$. By Proposition \ref{linebundleconnection} and Corollary \ref{linebundlecurvature}, $F_{\nablabasic} = \text{tr}(PdPdP)$ on $Y^{[2]}_+$. By similar arguments applied to the dual and trivial connections on $\Pbasic$, the result follows.
	\end{proof}
	We conclude our preliminary work on the basic bundle gerbe by presenting its curving and curvature, as calculated in \cite[Theorem 5.1]{53UNITARY}. We omit the proof here as it is highly non-trivial and requires results from holomorphic functional calculus. This result is provided only for completeness, and will not be referenced again. It is significant because it allowed the authors of \cite{53UNITARY} to compute the curving and curvature on the pullback of the basic bundle gerbe in a very explicit form, as we will see later. \\ 
	\indent For $z\in \C^*$, define $R_z$ to be the closed ray from the  origin through $z$. Define a branch of the logarithm $\text{log}_z:\C\backslash{R_z}\to \C$ by setting $\text{log}_z(1) = 0$. For $(z, g)\in Y$, let $C(z, g)$ be an anti-clockwise closed contour in $\C\backslash R_z$ enclosing spec$(g)$.  Let $\nabla_b$ be the connection from Proposition \ref{curvature of basic bg}.

	\begin{theorem}[{\cite[Theorem 5.1]{53UNITARY}}]  \label{theorem basic bundle gerbe data} \textbf{\textup{(Geometry of the basic bundle gerbe)}}
		Let $\nabla_b$ be the connection from \textup{Proposition \ref{curvature of basic bg}}. Define a $2$-form $\fbasic$ on $Y$ by the contour integral 
		$$\fbasic(z, g) := \frac{1}{8 \pi^2} \int_{C_{(z,g)}} \textup{log}_z\zeta \textup{tr}\left((\zeta-g)^{-1}dg(\zeta-g)^{-2}dg\right)d\zeta. $$ Then \begin{enumerate}[(1),font=\upshape]
			\item the $2$-form $\fbasic$ satisfies $\delta(\fbasic) = F_{\nablabasic}$, so $\fbasic$ is a curving for $\nablabasic;$
			\item the three-curvature $\omega_b \in \Omega^3(SU(n))$ of $(\nablabasic, f_b)$ is given by $$\omegabasic = -\frac{i}{12\pi}\textup{tr}(g^{-1}dg)^3 $$ for $g\in SU(n)$ the standard coordinate chart\,$;$
			\item the real Dixmier-Douady class of the basic bundle gerbe is represented by $$\nubasic= -\frac{1}{24\pi^2}\textup{tr}(g^{-1}dg)^3. $$ 
		\end{enumerate}
	\end{theorem}

\begin{remark}
Recall that there is a notion of \textit{equivariant connective structures} on equivariant bundle gerbes (Remark \ref{remarkonequiariantconnectivestructure}). It is shown in \cite[Theorem 5.2]{equivariantbundlegerbes} that the connective data on the basic bundle gerbe defined above is equivariant. We will see that the cup product bundle gerbe and pullback of the basic bundle gerbe by the Weyl map are not $D$-stably isomorphic, hence cannot be isomorphic as bundle gerbes with equivariant data. For this reason, we do not explore this concept further. 
\end{remark}
	\section{The pullback of the basic bundle gerbe by the Weyl map} \label{c3}
	\indent \indent This is the penultimate section of this chapter. Here, we will describe the pullback of the basic bundle gerbe by the Weyl map and consider its induced bundle gerbe connection using Proposition \ref{curvature of basic bg}. The connective data associated to this induced connection will then be presented in Proposition \ref{curvature} using results from \cite{53UNITARY}. Let us begin by defining the pullback of the basic bundle gerbe by the Weyl map.
	
	\begin{definition}\label{definition of pullback of bbg} 
		Consider the Weyl map $p:T\times \su\to SU(n)$. The \textit{pullback of the basic bundle gerbe by the Weyl map} $p$, depicted in Figure \ref{figbbgpullback}, is the bundle gerbe $$\pullbackbasicbg := \left(	(\hat{p}^{[2]})^{-1}(\Pbasic), p^{-1}(Y), T\times \su\right)$$ where $\hat{p}:p^{-1}(Y) \to Y$ is the canonical projection.
	\end{definition}
\begin{figure}
	$$ \begin{array}{ccc}
	(\hat{p}^{[2]})^{-1}(\Pbasic)      &          &        \\
	\downarrow &         &         \\
	p^{-1}(Y)^{[2]}& \stackrel{}{\dra} &  p^{-1}(Y)\\
	&               &      \downarrow \\
	&               &       T\times \su    \\
	\end{array} $$
	\captionof{figure}{Pullback of the basic bundle gerbe}
	\label{figbbgpullback}\end{figure}
	\begin{remark}\label{pullbackofbasicbundlegerbeisequivariant} 
		By Propositions \ref{proposition pullback equivariant bg is equivariant}, \ref{proposition basic bg is equivariant} and Remark \ref{weylmapisequiv}, the pullback of the basic bundle gerbe by the Weyl map is $SU(n)$-equivariant with respect to the action of $SU(n)$ on $T\times \su$ given by multiplication in the $\su$ component. 
	\end{remark}
	\begin{remark}\label{remark spaces in pullback bundle gerbe} 
	The spaces present in the pullback of the basic bundle gerbe can be described as follows.
	\begin{align} p^{-1}(Y) &\cong \left\{(gT, t, z) \in \su \times T \times Z\ | \ z\notin \text{spec}(t)\right\}  \label{eq:test} \\ 
	p^{-1}(Y)^{[2]} &\cong \left\{(gT, t, z_1, z_2)\in \su \times T \times Z^2 \ | \ z_1, z_2\notin \text{spec}(t)\right\}  \label{eq:test2} \\
	\LL\hat{p}^{[2]}\RR^{-1}(\Pbasic) &\cong \left \{ (gT, t, z_1, z_2, v) \in p^{-1}(Y)^{[2]}\times \Pbasic \ | \ v \in (\Pbasic)_{(z_1, z_2, gtg^{-1})}\right\}  \label{eq:test3} 
	\end{align} 
\end{remark}

	\indent Naturally, the connective data on the basic bundle gerbe induces connective data on the pullback. Let $(\nabla_b, f_b)$ be the connective data on the basic bundle gerbe from Proposition \ref{curvature of basic bg} and $\hat{p}:p^{-1}(Y)\to Y$ be the map from Definition \ref{definition of pullback of bbg}. Recall that $M_n$ is the space of $n\times n$ matrices. 
	%
	
	\begin{proposition}\label{proposition connective structure of pullback} \textbf{\textup{(A connection on the pullback of the basic bundle gerbe)}} 
		Let  $P: p^{-1}(Y^{[2]}_+)\to \text{M}_n$ be orthogonal projection defined as the pullback by $p$ of the orthogonal projection $Y^{[2]}_+\to \text{M}_n$ associated to the line bundle $L\to Y^{[2]}_+$. The pullback of the connective data on the basic bundle gerbe $(\pnablabasic, \pfbasic):= (\hat{p}, p)^*(\nablabasic, \fbasic)$ is connective data on $\pullbackbasicbg$. Moreover, the two-curvature  of $\nabla_{p^*b}$ satisfies $$F_{\pnablabasic}\big |_{p^{-1}\left(Y^{[2]}_+\right)}  = \textup{tr}(PdPdP), \ F_{\pnablabasic}\big |_{p^{-1}\left(Y^{[2]}_-\right)}  = -\textup{tr}(PdPdP), \ F_{\pnablabasic}\big |_{p^{-1}\left(Y^{[2]}_0\right)}  = 0.$$
		%
	\end{proposition}
	\begin{proof}
		It follows from Lemma \ref{curving of pullback}  that $(\pnablabasic, \pfbasic)$ is connective data on the pullback bundle gerbe. The result about the two-curvature follows by Proposition \ref{curvature of basic bg}.
	\end{proof}
	
	\indent The pullback of the connective data from Theorem \ref{theorem basic bundle gerbe data} defines connective data for the pullback connection. 
	In the next proposition, we present this data in terms of orthogonal projections, using implicitly Proposition \ref{WEyl map on projections} and Theorem \ref{theorem basic bundle gerbe data}. This proposition, together with some remarks, concludes this section. 

	\indent Recall that, for $i=1, ..., n$, $p_i: T\to S^1$ is projection onto the $i$-th diagonal. Let $P_i: T\times  \su \to M_n$ be orthogonal projection. Abuse notation and denote the pullback of these maps to $p^{-1}(Y)$ by $p_i$ and $P_i$. Let $\nabla_{p^*b}$ be the bundle gerbe connection from Proposition \ref{proposition connective structure of pullback}.
	
	\begin{proposition}[{\cite[Appendix B]{53UNITARY}}] \label{curvature} \textbf{\textup{(Geometry of the pullback of the basic bundle gerbe)}} Let $\nabla_{p^*b}$ be the connection from \textup{Proposition \ref{proposition connective structure of pullback}}. Define $f_{p^*b}\in \Omega^2(p^{-1}(Y))$ by 
		$$f_{p^*b} := \frac{i}{4\pi} \sum_{i\neq k} (\log_z p_i - \log_z p_k + (p_k - p_i)p_k^{-1})\textup{tr}(P_i dP_kdP_k).$$ Then \begin{enumerate}[(1),font=\upshape]
			\item the $2$-form $f_{p^*b}$ satisfies $\delta(f_{p^*b}) = F_{\nabla_{p^*b}}$, so $f_{p^*b}$ is a curving for $\nabla_{p^*b};$
			\item the three-curvature $\omega_{p^*b} \in \Omega^3\left(T\times \su \right)$ of $(\nabla_{p^*b}, f_{p^*b})$ is given by $$ \displayindent0pt
			\displaywidth\textwidth  \ \ \ \ \frac{i}{4\pi} \sum_{\substack{i,k=1 \\ i\neq k}}^n \left(p_i^{-1} dp_i - p_k^{-1}dp_k -p_k^{-1}dp_i + p_k^{-1}dp_k p_k^{-1}p_i\right) \textup{tr}(P_idP_kdP_k) - p_ip_k^{-1} \textup{tr}(dP_idP_kdP_k); $$
			\item a representative $\nu_{p^*b}$ of the real Dixmier-Douady class of the pullback of the basic bundle gerbe is given by $$\displayindent0pt
			\displaywidth\textwidth  \ \ \ \ \frac{1}{8\pi^2} \sum_{\substack{i, k=1 \\ i\neq k}}^n \left(p_i^{-1} dp_i - p_k^{-1}dp_k -p_k^{-1}dp_i + p_k^{-1}dp_k p_k^{-1}p_i \right) \textup{tr}(P_idP_kdP_k)-  p_ip_k^{-1}\textup{tr}(dP_idP_kdP_k).$$ 
		\end{enumerate}
	\end{proposition} 
	\begin{remark}
	
		The reader will recall that the curving in Theorem \ref{theorem basic bundle gerbe data} was described in terms of an integral over a curve. To derive the above equation for $f_{p^*b} := p^*f_b$ is a strenuous calculation, requiring knowledge of holomorphic functional calculus, see \cite[Appendix B]{53UNITARY}. For this reason we omit the proof of this result.
	\end{remark} 
	\begin{remark} Recall in the previous chapter that the connective data on the cup product bundle gerbe was also be described in terms of orthogonal projections $P_i$ (Proposition \ref{jello}). Therefore Proposition \ref{curvature} will enable us to easily compare the connective data of both of our bundle gerbes in the next chapter. \end{remark} 
	
	\section{Isomorphisms of the pullback of the basic bundle gerbe} \label{subsection stable isos}

	\indent \indent To conclude this chapter, we explore equivalent descriptions of the pullback of the basic bundle gerbe, culminating in the following important result.
		\begin{proposition}\label{pullbackstableisomainresult} 
		There is an $SU(n)$-equivariant isomorphism $$\pullbackbasicbg \cong_{SU(n)} \sideset{}{_\mathrm{red}}\bigotimes_{i=1}^n \left( J_i^{\varepsilon_i}, Y_T\times \su, T\times \su\right).$$
	\end{proposition}
	The significance of Proposition \ref{pullbackstableisomainresult} is that it realises the pullback of the basic bundle gerbe as a reduced product of general cup product bundle gerbes, similar to the cup product bundle gerbe from the previous chapter. This will allow us to directly apply the results from Subsection \ref{subsectionstableisocups} (which detailed stable isomorphisms of reduced products of general cup product bundle gerbes) to simplify our central research question. We will also show that the isomorphism in Proposition \ref{pullbackstableisomainresult} preserves the bundle gerbes connections, which will be useful in our final chapter. The reader should feel free to skim the results from this section, and return to them after their application in Chapter \ref{ch:four}. \\
	\indent  The proof of Proposition \ref{pullbackstableisomainresult} relies on a series of intermediary isomorphisms, which we outline as follows.  \begin{align*} p^{-1}(P_b, Y) \ &\stackrel{\text{Prop \ref{iso1}  }}{\cong_{SU(n)}   }    \left(  \PbasicT \times_T SU(n), Y_T\times \su \right)\\
	&\stackrel{\text{Prop \ref{exactly} }}{\cong_{SU(n)}} \left(\Motimes_{i=1}^nJ_i^{\varepsilon_i}, Y_T\times \su\right)\\
	&\stackrel{\text{Ex \ref{examplereducedproducts}}}{=: \  \, } \sideset{}{_\mathrm{red}}\bigotimes_{i=1}^n \left(J_i^{\varepsilon_i}, Y_T\times \su\right).
	\end{align*}

We begin with a dual proposition and definition. The proof of this result is left as an exercise to the reader. Let $Y_T:= Y|_{Z\times T}$ and $\PbasicT := (P_b)|_{Y_T^{[2]}}$. Denote the restriction of the basic bundle gerbe to $T$ by $\basicbgtorus$. 
	\begin{proposition}[{\cite[p.$\,$1582]{53UNITARY}}] 
		Define $\PbasicT\times_T SU(n)$ to be the set of equivalence classes in $\PbasicT\times SU(n)$ under the relation $$(v_1\wedge \cdots \wedge v_k, g) \sim (tv_1 \wedge \cdots \wedge tv_k, gt^{-1})$$ for all $t\in T$. Then ${\PbasicT \times_T SU(n)}$ is a line bundle over $Y_T^{[2]}\times \su$, and there is an associative multiplication induced by that on the bundle gerbe $(P_{b, T}, Y_T)$ over $T$ making \begin{align}\label{vege} \left(  \PbasicT \times_T SU(n), Y_T\times \su \right)\end{align} a bundle gerbe over $T\times \su$. 
	\end{proposition}
	\begin{remark}\label{remark equivariance of other bundle gerbe} 
		It is not difficult to show that the bundle gerbe (\ref{vege}) is $SU(n)$-equivariant with respect to the $SU(n)$ action on $T\times \su$ given by multiplication on the $\su$ component. 
	
	\end{remark}

	We next show that the bundle gerbe (\ref{vege}) is isomorphic to the pullback of the basic bundle gerbe. To do so, we will make use of the explicit descriptions of $p^{-1}(Y), \ p^{-1}(Y)^{[2]}$ and $	\LL\hat{p}^{[2]}\RR^{-1}(\Pbasic)$ from Remark \ref{remark spaces in pullback bundle gerbe}. A pictorial description of the following result is provided in Figure \ref{fig:pullbacktriple}. 
	\begin{proposition}[{\cite[Proposition 7.3]{53UNITARY}}] \label{iso1} 
		There exists an $SU(n)$-equivariant bundle gerbe isomorphism $$ \left(  \PbasicT \times_T SU(n), Y_T\times \su \right)\cong_{SU(n)} p^{-1}(P_b, Y).$$  
	\end{proposition} 
	\begin{proof} It is clear from equation \eqref{eq:test} that $Y_T\times \su \cong p^{-1}(Y)$, and that this identification is a smooth isomorphism of surjective submersions over $T\times \su$. It follows that $Y_T^{[2]} \times \su \cong p^{-1}(Y)^{[2]}$ and this diffeomorphism covers $Y_T\times \su \xrightarrow{\sim} p^{-1}(Y).$ 
		It remains to define a line bundle isomorphism as in Figure \ref{fig:pullbackpf} that respects the bundle gerbe multiplication.  \begin{figure}[h!] 
			\centering
			\begin{tikzcd}
				\PbasicT \times_T SU(n) \arrow[r, "f"]  \arrow[d]	& 	(\hat{p}^{[2]})^{-1}(\Pbasic)  \arrow[d]	  \\		
				Y_T^{[2]}\times \su \arrow[r, "\sim"]&	{p}^{-1}(Y)^{[2]}
			\end{tikzcd} 
			\caption{Isomorphism of line bundles }
			\label{fig:pullbackpf} \end{figure}
		We do so by considering the disconnected components of $Y_T^{[2]}$. Let $Y_{T, +}^{[2]}, Y_{T, 0}^{[2]}$ and $Y_{T, -}^{[2]}$ denote the restriction of $Y_{+}^{[2]}, Y_{0}^{[2]}$ and $Y_{-}^{[2]}$ to $Y_T^{[2]}$ respectively.
		Consider $(z_1, z_2, t) \in Y^{[2]}_{T, +}$ and $ v_1\wedge \cdots \wedge v_k\in (\PbasicT)_{(z_1, z_2, t)}$. Then each $v_i$ is an eigenvector of $t$ with eigenvalue $\lambda_i$, say. Recall from the proof of Proposition \ref{proposition basic bg is equivariant} that multiplication by $g$ induces a map $L_{(z_1, z_2, t)}\to L_{(z_1, z_2, gtg^{-1})}$. This in turn induces a (linear) map $(\PbasicT)_{(z_1, z_2, g)} \to (\PbasicT)_{(z_1, z_2, gtg^{-1})}$ defined by
		\begin{align*}
		v_1\wedge \cdots \wedge v_k \mapsto gv_1\wedge \cdots \wedge gv_k.
		\end{align*} 
		Define $f:\PbasicT\times_T SU(n) \to (\hat{p}^{[2]})^{-1}(\Pbasic)$ over fibres $(gT, z_1, z_2, t) \in Y_{T, +}^{[2]}\times \su$ by  $$[v_1\wedge \cdots \wedge v_k, g]  \mapsto (gT, t, z_1, z_2, gv_1 \wedge \cdots \wedge gv_k).$$
		This map is well-defined by definition of the $T$-action on $\PbasicT\times SU(n)$. It is also clearly a linear isomorphism on fibres covering $Y_{T, +}^{[2]}\times \su  \xrightarrow{\sim} p^{-1}(Y)$. By duality, we can extend $f$ to $Y_{T, -}^{[2]}\times \su$, and we can trivially extend $f$ to $Y_{T, 0}^{[2]}\times \su$. These extensions will remain linear isomorphisms. It can be verified that $f$ commutes with the bundle gerbe multiplication, so is the desired line bundle isomorphism. It is not difficult to show that this isomorphism is $SU(n)$-equivariant.\end{proof} 

		\begin{figure} \centering
			\begin{tikzcd}
				\PbasicT \times_T SU(n) \arrow[r, "\sim"]  \arrow[d]	& 	(\hat{p}^{[2]})^{-1}(\Pbasic) \arrow[r]  \arrow[d]		  &\Pbasic \ \  \arrow[d, shift right = .7ex]   \\
				Y_T^{[2]}\times \su \arrow[r, "\sim"] \arrow[d, shift right=1.2ex] \arrow[d,shift left=.1ex]&	{p}^{-1}(Y)^{[2]} \arrow[r, "\hat{p}^{[2]}"] \arrow[d, shift right=1.2ex] \arrow[d,shift left=.1ex]& Y^{[2]}  \arrow[d, shift right=1.2ex] \arrow[d,shift left=.1ex] \\
				Y_T \times \su \arrow[r, "\sim"]  \arrow[d]&	p^{-1}(Y)  \arrow[r, "\hat{p}"] \arrow[d] &  Y \  \  \arrow[d, shift right = .7ex] \\ 
				T\times \su \arrow[r, equal] &	T\times \su \arrow[r, " \  p"] & SU(n) \ \ 
			\end{tikzcd} 
			\captionof{figure}{(left-right) bundle gerbe (\ref{vege}), $\pullbackbasicbg$, and $(P_b, Y, SU(n))$ }
			\label{fig:pullbacktriple} \end{figure}

	\indent Next, we define the maps $\varepsilon_i$ and the $T$-spaces $\C_{p_i}$ which will be crucial to the remaining propositions. Recall that $T$ is the maximal torus of $SU(n)$ consisting of diagonal matrices, and $p_i:T\to S^1$ is the homomorphism sending $t\in T$ to its $i$-th diagonal. To define $\varepsilon_i$ we use the ordering on $Z$ from Section \ref{section: definition of basic bg}. Regard $i$ as an integer between $1$ and $n$ inclusive throughout.  
		\begin{definition}\label{defnepsilon}
		Define $\varepsilon_i: Y_T^{[2]}\to \Z$ by $$\varepsilon_i(z_1, z_2, t) = \begin{cases}
		1 & \text{if } z_1 > p_i(t) > z_2 \\
		-1 & \text{if } z_2 > p_i(t) > z_1 \\ 
		0 & \text{otherwise}.
		\end{cases} $$
	\end{definition}
		\begin{definition}
	Let $\C_{p_i}$ be the space $\C$ equipped with the right $T$-action $v\cdot t := p_i(t^{-1})v$. 
	\end{definition}
\begin{remark}
Throughout this section, let $\C_{p_i}^1:=\C_{p_i}, \C_{p_i}^{-1}:= \C_{p_i}^*,$ and $\C_{p_i}^0:=\C$, where $\C_{p_i}^0$ is equipped with the identity action. The space $\C_{p_i}^*$ can be understood as the dual of $\C_{p_i}$, or equivalently as the space $\C$ equipped with the dual action $v\cdot t = p_i(t)v$. 
\end{remark}
	\indent We can now introduce the general cup product bundle gerbes of $J_i\to \su$ and $\varepsilon_i$ in preparation for Proposition \ref{exactly}. Recall that $J_i\to \su$ is the $SU(n)$-homogeneous line bundle defined by setting $J_i  := \C\times_{p_i} SU(n)$ under the relation $(z, s) \sim_{p_i} (p_i(t^{-1})z, st)$ for all $t\in T$. 
\begin{proposition}\label{equivarianceofepsilon}
	There is an associative multiplication making \begin{align} \label{biscut}\left(J_i^{\varepsilon_i}, Y_T\times \su \right)\end{align} a bundle gerbe over $T\times \su$. Moreover, this bundle gerbe is $SU(n)$-equivariant with respect to the $SU(n)$-action on $T\times \su$ given by multiplication on the $\su$ component.
\end{proposition}
\begin{proof}
	To see that this is a general cup product bundle gerbe (Definition \ref{definition of general cup product}), and hence a bundle gerbe, it suffices to show that $\varepsilon_i: Y_T^{[2]}\to \Z$ satisfies the cocycle condition $$\varepsilon_i(z_1, z_2, t) + \varepsilon_i(z_2, z_3, t) = \varepsilon_i(z_1, z_3, t) $$ for all $(z_1, z_2, z_3, t)\in Y_T^{[3]}$. We verify this holds for one case and leave the remaining cases as an exercise. Let $(z_1, z_2, z_3, t)\in Y_T^{[3]}$, and assume $z_1<z_2<z_3$ with $z_1<p_i(t)<z_2$. Then $\varepsilon_i(z_1, z_2, t) = -1, \ \varepsilon_i(z_2, z_3, t) = 0$ and $\varepsilon_i(z_1, z_3, t) = -1$, so the cocycle condition is satisfied. Equivariance follows easily from equivariance of $J_i\to \su$. 
\end{proof}
	\indent Now we can present the second key result in this section, Proposition \ref{exactly}. Recall that, for $\lambda$ an eigenvalue of $g\in SU(n)$, $E_{(g, \lambda)}$ is the $\lambda$-eigenspace of $g$, $L_{(z_1, z_2, g)} = \oplus_{z_1>\lambda>z_2}E_{(g, \lambda)}$, and the line bundle $P_b\to Y^{[2]}$ is defined as the determinant of $L$ (Section \ref{section: definition of basic bg}). 

\begin{proposition}\label{exactly}
There exists an $SU(n)$-equivariant bundle gerbe isomorphism \begin{align}\label{yerp} \left(  \PbasicT \times_T SU(n), Y_T\times \su \right)\cong_{SU(n)} \left( \Motimes_{i=1}^nJ_i^{\varepsilon_i}, Y_T\times \su\right).\end{align}
\end{proposition}
\begin{proof}
First, we show there is a line bundle isomorphism $\PbasicT \cong \Motimes \C_{p_i}^{\varepsilon_i} \times Y_T^{[2]},$ which will allow us to show that the line bundles in (\ref{yerp}) are isomorphic in a moment. Consider $(z_1, z_2, t)\in Y_T^{[2]}$ with $z_1>z_2$. Suppose there are eigenvalues of $t$ between $z_1$ and $z_2$. Denote these eigenvalues by $p_{k_1}(t), ..., p_{k_m}(t)$ for $1\leq k_1\leq \cdots \leq k_m \leq n$. Then \begin{align}\label{choc} L_{(z_1, z_2, t)} = E_{(t, p_{k_1}(t))} \oplus \cdots \oplus E_{(t, p_{k_m}(t))} \end{align}	and $$
(\PbasicT)_{(z_1, z_2, t)} = \text{det}(L_{(z_1, z_2, t)}) = \Motimes_{z_1 >\lambda > z_2} \text{det} \left(E_{(t, \lambda)}\right)=E_{(t, p_{k_1}(t))}\otimes \cdots \otimes E_{\left(t,p_{k_m}(t)\right)}.
$$ 	Now, each eigenspace $E_{(t, p_k(t))} \cong \C$ is equipped with a $T$-action $v\cdot s := p_k(s^{-1}) v$, hence $E_{(t, p_k(t))}\cong \C_{p_k}$ as $T$-spaces for each $k$. Since $\varepsilon_{k_i}(z_1, z_2, t)=1$ for $i=1, ..., m$ and $\varepsilon_k=0$ otherwise, we see that $$ \left(\PbasicT\right)_{(z_1, z_2, t)} \cong \C_{p_{k_1}} \otimes \cdots \otimes \C_{p_{k_m}} \cong \C_{p_1}^{\varepsilon_1(z_1, z_2, t)}\otimes \cdots \otimes \C_{p_n}^{\varepsilon_n(z_1, z_2, t)}.$$ 
By almost identical arguments, this holds over the other components of $Y_T^{[2]}$. Therefore we have an isomorphism $\PbasicT \cong \Motimes_{i=1}^n \C_{p_i}^{\varepsilon_i} \times Y_T^{[2]},$ as claimed. \\
\indent With this result, we can show that the line bundles in (\ref{yerp}) are isomorphic. Now, the isomorphism $\PbasicT \cong \Motimes_{i=1}^n \C_{p_i}^{\varepsilon_i} \times Y_T^{[2]}$ implies we have an isomorphism \begin{align}\label{nope}\PbasicT\times_T SU(n) \cong \left(\Motimes \C_{p_i}^{\varepsilon_i} \times Y_T^{[2]}\right)\times_TSU(n),\end{align} where the latter line bundle is $SU(n)$-equivariant with $T$-action defined by $$(z_1, ..., z_n, u, g) \cdot t= (p_1(t^{-1})z_1, ..., p_n(t^{-1})z_n, u, gt).$$ It can be verified that the line bundle isomorphism (\ref{nope}) is $SU(n)$-equivariant. This will act as our `intermediary isomorphism'. Next, consider the space $\left(\Motimes_{i=1}^n \C_{p_i}^{\varepsilon_i} \right) \times_T SU(n)$ of equivalence classes under the $T$-action $$(z_1, ..., z_n, g)\cdot t = (p_1(t^{-1}) z_1, ..., p_n(t^{-1})z_n, gt). $$ This is an $SU(n)$-homogeneous line bundle over $\su$, and it can be verified that the natural map \begin{align}\label{pert} {\left(\C_{p_1}^{\varepsilon_1} \otimes \cdots \otimes \C_{p_n}^{\varepsilon_n}\times Y_T^{[2]}\right) \times_T SU(n) } \to Y_T^{[2]}\times \left(\C_{p_1}^{\varepsilon_1} \otimes \cdots \otimes \C_{p_n}^{\varepsilon_n} \times_T SU(n)\right) \end{align} is a well defined, $SU(n)$-equivariant line bundle isomorphism over $Y_T^{[2]} \times \su$. Now, by Theorem \ref{thrmbundles} and the definition of $J_i$, there are $SU(n)$-equivariant line bundle isomorphisms \begin{align*} \C_{p_1}^{\varepsilon_1} \otimes \cdots \otimes \C_{p_n}^{\varepsilon_n} \times_T SU(n) &\cong (\C_{p_1}^{\varepsilon_1}\times_TSU(n)) \otimes \cdots \otimes (\C_{p_n}^{\varepsilon_n} \times_T SU(n))\\  &\cong J_1^{\varepsilon_1}\otimes \cdots \otimes J_n^{\varepsilon_n}. \end{align*} This, combined with (\ref{pert}), implies there is an $SU(n)$-equivariant line bundle isomorphism $$\left(\Motimes_{i=1}^n \C_{p_i}^{\varepsilon_i} \times Y_T^{[2]} \right) \times_T SU(n)\cong   \Motimes_{i=1}^n J_i^{\varepsilon_i}$$ and hence, by (\ref{nope}),  we have an $SU(n)$-equivariant line bundle isomorphism $$P_{b, T}\times_T SU(n) \cong \Motimes_{i=1}^n J_i^{\varepsilon_i}. $$ It remains to show that this isomorphism preserves the bundle gerbe product. Consider $z_1 > z_2 > z_3$ as in Proposition \ref{teawoke}. Then the basic bundle gerbe product is induced from $$\text{det}(L_{(z_1, z_2, t)}) \otimes \text{det}(L_{(z_2, z_3, t)}) \cong \text{det}(L_{(z_1, z_2, t)} \oplus L_{(z_2, z_3, t)}) \cong \text{det}(L_{(z_1, z_3, t)}). $$ From the discussion above and equation (\ref{choc}), each $L_{(z_i, z_j, t)}$ decomposes into appropriate sums of the $J_l$ to powers, so this becomes $$\Motimes_{i=1}^n J_i^{\varepsilon_i(z_1, z_2, t)}\otimes \Motimes_{i=1}^n J_i^{\varepsilon_i(z_2, z_3, t)} \cong \Motimes_{i=1}^n J_i^{\varepsilon_i(z_1, z_3, t)},$$ which is the cup product multiplication as in the proof of Proposition \ref{equivarianceofepsilon}. The other cases proceed similarly.
\end{proof}

	Clearly, the reduced product of the $SU(n)$-equivariant bundle gerbes $(\ref{biscut})$ will be an $SU(n)$-equivariant bundle gerbe. This leads us to our final isomorphism, which follows from Propositions \ref{iso1} and \ref{exactly}.\\

\noindent \noindent \textbf{Proposition \ref{pullbackstableisomainresult}.}\textit{
	There is an $SU(n)$-equivariant isomorphism $$p^{-1}(P_b, Y) \cong_{SU(n)} \sideset{}{_\mathrm{red}}\bigotimes_{i=1}^n \left( J_i^{\varepsilon_i}, Y_T\times \su\right).$$}
\indent To conclude this section, we show that the isomorphism in Proposition \ref{pullbackstableisomainresult} preserves the bundle gerbe connections. This will be crucial to the proof of a key result in Chapter \ref{ch:four}, namely that the holonomies of the pullback of the basic bundle gerbe and the cup product bundle gerbe are not equal. 
\begin{proposition}\label{tired}
The bundle gerbe isomorphism in \textup{Proposition \ref{pullbackstableisomainresult}} preserves connections, that is, $$\nabla_{p^*b} = \nabla_{\otimes J_i^{\varepsilon_1}}$$ for $\nabla_{p^*b}$ the pullback connection from \textup{Proposition \ref{proposition connective structure of pullback}} and $\nabla_{\otimes J_i^{\varepsilon_1}}$ the product connection induced by the general cup product bundle gerbe connections from \textup{Proposition \ref{existsconnectiononcup}}.
\end{proposition}
\begin{proof}
Recall that the total space of the basic bundle gerbe, $P_b$, is defined to be the determinant of $L\to Y_+^{[2]}$ where $L$ is the sum of eigenspaces $L_{(z_1, z_2, g)} = \bigoplus_{z_1>\lambda >z_2} E_{(g, \lambda)}$. We claim that, fibrewise, $p^{-1}(L)$ is isomorphic to a sum of spaces $J_i$. Fix $(z_1, z_2, h)\in Y^{[2]}$. Write $h=gtg^{-1}$ for some $g\in SU(n)$ and $t\in T$. Denote the eigenvalues of $t$ between $z_1$ and $z_2$ by $t_{i_1}, ..., t_{i_k}$ for $i_1 \leq \cdots \leq i_k$. Now, $(z_1, z_2, t, gT)\in Y_T^{[2]}\times \su$ and $$p^{-1}(L)_{(z_1, z_2, t, gT)} \cong L_{(z_1, z_2, h)} \cong {SU(n)\times_T L_{(z_1, z_2, t)}}$$ where the final isomorphism is defined by $v\mapsto [g, g^{-1}v]$. Using similar arguments to those in the proof of Proposition \ref{exactly}, we find that $$SU(n)\times_T L_{(z_1, z_2, t)} \cong SU(n)\times_T \left(\C_{p_{i_1}}\oplus \cdots \oplus \C_{p_{i_k}}\right) \cong J_{i_1}\oplus \cdots \oplus J_{i_k} \big|_{gT}.$$ Therefore $p^{-1}(L)_{(z_1, z_2, t, gT)}  \cong J_{i_1}\oplus \cdots \oplus J_{i_k} \big|_{gT}$. 
Since the connection on $L$ is defined by orthogonal projection, so too is the pullback connection on $p^{-1}(L)$. Moreover, as $p^{-1}(L)$ decomposes into a sum of orthogonal spaces $J_{i}$, orthogonal projection onto $L$ will be equal to the sum of orthogonal projections onto each space $J_i$. Therefore the connection on $p^{-1}(L)$ will decompose into a product of connections $\nabla_{i}$ defined by orthogonal projection onto $J_{i}$. Finally, by taking the determinant of these connections, we obtain the result.  
\end{proof}

%% file: chapter4/CHP4.tex
\chapter{The stable isomorphism\label{ch:four}}

\indent \indent In this chapter, we combine the work of Chapters $2$, $3$, and $4$ to prove that the pullback of the basic bundle gerbe by the Weyl map is $SU(n)$-stably isomorphic to the cup product bundle gerbe, i.e.  \begin{align}\label{keyequation} \cupbg \cong_{SU(n)\text{-stab}} \pullbackbasicbg.\end{align} The significance of this result is that it deepens our understanding of the basic bundle gerbe over $\su$. As we have seen, the basic bundle gerbe over a Lie group $G$ is a standard example of a non-trivial bundle gerbe, and is, in some sense, the canonical bundle gerbe associated to $H^3(G, \Z)$. For this reason, it is one of the more interesting examples of the many gerbes and bundle gerbes that have been constructed over compact, simple, simply connected Lie groups, such as those studied in \cite{brylinski, 1996bundlegerbes, higherbgs, Meinrenken, mickelsson03}. To offer another perspective, the result (\ref{keyequation}) illustrates how the relatively modern cup product bundle gerbe construction can be exploited to understand classical bundle gerbes such as the basic bundle gerbe. \\
\indent There is a range of possible extensions of our result. The most obvious would be to replace $SU(n)$ with a general compact, simple, simply connected Lie group $G$ using the basic bundle gerbe defined in terms of path fibrations from \cite{dannythesis}. We are confident that there is an analogous result to (\ref{keyequation}) for the basic and cup product gerbes in \cite{brylinski}, which, to our knowledge, has not been explored in the literature. One could also consider our work in the holomorphic category using \cite{brylinskiarticle}. Another possible extension relates to descent problems for bundle gerbes - that is, when, given a fibration $M\to N$ and a bundle gerbe over $M$, does the bundle gerbe descend to $N$? We can interpret (\ref{keyequation}) to mean that the cup product bundle gerbe descends to $SU(n)$ via the Weyl map. The Weyl map is, however, not a fibration. This could be indicative of a more general descent result for bundle gerbes, which would require further research. \\
\indent We begin the first section of this chapter by explicitly describing the requirements for our bundle gerbes to be stably isomorphic. This discussion, which relies heavily on results from Chapters $3$ and $4$, culminates in the main result of this section, Proposition \ref{mainresult}. This proposition states a sufficient condition for our bundle gerbes to be $SU(n)$-stably isomorphic in terms of continuous integer-valued functions $h_i$, thereby simplifying our problem significantly. \\
\indent  In the second section of this chapter, we endeavour to prove the existence of the integer-valued functions $h_i$ from Proposition \ref{mainresult}. The tool we make use of to do so is the geometry of the basic and cup product bundle gerbes calculated in the previous chapters. In Chapter $2$, we described relationships obeyed by the connective data of stably isomorphic bundle gerbes (Proposition \ref{outofnamesforthings}). By inputting the connective data of our bundle gerbes into these relations, we are able to find explicit descriptions of the integer-valued functions $h_i$. We then show these $h_i$ satisfy the desired conditions (Proposition \ref{hisright}), thereby proving our bundle gerbes are $SU(n)$-stably isomorphic.\\
\indent We conclude this chapter with a brief study of the holonomy of our bundle gerbes. Using results from Chapter $2$, together with a short computation, we find that the holonomies of the pullback of the basic bundle gerbe and the cup product bundle gerbe are proportional, but not equal (Proposition \ref{holonomycomputation}). Therefore the pullback of the basic bundle gerbe and the cup product bundle gerbe are stably isomorphic (hence have the same Dixmier-Douady class), but are \textit{not} stably isomorphic as bundle gerbes with connection and curving (so do not have the same Deligne class) (Theorem \ref{finalresult}). 
\section{Set up of the problem}\label{d1}
\indent \indent In the proceeding discussion, we aim to simplify our central research problem using that the pullback of the basic bundle gerbe and cup product bundle gerbes are stably isomorphic to reduced products of general cup product bundle gerbes. Now, by Proposition \ref{redproductofcupproduct}, there is an $SU(n)$-equivariant stable isomorphism of the cup product bundle gerbe and a reduced product of general cup product bundle gerbes $$(P_c, X) \cong_{SU(n)\textup{-stab}}	\sideset{}{_\mathrm{red}}\bigotimes_{i=1}^n 
\left( J_i^{d_i}, \R^{n-1} \times \su \right).$$ Also, by Proposition \ref{pullbackstableisomainresult}, there is an $SU(n)$-equivariant isomorphism of the pullback of the basic bundle gerbe with a reduced product of general cup product bundles  \begin{align}\label{qqq}p^{-1}(P_b, Y)\cong_{SU(n)} \sideset{}{_\mathrm{red}}\bigotimes_{i=1}^n \left( J_i^{\varepsilon_i}, Y_T\times \su\right).\end{align} Therefore (\ref{keyequation}) is equivalent to \begin{align*} \sideset{}{_\mathrm{red}}\bigotimes_{i=1}^n \left( J_i^{d_i}, \R^{n-1}\times {SU(n)}/{T}\right) \cong_{SU(n)\textup{-stab}}  \sideset{}{_\mathrm{red}}\bigotimes_{i=1}^n \left( J_i^{\varepsilon_i}, Y_T\times \su \right).\end{align*} Since both of these bundle gerbes are reduced products of general cup product bundle gerbes, Corollary \ref{cormaincupproductresult} applies. Namely, to prove these bundle gerbes are stably isomorphic, it suffices to show that there exist smooth functions $h_i: \left(\R^{n-1} \times_T Y_T \right)\times \su \to \Z$ such that $$d_i(x, y, t) - \varepsilon_i(z, w, t) = h_i(y, w, t, gT) - h_i(x, z, t, gT)$$ for all $(x, y, z, w, t, gT)\in \left(\R^{n-1}\times_T Y_T\right)^{[2]} \times \su$, where $d_i$ and $\varepsilon_i$ are the maps defined in Remark \ref{defndi} and Definition \ref{defnepsilon}, respectively. Equivariance of the stable isomorphism will then follow trivially by equivariance of the general cup product bundle gerbes. This result, stated formally below, will be exploited in the next section to prove (\ref{keyequation}).

\begin{proposition}\label{mainresult}
	The pullback of the basic bundle gerbe is $SU(n)$-stably isomorphic to the cup product bundle gerbe if, for each $i=1, ..., n$, there exist smooth functions $h_i: \left(\R^{n-1}\times_T Y_T\right) \times \su \to \Z$ such that \begin{align}\label{equationofh} y_i-x_i - \varepsilon_i(z, w, t) = h_i(y, w, t, gT) - h_i(x, z, t, gT)\end{align} for all $(x, y, z, w, t, gT)\in \left(\R^{n-1} \times_T Y_T\right)^{[2]} \times \su $ with $x= (x_1, ..., x_n), y= (y_1, ..., y_n)$.
\end{proposition}
\section{Finding the stable isomorphism}\label{d2}

\indent \indent We have now simplified our central research problem to one of finding $n$ integer-valued functions $h_i$ satisfying equation (\ref{equationofh}). To find these functions, we utilise the connective data on our bundle gerbes presented in the previous chapters. By inputting this connective data into the relations we would expect to obtain if our bundle gerbes were stably isomorphic, we will find an explicit description for the functions $h_i$, and show that they satisfy the desired condition. \\
\indent Recall that the basic bundle gerbe and the cup product bundle gerbe have connective data $(\nabla_{p^*b}, f_{p^*b})$ and $(\nabla_c, f_c)$ respectively with associated three-curvatures $\omega_{p^*b}$ and $\omega_c$. By Proposition \ref{outofnamesforthings}, if the pullback of the basic bundle gerbe is stably isomorphic to the cup product bundle gerbe, there exists $\beta \in \Omega^2\left(T\times \su\right)$ and a trivialising line bundle $R$ with connection $\nabla_R$ such that \begin{align}\label{curvingscompare} f_{p^*b} - f_c = F_{\nabla_R} + \pi^*\beta\end{align} and \begin{align}\label{curvaturescompare}\omega_{p^*b} - \omega_c = d\beta\end{align} where $\pi$ is projection $ \left(\R^{n-1}\times_T Y_T\right) \times \su \to T\times \su$. Throughout this section we will take $R$ to be the line bundle \begin{align}\label{oragne} R:= \Motimes_{i=1}^n \pi_{\smath{{SU(n)}/{T}}}^{-1}(J_i)^{h_i}\to (\R^{n-1}\times_T Y_T) \times \su \end{align} for $\pi_{\smath{{SU(n)}/{T}}}: (\R^{n-1} \times_T Y_T) \times \su \to  \su$ projection and $h_i: (\R^{n-1}\times_TY_T)\times \su \to \Z$ the smooth maps that we aim to define. This is the trivialising line bundle that we will obtain from Corollary \ref{cormaincupproductresult} if Proposition \ref{mainresult} holds. 

Throughout this section, we abuse notation and let the homomorphisms $p_i$ and projections $P_i$ be defined on the spaces $p^{-1}(Y)$, $X$, or $X\times_{T\times {SU(n)}/{T}} p^{-1}(Y)$ depending on the context. Before commencing our calculations, recall by Propositions \ref{corollary data on cup prod} and \ref{curvature} that \begin{enumerate}[(1),wide, labelwidth=!, labelindent=0pt]
\item the curving of the cup product bundle gerbe is \begin{align} \label{fc} f_c = -\sum_{i=1}^n x_i \, \textup{tr}(P_i dP_idP_i);\end{align}
\item the three-curvature of the cup product bundle gerbe is \begin{align}\label{omegac} \omega_c = - \frac{1}{2\pi i}\sum_{i=1}^n p_i^{-1} dp_i\, \textup{tr}(P_idP_idP_i);\end{align}
\item the curving of the pullback of the basic bundle gerbe by the Weyl map is \begin{align}\label{fb} f_{p^*b} = \frac{i}{4\pi} \sum_{\substack{i=1 \\ i\neq k}}^n (\log_z p_i - \log_z p_k + (p_k - p_i)p_k^{-1})\textup{tr}(P_i dP_kdP_k);\end{align}
\item the three-curvature $\omega_{p^*b}$ of the pullback of the basic bundle gerbe by the Weyl map is \begin{align}\label{omegap} \frac{i}{4\pi} \sum_{\substack{i,k=1 \\ i\neq k}}^n \left(p_i^{-1} dp_i - p_k^{-1}dp_k -p_k^{-1}dp_i + p_k^{-1}dp_k p_k^{-1}p_i\right) \textup{tr}(P_idP_kdP_k) - p_ip_k^{-1} \textup{tr}(dP_idP_kdP_k). \end{align}
\end{enumerate}
\indent \indent Let us begin by comparing the three-curvatures of our bundle gerbes, and verifying that their difference is an exact form. This will provide us with an expression for $\beta$, which we will then input into equation (\ref{curvingscompare}) to obtain a formula for $F_{\nabla_R}$, and hence for $h_i$. The proceeding proofs rely on a series of lemmas that we relegate to the appendix in an effort to streamline this discussion. 
\begin{proposition}\label{terer}
 The three-curvature of the pullback of the basic bundle gerbe by the Weyl map \textup{(\ref{omegap})} can be rewritten as \begin{align}\label{hdf} \omega_{p^*b} = -\frac{1}{2\pi i}\sum_{k=1}^n p_k^{-1}dp_k \textup{tr}(P_kdP_kdP_k) + d\beta \end{align} for \begin{align}\label{equationforbeta1}\beta =  -\frac{i}{4\pi}\sum_{\substack{i=1 \\ i\neq k}}^n  \sum_{k=1}^n p_ip_k^{-1}\textup{tr}(P_idP_kdP_k).\end{align}
\end{proposition}
\begin{proof}
This is an immediate consequence of Lemma \ref{lemma111} and Lemma \ref{projections} (3). 
\end{proof}
\begin{proposition}	The $2$-form $\beta$ \textup{(\ref{equationforbeta1})} {satisfies} \textup{equation (\ref{curvaturescompare})}. \end{proposition}
\begin{proof}
This follows by comparing our expression for $\omega_{p^*b}$ in equation (\ref{hdf}) with that of $\omega_c$ in equation (\ref{omegac}). 
\end{proof}

This result shows that the three-curvatures of our bundle gerbes behave as we would expect if they are to be stably isomorphic. Furthermore, we now have an explicit expression for $\beta$, which we can substitute into equation (\ref{curvingscompare}). This in turn determines an expression for the functions $h_i$, as seen in the next proposition.

\begin{proposition}\label{findingtheh}
The curvings $f_{p^*b}$ \textup{(\ref{fb})}, $f_c$ \textup{(\ref{fc})}, $2$-form $\beta$ \textup{(\ref{equationforbeta1})}, and line bundle $R$ \textup{(\ref{oragne})} satisfy \textup{equation} \textup{(\ref{curvingscompare})} if, for each $i=1, ..., n$, $$h_i(x, z, t, gT) = x_i - \frac{1}{2\pi i}\log_z p_i(t)$$ for all $(x, z, t, gT) \in (\R^{n-1}\times_TY_T)\times \su$ with $x = (x_1, ..., x_n)$. 
\end{proposition}
\begin{proof} First, note that $$f_{p^*b} - f_c - \pi^*\beta=-\frac{1}{4\pi i} \sum_{\substack{i=1 \\ i\neq k}}^n \sum_{k=1}^n (\log_z p_i - \log_z p_k - 4\pi i x_i +1)\textup{tr}(P_i dP_kdP_k).$$ Since $ \sum_{\substack{i=1 \\ i\neq k}}^n \sum_{k=1}^n \text{tr}(P_idP_kdP_k)=\sum_{k=1}^n \text{tr}(P_kdP_kdP_k) = 0$,  this becomes $$\sum_{i=1}^n x_i  \text{tr}(P_idP_idP_i) - \frac{1}{2\pi i}\sum_{i=1}^n\log_z p_i \text{tr}(P_idP_idP_i).$$ Now, the induced connection $\nabla_R$ on $R$ has curvature $F_{\nabla_R} = \sum_{i=1}^n (h_i\circ \pi_i) \textup{tr}(P_idP_idP_i).$ So if equation (\ref{curvingscompare}) holds,  $$\sum_{i=1}^n \left(x_i - \frac{1}{2\pi i}\log_z p_i\right) \text{tr}(P_idP_idP_i) = \sum_{i=1}^n h_i\circ \pi_i \text{tr}(P_idP_idP_i).$$
By equating coefficients we obtain the result. \end{proof}

\indent It remains to be shown that the $h_i$ from Proposition \ref{findingtheh} satisfy equation (\ref{equationofh}). To do so, we require the following key lemma. Recall that, for $z\in \C^*$, $R_z$ is the closed ray from the origin through $z$, and the logarithm $\text{log}_z:\C\backslash{R_z}\to \C$ is defined by setting $\text{log}_z(1) = 0$. 
\begin{lemma}\label{lemmalogandepsilon} 
	For each $i=1, ..., n$ and $(z, w, t) \in Y_T^{[2]}$, $$\varepsilon_i(z, w, t) = {\frac{1}{2\pi i}}\left(\log_zp_i(t) - \log_wp_i(t)\right). $$
\end{lemma}
\begin{proof}  Let $(z, w, t)\in Y_T^{[2]}$ with $z>w$. If $w<p_i(t)<z$, $\log_zp_i(t) - \log_wp_i(t) = 2\pi i$. Otherwise, this difference is zero. Therefore in general $$\text{log}_zp_i(t) - \text{log}_wp_i(t) = \begin{cases}
	2\pi i &\text{if } z>p_i(t) > w\\
	-2\pi i & \text{if } w>p_i(t) > z\\
	0 &\text{otherwise}.
	\end{cases}$$ Dividing through by $2\pi i$, we see that this is precisely the definition of $\varepsilon_i$. \end{proof}
With this, we can prove that the functions $h_i$ satisfy the desired condition. 
\begin{proposition}\label{hisright}
For $i= 1, ..., n$, define $h_i: \left(\R^{n-1} \times_T Y_T\right) \times \su \to \Z$ by $$h_i(x, z, t, gT) = x_i - \frac{1}{2\pi i}\log_z p_i(t)$$ where $x = (x_1, ..., x_n)$. Then $h_i$ is smooth and satisfies $$ y_i-x_i - \varepsilon_i(z, w, t) = h_i(y, w, t, gT) - h_i(x, z, t, gT)$$ for all $(x, y, z, w, t, gT)\in \left(p_i^{-1}(\R) \times_T Y_T\right)^{[2]} \times \su $ with $x = (x_1, ..., x_n), y= (y_1, ..., y_n)$.
\end{proposition}
\begin{proof}
We first show $h_i$ is integer-valued. Now, $(x, z, t, gT) \in \left(\R^{n-1} \times_T Y_T\right) \times \su$ implies $[x_i] = e^{2\pi i x_i} = p_i(t)$. Therefore, exponentiating $h_i$, we obtain  $e^{2\pi i h_i} = e^{2\pi i x_i}p_i(t)^{-1}= 1$, and $h_i$ must be integer valued. Smoothness of $h_i$ follows by noting that $\log$ is smooth over the given range as $z\neq p_i(t)$ for $(x, z, t, gT)\in \left(\R^{n-1} \times_T Y_T\right) \times \su$. Finally, by Lemma \ref{lemmalogandepsilon}, \begin{align*}
 h_i(y, w, t, gT) - h_i(x, z, t, gT) &= y_i-x_i - \frac{1}{2\pi i}\left(\log_{z}p_i(t) - \log_{w}p_i(t)\right)\\
&= y_i - x_i - \varepsilon_i(z, w, t),
\end{align*} so these are the desired functions $h_i$.
\end{proof}
The following result then follows immediately from Propositions \ref{mainresult} and \ref{hisright}. A more precise statement of this result will be provided in Theorem \ref{finalresult}.
\begin{proposition}\label{prp}
	The cup product bundle gerbe is $SU(n)$-stably isomorphic to the pullback of the basic bundle gerbe by the Weyl map, i.e. $$\cupbg \cong_{SU(n)\textup{-stab}} \pullbackbasicbg.$$
\end{proposition}

\section{Comparing holonomies} \label{d3}

\indent\indent We conclude this chapter by briefly studying the holonomies of our bundle gerbes. Recall that $D$-stably isomorphic bundle gerbes over a surface have the same holonomy (Proposition \ref{holonomyequalforstabisobg}). Therefore if we can show our bundle gerbes have different holonomies on a surface $\Sigma \subset T\times \su$, then the restriction of our bundle gerbes to $\Sigma$ (and hence the original bundle gerbes) cannot be $D$-stably isomorphic, and their Deligne classes will not be equal (Proposition \ref{bgclassification}).\\
\indent By Proposition \ref{findingtheh}, the curvings of the pullback of the basic bundle gerbe and cup product bundle gerbes satisfy \begin{align}\label{wer} f_{p^*b} = f_c + F_{\nabla_R} + \pi^*\beta_n\end{align} for \begin{align}\label{werrr} \beta_n =  -\frac{i}{4\pi}\sum_{\substack{i=1 \\ i\neq k}}^n  \sum_{k=1}^n p_ip_k^{-1}\text{tr}(P_idP_kdP_k).\end{align} Here, we introduce the notation $\beta_n$ to emphasise that $\beta_n$ is defined on $T\times SU(n)$. Let $\Sigma \subset T\times \su$ be a surface. By Proposition \ref{tired}, $\nabla_{p^*b} = \nabla_c$. So by Proposition \ref{proposition holonomy for different curvings} and equation (\ref{wer}), \begin{align}\label{aadff} \text{hol}((\nabla_{p^*b}, f_{p^*b}), \Sigma) = \text{hol}((\nabla_c, f_c+F_{\nabla_R}+\pi^*\beta_n), \Sigma) = \text{exp}\left(\int_{\Sigma} \beta_n\right) \text{hol}((\nabla_c, f_c), \Sigma).\end{align} That is, the holonomy of the pullback of the basic bundle gerbe by the Weyl map differs from the holonomy of the cup product bundle gerbe by a multiple of $\text{exp}\left(\int_\Sigma \beta_n\right)$. It could, of course, be the case that $\int_\Sigma \beta_n = k2\pi i$ for some $k\in \Z$, implying these holonomies are equal. In the following proposition, we show that there exists a surface $\Sigma_2 \subset T\times {SU(2)}/{T}$ for which $\int_{\Sigma_2} \beta_2 \neq k2\pi i$ for any $k\in \Z$. We will then generalise this result to obtain a surface $\Sigma_n \subset T\times {SU(n)}/{T}$  for which $\text{hol}((\nabla_{p^*b}, f_{p^*b}), \Sigma_n)  \neq \text{hol}((\nabla_c, f_c), \Sigma_n)$.

\begin{proposition}\label{holonomycomputation}
Define a surface $\Sigma_2\subset T\times {SU(2)}/{T}\cong S^1\times S^1$ by $\Sigma_2 := \{e^{{\pi i}/{4}}\}\times S^2$. Then the holonomies of the pullback of the basic bundle gerbe over $SU(2)$ and the cup product bundle gerbe over $T\times {SU(2)}/{T}$ are not equal over $\Sigma_2$.
\end{proposition}
\begin{proof}
By equation (\ref{aadff}), we need only show $\int_{\Sigma_2} \beta_2 \neq k2\pi i$ for any $k\in \Z$. Now, $$\beta_2 = \frac{1}{4\pi i}\left(p_2p_1^{-1} \text{tr}(P_2dP_1dP_1) + p_1p_2^{-1}\text{tr}(P_1dP_2dP_2)\right) .$$ Since $P_1+P_2=1$, we can set $P:= P_1$, so $P^{-1}:= P_2$. Similarly, $p_2 = p_1^{-1}$, so we can set $p := p_1$ and $p^{-1}:= p_2$. Then a simple calculation gives us $$\beta_2 = \frac{i}{4\pi}(p^2-p^{-2})\text{tr}(PdPdP). $$ It is a standard fact that $\text{tr}(PdPdP)$ is the curvature of the tautological line bundle over $S^2$, which has Chern class minus one, i.e. $\frac{i}{2\pi}\int_{S^2} \text{tr}(PdPdP) = -1.$ Therefore \begin{align*}
	\int_{\Sigma_2} \beta_2 &= \frac{ie^{\tfrac{i\pi}{2}} - ie^{-\tfrac{i\pi}{2} }}{4\pi} \int_{S^2} \text{tr}(PdPdP) \\
	&= 	 \frac{-e^{\tfrac{i \pi}{2} } + e^{-\tfrac{i\pi}{2}}}{2\pi}\\
	&= \frac{\sin(\tfrac{\pi}{2})}{\pi i}\\
	&= \frac{1}{\pi i} \\ &\neq k2\pi i \ \forall \ k \in \Z,
\end{align*}
hence $\text{exp}\left(\int_{\Sigma_2} \beta_2\right)\neq 1$ and the holonomies are not equal over this surface.
\end{proof}
\begin{corollary}
There exists a surface $\Sigma_n \subset T\times \su$ such that \begin{align}\label{raining} \textup{hol}((\nabla_{p^*b}, f_{p^*b}), \Sigma_n)  \neq \textup{hol}((\nabla_c, f_c), \Sigma_n).\end{align}
\end{corollary}
\begin{proof}
First, note that surface $\Sigma_2 = \{e^{{\pi i}/{4}}\}\times S^2$ from Proposition \ref{holonomycomputation} is an embedded submanifold of $T\times \su$ with respect to the inclusion $\iota: {SU(2)}/{T_1} \hookrightarrow {SU(n)}/{T_{n-1}}$ defined by $$XT_1 \mapsto \left[
\begin{array}{c|c}
X & 0 \\
\hline
0 & \textbf{I}_{n-2}
\end{array}
\right]T_{n-1}.$$ Here, $T_1, T_{n-1}$ denote the subgroups of diagonal matrices in $SU(2)$ and $SU(n)$ respectively, and $\textbf{I}_{n-2}$ is the $(n-2)\times (n-2)$ identity matrix. Let $\Sigma_n := \iota\left(\Sigma\right)$. By equation (\ref{aadff}), to prove $(\ref{raining})$, it suffices to show that $$\int_{\Sigma_n} \beta_n = \int_{\Sigma_2} \iota^*\beta_n \neq k2\pi i$$ for any $k\in \Z$. To do so, we prove that $\iota^*\beta_n = \beta_2$, hence $\int_{\Sigma_n} \beta_n \neq k2\pi i$ by the proof of Proposition \ref{holonomycomputation}. We compute $\iota^*\beta_n$ as follows. Recall that the maps $p_i:T\to S^1$ were defined as projection onto the $i$-th diagonal (Definition \ref{salad}). Clearly $$p_i\circ \iota = \begin{cases}
p_i & \text{if} \ i = 1, 2\\
1 & \text{if} \ 2< i \leq n.
\end{cases} $$
Further recall that $P_i$ was defined to be orthogonal projection onto $J_i := \C\times_{p_i}SU(n)$, where $p_i$ was the relation $(z, s)\sim_{p_i}(p_i(t^{-1})z, st)$ for all $(z, s)\in \C\times SU(n)$. Now, when the maps $p_i$ are the constant value $1$, this relation is equality, and $J_i\to \su$ is isomorphic to the trivial line bundle over $\su$. In this case, $P_i$ will be the constant projection onto the span of $e_i$, the $i$-th standard basis vector of $\C^n$. That is, $P_i = O_i$ for $O_i$ the matrix with a $1$ in the $(i, i)$ position and zeros elsewhere. Therefore $$P_i\circ \iota = \begin{cases}
P_i & \text{if} \ i=1, 2\\
O_i & \text{if} \ 2< i \leq n.
\end{cases} $$ Of course, $dO_i = 0$, so any term of the form $\text{tr}(P_k dP_i dP_i)$ for $i>2$ in our expression for $\beta_n$ in (\ref{werrr}) will equal zero. Furthermore, any term of the form $\text{tr}(P_i dP_k dP_k)$ for $i>2$ will also be zero, by Lemma \ref{projections} (2). So $\iota^*\beta_n = \beta_2$, and $(\ref{raining})$ holds. \end{proof} 
We can now refine Proposition \ref{prp} as follows, thereby concluding this chapter. 
\begin{theorem}\label{finalresult}
	The cup product bundle gerbe is $SU(n)$-stably isomorphic to the pullback of the basic bundle gerbe by the Weyl map, i.e. $$\cupbg \cong_{SU(n)\textup{-stab}} \pullbackbasicbg,$$ with trivialising line bundle $$ \Motimes_{i=1}^n \pi_{\smath{{SU(n)}/{T}}}^{-1}(J_i)^{h_i}\to (\R^{n-1}\times_T Y_T) \times \su$$ for $\pi_{\smath{{SU(n)}/{T}}}: (\R^{n-1} \times_T Y_T) \times \su \to  \su$ projection and the smooth maps $h_i: \left(\R^{n-1} \times_T Y_T\right) \times \su \to \Z$ defined by $$h_i(x, z, t, gT) = x_i - \frac{1}{2\pi i}\log_z p_i(t)$$ for $x = (x_1, ..., x_n)$. Moreover, these bundle gerbes are not $D$-stably isomorphic with respect to the connective data $(\nabla_c, f_c)$ and $(\nabla_{p^*b}, f_{p^*b})$ from \textup{Propositions \ref{corollary data on cup prod}} and \textup{\ref{curvature}}.
\end{theorem}


%% file: appendix1/appendix.tex
\appendix

\chapter{Computational lemmas}

\indent \indent Here, we present the lemmas used to prove various results in Chapter \ref{ch:four}. Most importantly, these results allowed us to express the three-curvature $\omega_{p^*b}$ in terms of $\beta_n$ (Proposition \ref{terer}). Recall that projections $P_i$ and homomorphisms $p_i$ from Chapter \ref{ch:four}.
\begin{lemma}\label{lemma111}
	The three-curvature of the pullback of the basic bundle gerbe can be rewritten as  \begin{align}\label{eqn4} 
	\omega_{p^*b} = \frac{i}{4\pi}\sum_{\substack{i=1 \\ i\neq k}}^n  \sum_{k=1}^n 
	p_i^{-1} dp_i  \textup{tr}(P_i dP_kdP_k) +   \frac{i}{4\pi}\sum_{k=1}^n p_k^{-1}dp_k \textup{tr}(P_kdP_kdP_k) + d\beta_n
	\end{align} for $\beta_n :=  -\frac{i}{4\pi}\sum_{\substack{i=1 \\ i\neq k}}^n  \sum_{k=1}^n p_ip_k^{-1}\textup{tr}(P_idP_kdP_k)$.
\end{lemma}
\begin{proof}
	Observe that 
	$$\omega_{p^*b} =  
	\frac{i}{4\pi}\sum_{\substack{i=1 \\ i\neq k}}^n  \sum_{k=1}^n \left(
	p_i^{-1} dp_i - p_k^{-1}dp_k\right) \text{tr}(P_i dP_kdP_k) - d\left(p_ip_k^{-1}\text{tr}(P_idP_kdP_k)\right).$$
	Now, $\sum_{i=1}^n P_i = 1$ implies that $\sum_{\substack{i=1 \\ i\neq k}}^n  \sum_{k=1}^n  P_i = 1 - P_k$, and linearity of trace implies $\textup{tr}(dP_idP_i)= 0$. Therefore the second term in the above expression for $\omega_{p^*b}$ can be rewritten as $\sum_{k=1}^n p_k^{-1}dp_k \text{tr}(P_kdP_kdP_k),$  and the result follows.
\end{proof}
%
\begin{lemma}\label{projections} 
	Let $i, j, k \in \{1, ..., n\}$. Then \begin{enumerate}[(1),font=\upshape]
		\item for distinct $i, j, k$, $\textup{tr}(P_idP_jdP_k) = 0$;
		\item if $i\neq j$, $\textup{tr}(P_idP_jdP_j) = -\textup{tr}(P_jdP_idP_i)$;
		\item $\sum_{\substack{i=1 \\ i\neq k}}^n  \sum_{k=1}^n  p_i^{-1}dp_i \textup{tr}(P_idP_kdP_k) = \sum_{i=1}^n p_i^{-1}dp_i \textup{tr}(P_idP_idP_i).$
	\end{enumerate}
\end{lemma}
\begin{proof}
	To prove (1), note that $P_iP_j = 0$ if $i\neq j$, and $dP_i = dP_iP_i + P_i dP_i$ (where we obtain the second equation by differentiating $P_i^2 = P_i$). So for distinct $i, j$ and $k$ we have \begin{align*}
	\text{tr}(P_i dP_jdP_k) &= \text{tr}(P_i (P_jdP_j + dP_jP_j) dP_k)\\
	&= \text{tr}(P_idP_jP_jdP_k)\\
	&= \text{tr}(P_idP_jP_j(P_kdP_k + dP_kP_k))\\
	&= \text{tr}(P_idP_jP_jdP_kP_k)\\
	&= \text{tr}(P_kP_idP_jP_jdP_k)=0,\end{align*} thereby proving $(1)$. 
	Next, by differentiating the identity $P_iP_j = 0$, we obtain $dP_iP_j = -P_idP_j$ for $i\neq j$. Therefore, using $(1)$ and that $\sum_{i=1}^n dP_i = 0$, we obtain \begin{align*}
	\text{tr}(P_idP_jdP_j) &= -\text{tr}(dP_iP_jdP_j) \\
	&= \text{tr}(P_jdP_jdP_i)\\
	&= \text{tr}\left(P_j \left(-\textstyle \sum_{k\neq j} dP_k\right)dP_i \right)\\
	&= -\sum_{k\neq j} \text{tr}(P_jdP_kdP_i)\\
	&= -\text{tr}(P_jdP_idP_i),
	\end{align*} thereby proving $(2)$. By $(2)$ and since $\sum_{i=1}^n P_i=1$, $\text{tr}(dP_idP_i) = 0$, we have \begin{align*}
	\sum_{\substack{i=1 \\ i\neq k}}^n  \sum_{k=1}^n  p_i^{-1}dp_i \text{tr}(P_idP_kdP_k) &= \sum_{i=1}^n  p_i^{-1}dp_i  \cdot \sum_{\substack{k=1 \\ k\neq i}}^n \text{tr}(P_idP_kdP_k)\\
	&= \sum_{i=1}^n  p_i^{-1}dp_i  \cdot \sum_{\substack{k=1 \\ k\neq i}}^n \text{tr}(-P_kdP_idP_i) \\
	&= \sum_{i=1}^n  p_i^{-1}dp_i  \text{tr}\left(\left(P_i-1\right) dP_idP_i \right)\\
	&= \sum_{i=1}^n  p_i^{-1}dp_i  \text{tr}(P_idP_idP_i).
	\end{align*} 
\end{proof}